\documentclass[twoside]{article}
\usepackage[letterpaper, left=3.5cm, right=3.5cm, top=4cm, bottom=2.5cm]{geometry}
\usepackage{graphicx}
\usepackage{amsmath,amssymb,amsthm}
\usepackage{authblk}
\usepackage[utf8]{inputenc}
\usepackage{microtype}
\usepackage{mathrsfs} 
\usepackage{enumitem}
\usepackage{calc}
\usepackage{esint}
\usepackage{fancyhdr}
\usepackage{xcolor}
\usepackage[labelfont = bf]{caption}
\usepackage{doi}
\usepackage{subfigure}
\usepackage[numbers]{natbib}
\usepackage{float}
\usepackage{stmaryrd}
\usepackage{mathtools}
\usepackage{booktabs}
\usepackage{colortbl}
\usepackage{tabulary}
\usepackage{diagbox}
\usepackage{caption}

\DeclareMathAlphabet{\mathbbold}{U}{bbold}{m}{n}
\newcommand{\mathbbone}{\mathbbold{1}}
\newcommand{\mathbbzero}{\mathbbold{0}}

\usepackage{physics} 
\usepackage{tikz}
\usepackage{mathdots}
\usepackage{yhmath}
\usepackage{cancel}
\usepackage{color}
\usepackage{siunitx}
\usepackage{array}
\usepackage{multirow}
\usepackage{amssymb}
\usepackage{gensymb}
\usepackage{tabularx}
\usepackage{extarrows}
\usepackage{booktabs}
\usetikzlibrary{fadings}
\usetikzlibrary{patterns}
\usetikzlibrary{shadows.blur}
\usetikzlibrary{shapes}

\newtheorem{theorem}{Theorem}[section]
\numberwithin{equation}{section}
\newtheorem{proposition}[theorem]{Proposition}
\newtheorem{definition}[theorem]{Definition}
\newtheorem{corollary}[theorem]{Corollary}
\newtheorem{remark}[theorem]{Remark}
\newtheorem{lemma}[theorem]{Lemma}
\newtheorem{example}[theorem]{Example}
\newtheorem{examples}[theorem]{Examples}
\newtheorem{algorithm}[theorem]{Algorithm}
\newtheorem{assumption}[theorem]{Assumption}

\definecolor{blue1}{rgb}{0.031,0.188,0.419} % darkest
\definecolor{blue2}{rgb}{0.063,0.357,0.643}
\definecolor{blue3}{rgb}{0.231,0.545,0.760}
\definecolor{blue4}{rgb}{0.459,0.705,0.848}
\definecolor{blue5}{rgb}{0.716,0.833,0.916}
\definecolor{blue6}{rgb}{0.895,0.937,0.976} % lightest

\usepackage{titlesec}
\titleformat{\section}{\normalfont\scshape\centering}{\thesection.}{0.5em}{}
\titleformat*{\subsection}{\itshape}
\titleformat*{\subsubsection}{\itshape}

\providecommand{\keywords}[1]
{
	{\small\emph{Keywords:} #1}
}

\providecommand{\MSC}[1]
{
	{\small\emph{AMS MSC (2020):~~} #1}
}

\definecolor{denim}{rgb}{0.08, 0.38, 0.74}
\definecolor{byzantium}{rgb}{0.44, 0.16, 0.39} 
\definecolor{shamrockgreen}{rgb}{0.0, 0.62, 0.38} 

\usepackage{hyperref}
\hypersetup{
	colorlinks=true,
	linkcolor=denim,
	citecolor = shamrockgreen,
	filecolor=magenta, 
	urlcolor=byzantium,
}

\begin{document}
	\setlength{\abovedisplayskip}{5.5pt}
	\setlength{\belowdisplayskip}{5.5pt}
	\setlength{\abovedisplayshortskip}{5.5pt}
	\setlength{\belowdisplayshortskip}{5.5pt}

	\title{\vspace{-15mm}A $\operatorname{prox}$-Based Semi-Smooth Newton Method\\
for Convex Variational Problems} 
	\author[1]{Sören Bartels\thanks{Email: \url{bartels@mathematik.uni-freiburg.de}}}
	\author[2,3]{Alex Kaltenbach\thanks{Email: \url{kaltenbach@math.tu-berlin.de}}}
	\date{\today\vspace{-1.5mm}}
	\affil[1]{\small{Department of Applied Mathematics, University of Freiburg, Hermann--Herder-Str. 10,  D-79104  Freiburg}}\affil[2]{\small{Department of Mathematical Sciences and the Center for Mathematics and Artificial Intelligence (CMAI), George Mason University, Fairfax, VA 22030, USA.}}
	\affil[3]{\small{Institute of Mathematics, Technical University of Berlin, Stra\ss e des 17.\ Juni 136, D-10623 Berlin\vspace{-1.5mm}}}
	\maketitle

	\pagestyle{fancy}
	\fancyhf{}
	\fancyheadoffset{0cm}
	\addtolength{\headheight}{-0.25cm}
	\renewcommand{\headrulewidth}{0pt} 
	\renewcommand{\footrulewidth}{0pt}
	\fancyhead[CO]{A \textsc{$\operatorname{prox}$-based Semi-Smooth Newton Method}}
	\fancyhead[CE]{\textsc{S. Bartels and A. Kaltenbach}}
	\fancyhead[R]{\thepage}
	\fancyfoot[R]{}
	
	\begin{abstract}
		In this paper, we devise a $\operatorname{prox}$-based semi-smooth Newton method that is applicable to a finite element discretization of a broad class of nonsmooth convex variational problems, including the TV-minimization problem, the $p$-Dirichlet problem, the obstacle problem, and the elasto-plastic torsion problem.
        To this end, on the basis of the proximity operator, the discrete  primal-dual optimality conditions are reformulated
        as  nonlinear operator equations with  Newton-differentiable structure.  
        Under suitable assumptions~on~the~energy~densities, 
        we establish the global well-posedness and local super-linear convergence of the resulting semi-smooth Newton method.
        The proposed approach coincides with established semi-smooth Newton methods for obstacle-type problems,  satisfies a primal-dual invariance, and,\linebreak under \hspace{-0.1mm}suitable 
\hspace{-0.1mm}additional \hspace{-0.1mm}assumptions, \hspace{-0.1mm}is \hspace{-0.1mm}globally \hspace{-0.1mm}well-posed \hspace{-0.1mm}in \hspace{-0.1mm}the \hspace{-0.1mm}infinite-dimensional~\hspace{-0.1mm}\mbox{setting}.
	\end{abstract}
	
	\keywords{Proximity operator; semi-smooth Newton method; convex variational problems;
Fenchel duality; finite element method; total variation minimization;
$p$-Dirichlet problem; elasto-plastic~torsion}
	
	\MSC{49M15; 65K10; 65N30; 49M29; 49J52; 90C25}
	
	\section{Introduction}\thispagestyle{empty}\enlargethispage{15mm}\vspace{-1mm}

\hspace{5mm}In this paper, we devise a $\operatorname{prox}$-based semi-smooth Newton method that is applicable to a finite element discretization of a broad class of nonsmooth convex variational~problems and is based on a Fenchel duality framework.\vspace{-1mm}

\subsection{Fenchel duality for general convex variational problems}\vspace{-1.5mm}

\hspace{5mm}Let a general convex minimization problem be given via the minimization of the primal energy functional $I\colon W^{1,p}_D(\Omega)\to \mathbb{R}\cup\{+\infty\}$, where $p\in [1,+\infty)$, for every $v\in W^{1,p}_D(\Omega)$~defined~by
\begin{align}\label{intro:primal}
  I(v) \coloneqq \int_{\Omega}{\phi(\cdot,\nabla v)\,\mathrm{d}x} + \int_{\Omega}{\psi(\cdot,v)\,\mathrm{d}x}\,,
\end{align}
where the energy densities $\phi\colon \Omega\times\mathbb{R}^d\to \mathbb{R}\cup \{+\infty\}$ and $\psi\colon \Omega\times \mathbb{R}\to \mathbb{R}\cup \{+\infty\}$~are~convex~normal\linebreak integrands   having the lower compactness property (\textit{cf}.~Assumption \ref{ass:energy_densities}).

Under appropriate additional assumptions on the energy density $\psi$ (\textit{cf}.\ Assumption \ref{ass:energy_densities_conjugates}), 
a (Fenchel) dual problem (in the sense of \cite[Rem.\ 4.2, p.\ 60/61]{EkelandTemam1999}) is given via the maximization of the dual energy functional $D\colon \smash{W^{p'}_N(\operatorname{div};\Omega)}\to \mathbb{R}\cup\{-\infty\}$, for every $y\in  \smash{W^{p'}_N(\operatorname{div};\Omega)}$ defined by
\begin{align}\label{intro:dual}
  D(y) \coloneqq -\int_{\Omega}{\phi^{*}(\cdot,y)\,\mathrm{d}x} - \int_{\Omega}{\psi^{*}(\cdot,\operatorname{div} y)\,\mathrm{d}x}\,,
\end{align}
where $\smash{\phi^*\colon \Omega\times\mathbb{R}^d\to \mathbb{R}\cup \{+\infty\}}$ and $\smash{\psi^*\colon \Omega\times \mathbb{R}\to \mathbb{R}\cup \{+\infty\}}$ denote the Fenchel conjugates (with respect to the second argument) of the energy densities $\phi$ and $\psi$, respectively. 

A minimizer $u\in \smash{W^{1,p}_D(\Omega)}$ of the primal energy functional \eqref{intro:primal}, called \emph{primal solution},~and a maximizer \hspace{-0.15mm}$z\hspace{-0.15em}\in \hspace{-0.15em}\smash{W^{p'}_N(\operatorname{div};\Omega)}$ \hspace{-0.15mm}of \hspace{-0.15mm}the \hspace{-0.15mm}dual \hspace{-0.15mm}energy \hspace{-0.15mm}functional \hspace{-0.15mm}\eqref{intro:dual}, \hspace{-0.15mm}called \hspace{-0.15mm}\emph{dual \hspace{-0.15mm}solution},~\hspace{-0.15mm}are~\hspace{-0.15mm}\mbox{characterized} and related by the \emph{primal optimality inclusions}\vspace{-0.5mm}
\begin{subequations}\label{intro:primal-optimality-inclusions}
\begin{alignat}{2}
  z &\in \partial_{t}\phi(\cdot,\nabla u)&&\quad\text{ a.e.\ in }\Omega\,,\\
  \operatorname{div} z &\in \partial_{s}\psi(\cdot,u)&&\quad\text{ a.e.\ in }\Omega\,,
\end{alignat}\\[-4.5mm]
\end{subequations}
where $\smash{\partial_{t}\phi\colon \Omega\times\mathbb{R}^d\to 2^{\mathbb{R}^d}}$ and $\smash{\partial_{s}\psi\colon \Omega\times\mathbb{R}\to 2^{\mathbb{R}}}$ denote the subdifferentials (with respect to the second argument) of the energy densities $\phi$ and $\psi$, respectively. Note that %in can be shown that 
the primal optimality inclusions \eqref{intro:primal-optimality-inclusions}
are equivalent to the \emph{strong duality relation} (\textit{cf}.\ \cite[Prop.\ 3.1(ii)]{BartelsKaltenbach2024Overview})\vspace{-0.5mm}
\begin{align}\label{intro:strong-duality}
    I(u) = D(z)\,.\\[-6.5mm]\notag
\end{align} 

\subsection{Equivalence of the primal optimality inclusions to a root-finding problem}\enlargethispage{7.5mm}\vspace{-0.5mm}

\hspace{5mm}Inspired by the derivation of the first-order primal-dual algorithm for general convex variational problems by Chambolle~and~Pock (\textit{cf}.~\cite{ChambollePock2011}) 
as well as the derivation of the semi-smooth Newton method for Tikhonov functionals with $\ell^1$-sparsity constraints~by~Herzog~and~Lorenz~(\textit{cf}.~\cite{GriesseLorenz2008}),
the primal optimality inclusions \eqref{intro:primal-optimality-inclusions} may be rewritten as a root-finding problem with Newton-type differentiability properties:

Multiplying the primal optimality inclusions \eqref{intro:primal-optimality-inclusions} by \emph{proximity parameters} $\gamma_{1},\gamma_{2}>0$ and adding $\nabla u$ and $u$ to both sides, respectively, we obtain the equivalent inclusions\vspace{-0.5mm}
\begin{subequations}\label{eq:equivalent-inclusions}
\begin{alignat}{2}
  \nabla u + \gamma_{1}z &\in (\operatorname{id}_{\mathbb{R}^d}+ \gamma_{1}\partial_{t}\phi)(\cdot,\nabla u)&&\quad\text{ a.e.\ in }\Omega\,,\\
 u + \gamma_{2}\operatorname{div} z &\in (\operatorname{id}_{\mathbb{R}} + \gamma_{2}\partial_{s}\psi)(\cdot,u)&&\quad\text{ a.e.\ in }\Omega\,.
\end{alignat}\\[-4.5mm]
\end{subequations} 
Using the proximity operators (with respect to the second argument) $\operatorname{prox}_{\gamma_{1}\phi} \coloneqq(\operatorname{id}_{\mathbb{R}^d}
+\gamma_1\partial_t\phi)^{-1}: $ $ \hspace{-0.15em}\Omega\hspace{-0.1em}\times\hspace{-0.1em} \mathbb{R}^d\hspace{-0.15em}\to\hspace{-0.15em} \mathbb{R}^d$ \hspace{-0.1mm}and \hspace{-0.1mm}$\operatorname{prox}_{\gamma_{2}\psi}\coloneqq(\operatorname{id}_{\mathbb{R}}
+\gamma_2\partial_s\psi)^{-1} \colon\hspace{-0.15em}\Omega\hspace{-0.1em}\times\hspace{-0.1em} \mathbb{R}\hspace{-0.15em}\to\hspace{-0.15em} \mathbb{R}$ of the energy~\hspace{-0.1mm}\mbox{densities}~\hspace{-0.1mm}$\phi$~\hspace{-0.1mm}and~\hspace{-0.1mm}$\psi$,~\hspace{-0.1mm}respectively, 
%for a.e.\ $x\in \Omega$ as well as every $t\in \mathbb{R}^d$ and $s\in \mathbb{R}$, respectively, defined by
%\begin{align*}
%  \operatorname{prox}_{\gamma_{1}\phi}(x,t) &\coloneqq \underset{\widehat{t}\in \mathbb{R}^d}{\operatorname{arg\,min}}{\big\{\tfrac{1}{2\gamma_{1}}|\widehat{t}-t|^{2} + \phi(x,\widehat{t})\big\}}\,,\\[-1mm]
%   \operatorname{prox}_{\gamma_{2}\psi}(x,s) &\coloneqq \underset{\widehat{s}\in \mathbb{R}}{\operatorname{arg\,min}}{\big\{\tfrac{1}{2\gamma_{2}}|\widehat{s}-s|^{2} + \psi(x,\widehat{s})\big\}}\,,\\[-6.5mm]
%\end{align*}
%which coincide with the resolvents (with respect to the second argument) of  $\smash{\partial_t \phi \colon \Omega\times\mathbb{R}^d\to 2^{\mathbb{R}^d}}$~and $\smash{\partial_s \psi \colon \Omega\times\mathbb{R}\to 2^{\mathbb{R}}}$, \textit{i.e.}, for a.e.\ $x\in \Omega$ as well as every $t\in \mathbb{R}^d$ and $s\in \mathbb{R}$, respectively,~we~have~that
%\begin{align*}
%    \operatorname{prox}_{\gamma_{1}\phi}(x,t)&=(\operatorname{id}_{\mathbb{R}^d}+\gamma_{1} \partial_t \phi(x,\cdot))^{-1}(t)\,,\\
%    \operatorname{prox}_{\gamma_{2}\psi}(x,s)&=(\operatorname{id}_{\mathbb{R}}+\gamma_{2} \partial_s \psi(x,\cdot))^{-1}(s)\,,
%\end{align*}
the inclusions \eqref{eq:equivalent-inclusions} may be rewritten  equivalently~as~the~\emph{proximal optimality~conditions}\vspace{-4.5mm}
\begin{subequations}\label{eq:equivalent-inclusions.2}
\begin{alignat}{2}\label{eq:equivalent-inclusions.2.1}
    \nabla u &= \operatorname{prox}_{\gamma_{1}\phi}(\cdot,\nabla u + \gamma_{1} z)&&\quad \text{ a.e.\ in }\Omega\,,\\
    u &= \operatorname{prox}_{\gamma_{2}\psi}(\cdot,u + \gamma_{2}\operatorname{div} z)&&\quad \text{ a.e.\ in }\Omega\,.\label{eq:equivalent-inclusions.2.2}
\end{alignat}\\[-8.5mm]
\end{subequations}
%The equations \eqref{eq:equivalent-inclusions.2} form the starting point of the primal-dual algorithm~by Chambolle~and~Pock (\textit{cf}.~\cite{ChambollePock2011}).

The proximal optimality conditions \eqref{eq:equivalent-inclusions.2} may be viewed as a fixed-point system involving $1$-Lipschitz proximity operators, as in Chambolle--Pock  \cite{ChambollePock2011}. Equivalently,~the~primal~\mbox{optimality} inclusions \eqref{intro:primal-optimality-inclusions} are equivalent to the root-finding problem for the nonlinear mapping
$\mathtt{F}\colon \smash{W^{p'}_N(\operatorname{div};\Omega)}\times W^{1,p}_D(\Omega)\to (L^{\min\{p,p'\}}(\Omega))^d\times L^{\min\{p,p'\}}(\Omega)$, for every~${(y,v)\in \smash{W^{p'}_N(\operatorname{div};\Omega)}\times W^{1,p}_D(\Omega)}$~defined~by\vspace{-0.5mm}
\begin{align}\label{def:F-general}
    \mathtt{F}(y,v) \coloneqq
  \begin{bmatrix}
    \nabla v - \operatorname{prox}_{\gamma_{1}\phi}(\cdot,\nabla v + \gamma_{1} y)\\
    v - \operatorname{prox}_{\gamma_{2}\psi}(\cdot,v + \gamma_{2}\operatorname{div} y)
  \end{bmatrix}\quad\text{ a.e.\ in }\Omega\,,\\[-6mm]\notag
\end{align}
which, by the equations \eqref{eq:equivalent-inclusions.2}, admits the root $(z,u)\in \smash{W^{p'}_N(\operatorname{div};\Omega)\times W^{1,p}_D(\Omega)}$, \textit{i.e.},~we~have~that\vspace{-0.5mm}
\begin{align*}
     \mathtt{F}(z,u) = (\mathtt{0}_d,0) \quad \text{ a.e.\ in }\Omega\,.\\[-6.5mm]
\end{align*}

\subsection{Newton-type differentiability properties of the root-finding problem}\vspace{-0.5mm}

\hspace{5mm}For a.e.\ $x\in \Omega$, the proximity operators $\operatorname{prox}_{\gamma_{1}\phi}(x,\cdot)\colon \mathbb{R}^d\to \mathbb{R}^d$ and $\operatorname{prox}_{\gamma_{2}\psi}(x,\cdot)\colon \mathbb{R}\to \mathbb{R}$ 
are $1$-Lipschitz continuous and, consequently, according to \cite[Thm.~2.6]{ChenNashedQi2000}, point-wise Newton~differen\-tiable \hspace{-0.1mm}(\textit{cf}.\  \hspace{-0.1mm}Definition \ref{def:newton_differentiable}) \hspace{-0.1mm}with \hspace{-0.1mm}point-wise Newton \hspace{-0.1mm}derivatives \hspace{-0.1mm}${\mathtt{J}_{\smash{\operatorname{prox}_{\gamma_1\phi}}}(x,\hspace{-0.1em}\cdot)\colon\hspace{-0.175em} \smash{B_{\delta_{x,t}}^d(t)}\hspace{-0.175em}\to \hspace{-0.175em}\mathbb{R}^{d\times d}}$,~${\delta_{x,t}\hspace{-0.175em}>\hspace{-0.175em}0}$, for all $t\in \mathbb{R}^d$ and $\mathtt{J}_{\smash{\operatorname{prox}_{\gamma_2\psi}}}(x,\cdot)\colon \smash{B_{\delta_{x,s}}^1(s)}\to \mathbb{R}$,~$\delta_{x,s}>0$, for all $s\in \mathbb{R}$, respectively.

Similarly, the  mapping
$\mathtt{F}\colon \hspace{-0.1em}W^{p'}_N(\operatorname{div};\Omega)\hspace{-0.1em}\times \hspace{-0.1em}W^{1,p}_D(\Omega)\hspace{-0.1em}\to\hspace{-0.1em} (L^{\min\{p,p'\}}(\Omega))^d\hspace{-0.1em}\times\hspace{-0.1em} L^{\min\{p,p'\}}(\Omega)$, defined~by \eqref{def:F-general}, is Lipschitz continuous and, consequently, according to \cite[Thm.~2.6]{ChenNashedQi2000}, point-wise Newton differentiable with a point-wise Newton derivative at a root that may depend~on~this~root. 
Such a dependence of the point-wise Newton derivative on this root would render~it~\mbox{impractical}~for~use in a semi-smooth Newton method, as the respective root is
not known \emph{a priori}. Therefore,~it~is~in-\linebreak dispensable \hspace{-0.1mm}to \hspace{-0.1mm}identify \hspace{-0.1mm}a \hspace{-0.1mm}Newton
\hspace{-0.1mm}derivative \hspace{-0.1mm}that \hspace{-0.1mm}can \hspace{-0.1mm}be \hspace{-0.1mm}evaluated \hspace{-0.1mm}without~\hspace{-0.1mm}prior~\hspace{-0.1mm}\mbox{knowledge}~\hspace{-0.1mm}of~\hspace{-0.1mm}a~\hspace{-0.1mm}root.
However, constructing an explicit instance of such a derivative is
severely obstructed~by~the~typical \emph{norm gap phenomenon}
(\textit{cf}.\ \cite[Prop.~4.1]{HintermuellerItoKunisch2003}),
which is likewise expected in the present setting.  

Apparently, a canonical candidate for a Newton derivative is given via the Newton-type derivative  $\mathtt{J}_{\mathtt{F}}\colon \smash{W^{p'}_N(\operatorname{div};\Omega)}\times W^{1,p}_D(\Omega)\to \mathcal{L}(\smash{W^{p'}_N(\operatorname{div};\Omega)}\times W^{1,p}_D(\Omega);\smash{(L^{\min\{p,p'\}}(\Omega))^d}\times \smash{L^{\min\{p,p'\}}(\Omega)})$, for every $(y,v),(\widehat{y},\widehat{v})\in \smash{W^{p'}_N(\operatorname{div};\Omega)\times W^{1,p}_D(\Omega)}$ defined by\vspace{-0.75mm}
\begin{align}\label{def:J_F-general}
  \mathtt{J}_{\mathtt{F}}(y,v)(\widehat{y},\widehat{v}) \coloneqq
  \begin{bmatrix}
    (\mathbbone - \mathtt{J}_{\operatorname{prox}_{\gamma_{1}\phi}}(a))\nabla \widehat{v}-\gamma_{1} \mathtt{J}_{\operatorname{prox}_{\gamma_{1}\phi}}(a)\widehat{y} \\
     (\operatorname{1} - \mathtt{J}_{\operatorname{prox}_{\gamma_{2}\psi}}(b))\widehat{v}-\gamma_{2} \mathtt{J}_{\operatorname{prox}_{\gamma_{2}\psi}}(b)\operatorname{div}\widehat{y} 
  \end{bmatrix}\quad\text{ a.e.\ in }\Omega\,,\\[-6mm]\notag
\end{align}
where $a\coloneqq \nabla v + \gamma_{1}y\in \smash{(L^{\smash{\min\{p,p'\}}}(\Omega))^d}$ and $ b\coloneqq v + \gamma_{2}\operatorname{div} y\in \smash{L^{\smash{\min\{p,p'\}}}(\Omega)}$.

Although the Newton-type derivative \eqref{def:J_F-general}, in general, does not represent a genuine Newton derivative on the continuous level,  it leads to a rigorous justification of semi-smooth Newton methods for corresponding discretizations.  For well-posedness~of~the~semi-smooth~Newton~steps, the \hspace{-0.15mm}dual \hspace{-0.15mm}components \hspace{-0.15mm}of 
\hspace{-0.15mm}iterates, \hspace{-0.15mm}update \hspace{-0.15mm}directions, \hspace{-0.15mm}and  \hspace{-0.15mm}roots 
\hspace{-0.15mm}are \hspace{-0.15mm}sought \hspace{-0.15mm}in \hspace{-0.15mm}$W^p_N(\operatorname{div};\Omega)$~\hspace{-0.15mm}rather~\hspace{-0.15mm}than \hspace{-0.15mm}in the natural
Fenchel-dual space $\smash{W^{p'}_N(\operatorname{div};\Omega)}$ which, under suitable
assumptions, allows~the~linear\-ized equations to be reduced to elliptic
problems for the primal components of update directions~in
$W^{1,p}_D(\Omega)$ with right-hand sides in $W^{-1,p}_D(\Omega)$. Accordingly, we introduce\vspace{-0.5mm} %the algorithmic domain and residual space
\begin{align*} 
\smash{\mathbb{A}_p\coloneqq 
W^p_N(\operatorname{div};\Omega)\times W^{1,p}_D(\Omega)\,,
\qquad
\mathbb{B}_p
\coloneqq
(L^p(\Omega))^d\times L^p(\Omega)\,,}\\[-6mm]
\end{align*}
and formulate the resulting $\operatorname{prox}$-based semi-smooth Newton iteration as follows:\vspace{-0.5mm}

\begin{algorithm}[$\operatorname{prox}$-based \hspace{-0.1mm}semi-smooth \hspace{-0.1mm}Newton \hspace{-0.1mm}method \hspace{-0.1mm}for \hspace{-0.1mm}general \hspace{-0.1mm}convex \hspace{-0.1mm}variational~\hspace{-0.1mm}\mbox{problems}]\label{alg:SSNM-cont-general}
Let $\gamma_1,\gamma_2>0$ be proximity parameters, let $\varepsilon_{\mathtt{abs}},\varepsilon_{\mathtt{rel}} >0$ be stopping parameters, let $(z^{0},u^{0})\in \mathbb{A}_p$ be an initial iterate, and let $k_{\mathtt{max}}\in  \mathbb{N}\cup\{+\infty\}$ be a maximum number of iterations. Then, for $k=0,\ldots,k_{\mathtt{max}}$, perform the following~\mbox{iteration}~loop:
\begin{itemize}[noitemsep,topsep=2pt,leftmargin=!,labelwidth=\widthof{(2)}]
\item[(1)] \hypertarget{alg:SSNM-cont.1}{} Compute the \emph{primal-dual update direction} $\smash{(\delta z^k,\delta u^k) \in \mathbb{A}_p}$ such that\vspace{-0.5mm}
\begin{align}\label{alg:SSNM-cont-general.0}
  \smash{\mathtt{J}_{\mathtt{F}}(z^{k},u^{k})(\delta z^k,\delta u^k) = -\mathtt{F}(z^{k},u^{k})\quad\text{ in }\mathbb{B}_p\,,}\\[-6mm]\notag
\end{align}
and the updated iterate $\smash{(z^{k+1},u^{k+1}) \coloneqq (z^{k},u^{k}) +(\delta z^k,\delta u^k)\in   \mathbb{A}_p}$; %,  where $\alpha_k\hspace{-0.1em}>\hspace{-0.1em}0$~is~a~(\mbox{possibly} variable) step size;
\item[(2)] If $\smash{\|\mathtt{F}(z^{k+1},u^{k+1})\|_{\mathbb{B}_p}<\max\{\varepsilon_{\mathtt{abs}},\varepsilon_{\mathtt{rel}}\|\mathtt{F}(z^0,u^0)\|_{\mathbb{B}_p}\}}$,
then \textup{STOP}; otherwise, $k\mapsto  k+1$ and continue with (\hyperlink{alg:SSNM-cont.1}{1}).\vspace{-0.5mm}
\end{itemize}
\end{algorithm}
The well-posedness of Algorithm \ref{alg:SSNM-cont-general} can be ensured (\textit{cf}.\ \cite[Thm.\ 2.7]{BartelsKaltenbach2026}) for $p$~in~a~small~\mbox{interval} around~$2$~by~Gröger
regularity theory, or, under additional regularity
assumptions, for all $p\in (1,+\infty)$ by Schauder regularity theory. However, \hspace{-0.1mm}due \hspace{-0.1mm}to \hspace{-0.1mm}the \hspace{-0.1mm}lack \hspace{-0.1mm}of \hspace{-0.1mm}genuine \hspace{-0.1mm}Newton~\hspace{-0.1mm}\mbox{differentiability}\linebreak of  $\mathtt{F}\colon W^p_N(\operatorname{div};\Omega)\times W^{1,p}_D(\Omega)\to (L^p(\Omega))^d\times L^p(\Omega)$, defined by \eqref{def:F-general}, in the natural function spaces by the norm-gap phenomenon, it is unclear whether one can expect local~super-linear~convergence. 

In the present paper, we focus on a finite-dimensional counterpart of the above construction based on a finite element discretization of the primal \eqref{intro:primal} and the dual \eqref{intro:dual} problem, in which the norm-gap phenomenon does not occur and, therefore, a thorough~convergence~analysis~is~\mbox{possible}.\enlargethispage{7.5mm}

Proximity-operator and Moreau--Yosida reformulations are well-established in
nonsmooth optimization, in particular, through splitting schemes and proximal
Newton-type methods for composite minimization in finite-dimensional and
Hilbert-space settings; see,~\textit{e.g.},~\mbox{\cite{ParikhBoyd2014,LeeSunSaunders2014,PoetzlSchielaJaap2022,PoetzlSchielaJaap2024}}. 
Closely related proximal optimality conditions occur in inverse and
PDE-constrained~\mbox{optimization}~with nonsmooth control or sparsity terms, where
reduced or dual formulations often exploit smoothing by a solution operator;
see,~\textit{e.g.},~\cite{GriesseLorenz2008,Wachsmuth2026}.
By contrast, the present work treats nonsmooth convex variational energies
directly, without relying on such solution-operator smoothing, and combines the
proximal optimality conditions \eqref{eq:equivalent-inclusions.2} with a discrete Fenchel duality~framework.

\textit{The outline of the article is as follows.} In Section \ref{sec:preliminaries}, we recall important~\mbox{definitions}~and~results in connection with Newton differentiability, semi-smooth Newton methods, Moreau envelopes, and proximity operators as well as notation for standard continuous and discrete function spaces. In Section \ref{sec:continuous_duality}, we present a (Fenchel) duality framework for a broad class of convex integral functionals, which we transfer to a discrete level in Section \ref{sec:discrete_duality}. In Section \ref{sec:semi-smooth_newton}, starting from the discrete (Fenchel) duality framework in Section \ref{sec:discrete_duality}, we derive and study a discrete variant of the $\operatorname{prox}$-based  semi-smooth Newton method (\textit{cf}.\ Algorithm \ref{alg:SSNM-cont-general}). In Section \ref{sec:applications}, we apply the discrete (Fenchel) duality framework and the discrete $\operatorname{prox}$-based semi-smooth Newton method derived in Section \ref{sec:semi-smooth_newton} to the TV-minimization problem, the $p$-Dirichlet problem, and the elasto-plastic torsion problem. In Section \ref{sec:experiments}, we present numerical tests for these three benchmark problems.\pagebreak

    \section{Preliminaries}\vspace{-0.5mm}\label{sec:preliminaries}
    
    %\hspace{5mm}In this section, we collect the most important results and definitions that will be used in the forthcoming analysis.

    \subsection{Newton differentiability and the semi-smooth Newton method}\vspace{-0.5mm}

    \hspace{5mm}In this subsection, we recall the notion of Newton (or slant) differentiability along with related results, \hspace{-0.1mm}outline \hspace{-0.1mm}the \hspace{-0.1mm}semi-smooth \hspace{-0.1mm}Newton \hspace{-0.1mm}method, \hspace{-0.1mm}and \hspace{-0.1mm}state \hspace{-0.1mm}the \hspace{-0.1mm}main \hspace{-0.1mm}result~\hspace{-0.1mm}\mbox{concerning}~\hspace{-0.1mm}its~\hspace{-0.1mm}\mbox{local}~\hspace{-0.1mm}well-posedness and local (super-linear) convergence. For a detailed presentation, we refer~to~\cite{ChenNashedQi2000,Ulbrich2003,HintermuellerItoKunisch2003}.
    In doing so, throughout this subsection, let $(X,\|\cdot\|_X)$ and $(Y,\|\cdot\|_Y)$ be normed vector spaces.~~~

    At the basis of the semi-smooth Newton method is a generalized concept of differentiability.\enlargethispage{12mm}\vspace{-0.5mm}

    \begin{definition}[Newton (or slant)  differentiability]\label{def:newton_differentiable}
        Let $U\subseteq X$ be an open subset.~Then,~a~mapping $F\colon U\to Y$ is called \emph{Newton} (or \emph{slantly}) \emph{differentiable at} $x\in U$ if there exists~a~(\mbox{non-trivial})\footnote{Here, \emph{`non-trivial'} means that $\partial_{\mathrm{N}} F(\widehat{x})\neq \emptyset$ for all $\widehat{x}\in B_{\varepsilon}^X(x)$.} mapping $\partial_{\mathrm{N}} F\colon B_{\varepsilon}^X(x) \subseteq U\to  2^{\smash{\mathcal{L}(X;Y)}}$, $\varepsilon>0$, called \emph{Newton~\mbox{derivative}~of~$F$~at}~${x\in U}$, where $ 2^{\smash{\mathcal{L}(X;Y)}}$ denotes the power set of the space of linear and bounded operators $\mathcal{L}(X;Y)$,~such~that\vspace{-0.5mm}
        \begin{align}\label{def:newton_differentiable.1}
            \lim_{\tau\to 0_X}{\Big\{\smash{\sup_{\mathtt{L}\in \partial_{\mathrm{N}} F(x+\tau)}}{\big\{\tfrac{1}{\|\tau\|_X}\|F(x+\tau)-F(x)-\mathtt{L}\tau\|_Y\big\}}\Big\}}=0\,.\\[-6mm]\notag
        \end{align}
        
        A selection $\mathtt{J}_F\colon B_{\varepsilon}^X(x)\to \mathcal{L}(X;Y)$ of the \emph{Newton derivative} $\partial_{\mathrm{N}} F\colon B_{\varepsilon}^X(x)\to 2^{\smash{\mathcal{L}(X;Y)}}$, \textit{i.e.}, %we have that 
        $\mathtt{J}_F(\widehat{x})\in \partial_{\mathrm{N}} F(\widehat{x})$ for all $\widehat{x}\in B_\varepsilon^X(x)$, 
       %\begin{align*}
       %    \mathtt{J}_F(y)\in \partial_{\mathrm{N}} F(y)\quad\text{ for all }y\in B_\varepsilon^Y(x)\,,
       %\end{align*}
        is called \emph{Newton derivative selection} (or \emph{slanting function}).

        The mapping $F\colon U\to Y$ is called \emph{Newton} (or \emph{slantly}) \emph{differentiable in} $U$ if there exists a (non-trivial) mapping $\partial_{\mathrm{N}} F\colon U\to 2^{\smash{\mathcal{L}(X;Y)}}$ such that \eqref{def:newton_differentiable.1} holds at every $x\in U$.

        %The function $F\colon U\to Y$ is called \emph{semi-smooth} at $x\in U$ (in~$U$, respectively), if it is Newton differentiable  at $x\in U$ (in $U$, respectively) and locally Lipschitz continuous.
    \end{definition}

    \begin{remark}[on Definition \ref{def:newton_differentiable}]\label{rem:newton_differentiable} Let $U\subseteq X$ be an open subset and $F\colon U\to Y$~(\textit{cf}.~\cite[Sec.\ 2]{ChenNashedQi2000}):
        \begin{itemize}[noitemsep,topsep=2pt,leftmargin=!,labelwidth=\widthof{(iii)}]
            \item[(i)] \hypertarget{rem:newton_differentiable.i}{} In general, a Newton derivative $\partial_{\mathrm{N}} F\colon B_{\varepsilon}^X(x)\to 2^{\smash{\mathcal{L}(X;Y)}}$ is not unique;
            
            \item[(ii)] \hypertarget{rem:newton_differentiable.ii}{} If $F\colon U\to Y$ is continuously Fr\'echet differentiable in $U$, then it is Newton differentiable~in~$U$ and the Fr\'echet derivative $\mathrm{D}F\colon U\to \mathcal{L}(X;Y)$ is a possible Newton derivative selection;

            \item[(iii)] \hypertarget{rem:newton_differentiable.iii}{} If $F\colon U\to Y$ is Lipschitz continuous at $x\in U$, \textit{i.e.}, there exists some $L>0$ such that $\|F(x+\tau)-F(x)\|_Y\leq L\|\tau\|_X$ for sufficiently small $\tau\in X$,
            then it is Newton~differentiable~at~$x\in U$.
            %the Newton derivative %, for every $x\in U$, %$\partial_{\mathrm{N}} F\colon U\to 2^{\smash{\mathcal{L}(X;Y)}}$ 
           % given~via~${\partial_{\mathrm{N}} F(x)=\{\mathrm{D}F(x)\}}$~for~all~$x\in U$.

            %\item[(ii)] If $X$ is a Hilbert space and $Y=\mathbb{R}$, then the function $F\colon X\to \mathbb{R}$, for every $x\in X$ defined by $F(x)=\|x\|_X$, is semi-smooth in $X$ and a Newton derivative $\partial_{\mathrm{N}} F\colon X\to \smash{2^X}$, for every $x\in X$, given via $\partial_{\mathrm{N}} F(x)=\{\frac{x}{\|x\|_X}\}$ if $x\neq 0_X$ and $\partial_{\mathrm{N}} F(0)=B_1^X(0)$;

            %\item[(iii)] The function $F\coloneqq \max\{0,\cdot\}\colon \mathbb{R}\to \mathbb{R}$ is Newton differentiable in $\mathbb{R}$ and a Newton~derivative $\partial_{\mathrm{N}} F\colon \mathbb{R}\to 2^{\mathbb{R}}$ %, %for every $x\in \mathbb{R}$,  given via $\partial_{\mathrm{N}} F(x)=\{0\}$ if $x<0$, $\partial_{\mathrm{N}} F(x)=\{1\}$ if $x>0$, and $\partial_{\mathrm{N}} F(0)=[0,1]$.
        \end{itemize}
    \end{remark} 

    Building on %the generalized concept of  differentiation in 
    Definition \ref{def:newton_differentiable}, the semi-smooth Newton method is analogous to the classical Newton method, while employing a Newton derivative selection rather than the Fr\'echet derivative.\vspace{-0.5mm}
    
    \begin{algorithm}[Semi-smooth Newton method]\label{alg:SSN}
    Let $U\subseteq X$ be an open subset and $F\colon U\to Y$.
    Let $\varepsilon_{\mathtt{abs}},\varepsilon_{\mathtt{rel}}>0$ be stopping parameters, let $x^0\in U$ be an initial iterate, and 
 let~${k_{\mathtt{max}}\in \mathbb{N}\cup\{+\infty\}}$ be a maximum number of iterations.
 Then, for  $k=0,\ldots,k_{\mathtt{max}}$, perform the following~\mbox{iteration}~loop:
    \begin{itemize}[noitemsep,topsep=2pt,leftmargin=!,labelwidth=\widthof{(iii)}]
        %\item[(i)] \hypertarget{alg:SSN.ii}{} If $F(x^k)=0_Y$ in $Y$, then  \textup{STOP}; otherwise, continue with Step (\hyperlink{alg:SSN.ii}{ii}); 
        \item[(i)] \hypertarget{alg:SSN.i}{} Check if $F\colon U\to Y$ is Newton differentiable at $x^k\in U$ and compute a Newton derivative selection $\mathtt{J}_F(x^k)\in \mathcal{L}(X;Y)$;
        \item[(ii)] \hypertarget{alg:SSN.ii}{} Compute the \emph{update direction $\delta x^k\in X$} such that\vspace{-0.5mm}
        \begin{align}\label{alg:SSN.1}
            \smash{\mathtt{J}_F(x^k)\delta \smash{x^k}=-F(x^k)\quad\text{ in }Y\,,}\\[-6mm]\notag
        \end{align}
        and the \emph{updated iterate} $\smash{x^{k+1}\hspace{-0.175em}\coloneqq \hspace{-0.175em}x^k\hspace{-0.175em}+\hspace{-0.175em}\alpha_k \delta x^k\hspace{-0.175em}\in \hspace{-0.175em}X}$, where $\alpha_k\hspace{-0.175em}>\hspace{-0.175em}0$ is a~(possibly~\mbox{variable})~\mbox{step~size};
        \item[(iii)] \hypertarget{alg:SSN.iii}{} If $\|F(x^{k+1})\|_Y< \max\{\varepsilon_{\mathtt{abs}},\varepsilon_{\mathtt{rel}}\|F(x^0)\|_Y\}$, then \textup{STOP}; otherwise, $k\to k+1$ and continue with Step (\hyperlink{alg:SSN.i}{i}). 
    \end{itemize}
    \end{algorithm}

    If the mapping  $F\colon \hspace{-0.15em}U\hspace{-0.15em}\to\hspace{-0.15em} Y$ is Newton differentiable at a root $x\hspace{-0.15em}\in\hspace{-0.15em} U$ and 
    in a ball~${B_\varepsilon^X(x)\hspace{-0.15em}\subseteq\hspace{-0.15em} U}$,~${\varepsilon\hspace{-0.15em}>\hspace{-0.15em}0}$, a uniformly  bounded inverse of a Newton derivative selection exists,~then~Algorithm~\ref{alg:SSN}~is~\emph{locally well-posed}, \textit{i.e.},  if %the initial iterate 
    $x^0\in U$ is sufficiently close to $x$, all iterates $\{x^k\}_{k\in \mathbb{N}}\subseteq B_\varepsilon^X(x)$ are computable, and \emph{locally super-linearly convergent} to $x$, \textit{i.e.}, if %the initial iterate 
    $x^0\in U$ is sufficiently close to $x$, for every $\eta>0$, there exists some $K\in \mathbb{N}$, such that for every $k\in \mathbb{N}$ with $k\ge K$, there holds\vspace{-0.5mm}
         \begin{align*}
             \smash{\|x^{k+1}-x\|_X \leq \eta\|x^k-x\|_X\,.}\\[-6mm]
         \end{align*}

    \begin{theorem}[Well-posedness and convergence]\label{thm:wellposed_convergent}
         Let $U\subseteq X$ be an open subset and $F\colon U\to Y$ %a function 
         such that $F(x)\hspace{-0.1em}=\hspace{-0.1em}0_Y$ for some $x\hspace{-0.1em}\in\hspace{-0.1em} U$. If $F\colon\hspace{-0.1em} U\hspace{-0.1em}\to\hspace{-0.1em} Y$
         is Newton differentiable~at~$x\hspace{-0.1em}\in\hspace{-0.1em} U$~with~\mbox{point-wise}\linebreak invertible \hspace{-0.1mm}Newton \hspace{-0.1mm}derivative \hspace{-0.1mm}selection 
         \hspace{-0.1mm}$\mathtt{J}_F\colon \hspace{-0.175em}B_\varepsilon^X(x)\hspace{-0.175em}\to\hspace{-0.175em} \mathcal{L}(X;Y)$ \hspace{-0.1mm}and \hspace{-0.1mm}${\sup_{\widetilde{x}\in B_\varepsilon^X(x)}{\hspace{-0.1em}\{\|\mathtt{J}_F(\widetilde{x})^{-1}\|_{\mathcal{L}(Y;X)}\}}\hspace{-0.175em}}$  $<\hspace{-0.2em}+\infty$ for some $\varepsilon>0$, %then, if the initial iterate $x^0\in U$ is sufficiently close to $x$, 
         %then 
         Algorithm \ref{alg:SSN} (with $\alpha_k=1$ for all $k\in \mathbb{N}$) is locally well-posed  
         and locally super-linearly convergent~to~$x$.
    \end{theorem}

    \begin{proof}
        See \cite[Thm.\ 3.4]{ChenNashedQi2000}.
    \end{proof}\vspace{-1mm}
    \newpage
    
    \subsection{Moreau envelope, proximity operator, resolvent operator, and Yosida approximation}

    \hspace{5mm}In this subsection, we collect the most important results connected with the Moreau envelope,  proximity operator,  resolvent operator, and  Yosida approximation. For~a~\mbox{detailed}~\mbox{presentation}, we \hspace{-0.1mm}refer \hspace{-0.1mm}to \hspace{-0.1mm}\cite{BauschkeCombettes2017,RockafellarWets1998}. \hspace{-0.1mm}Throughout \hspace{-0.1mm}this \hspace{-0.1mm}subsection, \hspace{-0.1mm}if \hspace{-0.1mm}not \hspace{-0.1mm}otherwise \hspace{-0.1mm}specified, \hspace{-0.1mm}let \hspace{-0.1mm}$(X,(\cdot,\cdot)_X)$~\hspace{-0.1mm}be~\hspace{-0.1mm}a~\hspace{-0.1mm}Hilbert space and $\Gamma_0(X)$ the class of proper, convex, and lower semi-continuous~\mbox{functionals}~\mbox{defined}~on~$X$.

    Then, a Fr\'echet differentiable approximation of a functional from $\Gamma_0(X)$, which also preserves the minimizers of the latter, is obtained via infimal convolution with a scaled quadratic functional.
    
    \begin{definition}[Moreau envelope]\label{def:moreau}
        Let $f\colon X\to \mathbb{R}\cup\{+\infty\}$ be a functional. Then, for every~$\gamma>0$, the \emph{Moreau envelope} of $f$ (of \emph{proximity parameter} $\gamma>0$) ${f^\gamma\colon\hspace{-0.1em} X\hspace{-0.1em}\to\hspace{-0.1em} \mathbb{R}\hspace{-0.1em}\cup\hspace{-0.1em}\{+\infty\}}$,~for~\mbox{every}~${x\hspace{-0.1em}\in\hspace{-0.1em} X}$, is defined by
        \begin{align}\label{eq:moreau}
            f^\gamma(x)\coloneqq \inf_{y\in X}{\big\{f(y)+\tfrac{1}{2\gamma}\|x-y\|_X^2\big\}}\,.
        \end{align}
    \end{definition}

    For functionals from  $\Gamma_0(X)$, the Moreau envelope offers the following approximation properties.%, which are illustrated in Figure \ref{fig:moreau_envelope}.

    \begin{lemma}%[Approximation properties of Moreau envelope]
    \label{lem:approx_properties_moreau}
    Let  $f\in \Gamma_0(X)$ be a functional. Then, for every proximity parameter $\gamma>0$, the Moreau envelope $f^\gamma\colon X\to  \mathbb{R}$ is real-valued, convex, and continuous. Moreover, 
    for every $x\in X$ and the associated net $\{f^\gamma(x)\}_{\gamma>0:\gamma\to 0^+}\subseteq \mathbb{R}$, the following statements~apply:
    \begin{itemize}[noitemsep,topsep=2pt,leftmargin=!,labelwidth=\widthof{(iii)}]
        \item[(i)] $\{f^\gamma(x)\}_{\gamma>0:\gamma\to 0^+}\subseteq \mathbb{R}$ is non-decreasing;
        \item[(ii)] $f^\gamma(x)\nearrow f(x)$ $(\gamma\to 0^+)$;
        \item[(iii)]  $f^\gamma(x)\searrow \inf_{y\in X}{\{f(y)\}}$ $(\gamma\to +\infty)$. In particular, %there holds 
        $\inf_{y\in X}{\{f^\gamma(y)\}}=\inf_{y\in X}{\{f(y)\}}$.
    \end{itemize}
    \end{lemma}

    \begin{proof}
        See \cite[Props.\ 12.15 \& 12.32]{BauschkeCombettes2017}.
    \end{proof}

    The approximation properties of Moreau envelopes (\textit{cf.} Lemma \ref{lem:approx_properties_moreau}) are illustrated in Figure~\ref{fig:moreau_envelope} by  two \hspace{-0.1mm}examples \hspace{-0.1mm}of \hspace{-0.1mm}functionals \hspace{-0.1mm}from \hspace{-0.1mm}$\Gamma_0(\mathbb{R}^d)$ 
 \hspace{-0.1mm}relevant~\hspace{-0.1mm}for~\hspace{-0.1mm}the~\hspace{-0.1mm}forthcoming~\hspace{-0.1mm}\mbox{analysis}~(\textit{cf}.~\hspace{-0.1mm}\mbox{\cite[Chap.~\hspace{-0.1mm}6]{Beck2017}}).\enlargethispage{8mm}

    \begin{examples}[for Moreau envelopes]\label{expl:moreau_envelopes} Let $\gamma>0$ be a proximity parameter and $d\in \mathbb{N}$.
        \begin{itemize}[noitemsep,topsep=2pt,leftmargin=!,labelwidth=\widthof{(ii)}]
            \item[(i)] \hypertarget{expl:moreau_envelopes.i}{} The Moreau envelope of the Euclidean distance $\smash{f\hspace{-0.1em}\coloneqq \hspace{-0.1em}\vert \cdot\vert\hspace{-0.1em}\in \hspace{-0.1em}\Gamma_0(\mathbb{R}^d)}$, called \emph{Huber~\mbox{regularization}} (\textit{cf}.\ \cite{Huber1964}), $ f^\gamma\colon \mathbb{R}^d\to \mathbb{R}$ (\textit{cf}.\ Figure \ref{fig:moreau_envelope}(\textit{left})), for every $t\in \mathbb{R}^d$,~is~given~via 
        \begin{align*}
            \smash{f^\gamma(t)
            =\min\{\tfrac{1}{2\gamma} \vert t\vert^2,\vert t\vert-\tfrac{\gamma}{2}\}\,.}
            %\begin{cases}
            %    \frac{1}{2\gamma} \vert t\vert^2& \text{ if }\vert t\vert\leq \gamma\,,\\
            %     \vert t\vert-\frac{\gamma}{2}& \text{ if }\vert t\vert> \gamma\,.\\
            %\end{cases}
        \end{align*}
         \item[(ii)] \hypertarget{expl:moreau_envelopes.ii}{} The Moreau envelope of the functional $\smash{f\coloneqq \frac{1}{2}\vert \cdot\vert^2+I_{\smash{K_1^d(0)}}\in \Gamma_0(\mathbb{R}^d)}$\footnote{Here, $K_1^d(0)\coloneqq \{x\in \mathbb{R}^d\mid \vert x\vert\leq 1\}$},
         where~${I_{\smash{K_1^d(0)}}(t)\coloneqq 0}$~if $\vert t\vert\hspace{-0.12em} \leq\hspace{-0.12em} 1$ and $I_{\smash{K_1^d(0)}}(t)\hspace{-0.12em}\coloneqq\hspace{-0.12em} +\infty$ else, $f^\gamma\hspace{-0.12em}\colon \hspace{-0.12em}\mathbb{R}^d\hspace{-0.15em}\to\hspace{-0.15em} \mathbb{R}$ (\textit{cf}.\ Figure \ref{fig:moreau_envelope}(\textit{right})), for~every~${t\hspace{-0.12em}\in\hspace{-0.12em} \mathbb{R}^d}$,~is~given~via\footnote{Here, $(\cdot)_+\coloneqq\max\{0,\cdot\}\colon \mathbb{R}\to \mathbb{R}_{\ge 0}$.}
        \begin{align*}
            \smash{f^\gamma(t)=\tfrac{1}{2(1+\gamma)}\vert t\vert^2+\tfrac{1}{2\gamma(1+\gamma)}(\vert t\vert-(1+\gamma))_+^2
            %\max\{\tfrac{1}{2(1+\gamma)} \vert t\vert^2,\tfrac{1}{2\gamma}(\vert t\vert-1)^2+\tfrac{1}{2}\}
            \,.}
            %\begin{cases}
            %    \frac{1}{2\gamma} \vert t\vert^2& \text{ if }\vert t\vert\leq \gamma\,,\\
            %     \vert t\vert-\frac{\gamma}{2}& \text{ if }\vert t\vert> \gamma\,.\\
            %\end{cases}
        \end{align*}
        \end{itemize} 
    \end{examples}
    
    \begin{figure}[H]\vspace{-1mm}
        \centering
        \includegraphics[width=0.5\linewidth]{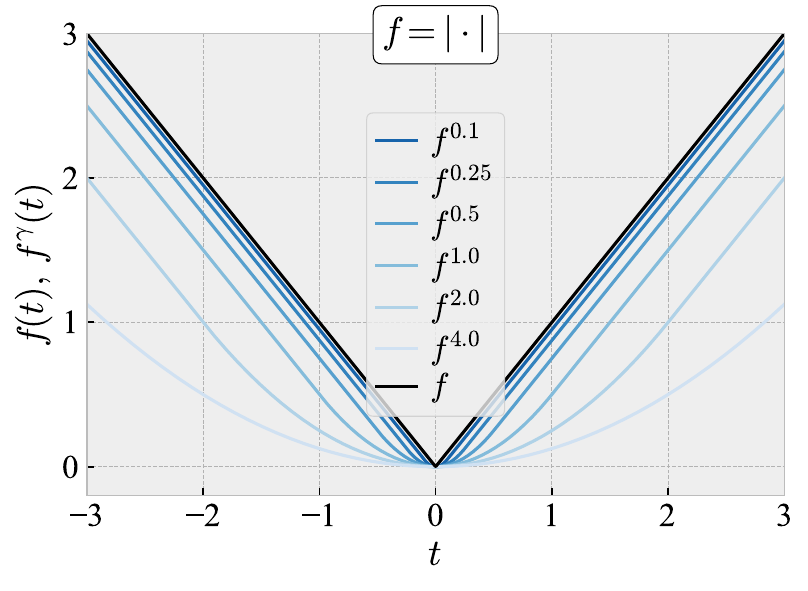}\includegraphics[width=0.5\linewidth]{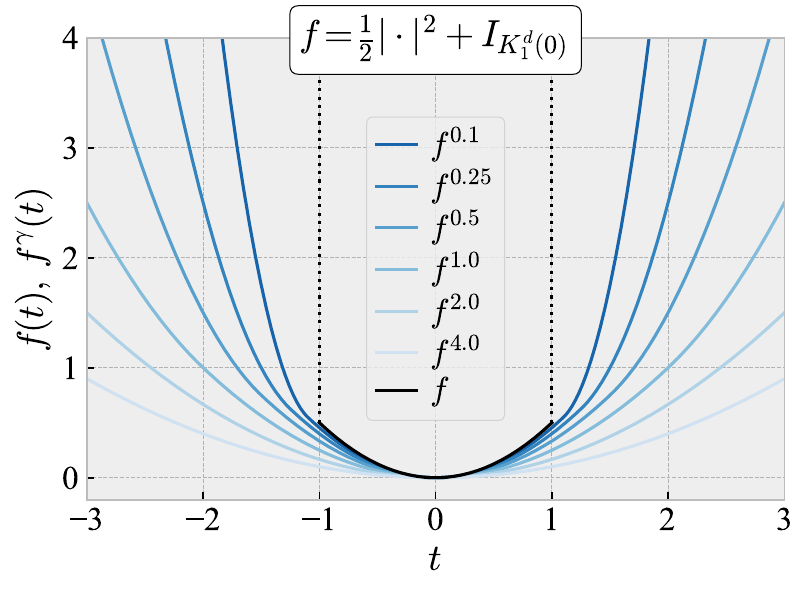}\vspace{-2.5mm}
        \caption{\textit{left:} Euclidean distance $f\hspace{-0.15em}\coloneqq\hspace{-0.15em} \vert \cdot\vert \hspace{-0.15em}\in\hspace{-0.15em} \Gamma_0(\mathbb{R})$; \textit{right:} the functional $f\hspace{-0.15em}\coloneqq\hspace{-0.15em} \frac{1}{2}\vert \cdot\vert^2+I_{\smash{K_1^d(0)}}\hspace{-0.15em}\in \hspace{-0.15em}\Gamma_0(\mathbb{R})$; each with Moreau envelopes $f^\gamma\colon \mathbb{R}\to \mathbb{R}$, $\gamma\in \{0.1(\text{{\color{blue1}\raisebox{2pt}{\rule{0.5em}{1pt}}}}), 0.25(\text{{\color{blue2}\raisebox{2pt}{\rule{0.5em}{1pt}}}}), 0.5(\text{{\color{blue3}\raisebox{2pt}{\rule{0.5em}{1pt}}}}), 1(\text{{\color{blue4}\raisebox{2pt}{\rule{0.5em}{1pt}}}}), 2(\text{{\color{blue5}\raisebox{2pt}{\rule{0.5em}{1pt}}}}), 4(\text{{\color{blue6}\raisebox{2pt}{\rule{0.5em}{1pt}}}})\}$,~which,~consistent~with Lemma \ref{lem:approx_properties_moreau}, approximate $f\in \Gamma_0(\mathbb{R}) \text{ ({\color{black}\raisebox{2pt}{\rule{0.5em}{1pt}}})}$ point-wise from below with~matching~\mbox{minima}.}
        \label{fig:moreau_envelope}
    \end{figure}\newpage
    
    Beyond that, for functionals from $\Gamma_0(X)$, by the direct method in the calculus~of~variations, the infimum in \eqref{eq:moreau} is always attained, which gives rise to the following definition.

    \begin{definition}[Proximity operator]\label{def:proximity}
        Let \hspace{-0.1mm}$f\hspace{-0.175em}\in \hspace{-0.175em}\Gamma_0(X)$ \hspace{-0.1mm}be \hspace{-0.1mm}a \hspace{-0.1mm}functional. \hspace{-0.1mm}Then, \hspace{-0.1mm}the~\hspace{-0.1mm}\emph{\mbox{proximity}~\hspace{-0.1mm}\mbox{operator}} of $f$ --or \emph{proximal mapping}-- $\mathrm{prox}_{f}\colon X\to X$, for every $x\in X$, is defined by\vspace{-0.5mm}
        \begin{align}\label{eq:proximity}
            \mathrm{prox}_{f}(x)\coloneqq \underset{y\in X}{\operatorname{arg\,min}}{\big\{f(y)+\tfrac{1}{2}\|x-y\|_X^2\big\}}\,.\\[-5.5mm]\notag
        \end{align}
    \end{definition}

    Important properties of the proximity operator  are collected in the following~lemma.

    \begin{lemma}%[Properties of the proximity operator]
    \label{lem:prop_proximity}
        Let $f\in \Gamma_0(X)$ be a functional. Then, the following statements~apply:
        \begin{itemize}[noitemsep,topsep=2pt,leftmargin=!,labelwidth=\widthof{(iii)}]
            \item[(i)] \hypertarget{lem:prop_proximity.i}{} For every $x\in X$ and $\gamma>0$, there holds the \emph{Moreau decomposition}\vspace{-0.5mm}
        \begin{align}\label{lem:prop_proximity.i}
            \smash{x=\gamma\,\mathrm{prox}_{\smash{\frac{1}{\gamma}f^*}}(\smash{\tfrac{1}{\gamma}x})+\mathrm{prox}_{\gamma f}(x)}\,,\\[-5.5mm]\notag
        \end{align}
        where, identifying $X$ with its (topological) dual space $X^*$, %by Riesz representation theorem, 
        the \emph{(Fenchel) conjugate} functional $f^*\colon X\to \mathbb{R}\cup\{+\infty\}$, 
        for every $x\in X$, is defined by $f^*(x)\coloneqq \sup_{y\in X}{\{(x,y)_X-f(y)\}}$;
        \item[(ii)] \hypertarget{lem:prop_proximity.ii}{} The proximity operator $\mathrm{prox}_{f}\colon \hspace{-0.125em}X\hspace{-0.125em}\to\hspace{-0.125em} X$  and $\mathbbone_X-\mathrm{prox}_{f}\colon \hspace{-0.125em}X\to\hspace{-0.125em} X$, where~${\mathbbone_X\hspace{-0.125em}\in\hspace{-0.125em} \mathcal{L}(X)\hspace{-0.125em}\coloneqq\hspace{-0.125em} \mathcal{L}(X;X)}$ denotes the identity operator (\textit{i.e.}, $\mathbbone_X(x)\coloneqq x$ for all $x\in X$), are monotone and $1$-Lipschitz continuous. More specifically,  for every $x,y\in X$, there holds\vspace{-0.5mm}
        \begin{subequations}\label{lem:prop_proximity.ii}
            \begin{align}
                \|\mathrm{prox}_{f}(x)-\mathrm{prox}_{f}(y)\|^2_X&\leq (\mathrm{prox}_{f}(x)-\mathrm{prox}_{f}(y),x-y)_X\,,\\
                \|x-\mathrm{prox}_{f}(x)-(y-\mathrm{prox}_{f}(y))\|^2_X&\leq (x-\mathrm{prox}_{f}(x)-(y-\mathrm{prox}_{f}(y)),x-y)_X\,;
            \end{align}
        \end{subequations}
        %    In particular, if $X=\mathbb{R}^n$, $n\in  \mathbb{N}$, the proximity operator $\mathrm{prox}_{f}\colon X\to X$~is~a.e.~\mbox{differentiable};

        \item[(iii)] \hypertarget{lem:prop_proximity.iii}{} The proximity operator $\mathrm{prox}_{f}\colon X\to X$ is Newton differentiable point-wise at every $x\in X$ and for every $\mathtt{L}\in \partial_{\mathrm{N}} \mathrm{prox}_{f}(x)$, there holds $\mathtt{L}\succeq \mathbbzero_X$ as well as $\|\mathtt{L}\|_{\mathcal{L}(X)}\leq 1$, where $\mathbbzero_X\in \mathcal{L}(X)$ is the zero operator  (\textit{i.e.}, $\mathbbzero_X(x)=0_X$ for all $x\in X$).

        \item[(iv)] \hypertarget{lem:prop_proximity.iv}{} If the functional $f\colon X\to \mathbb{R}$ is twice continuously Fr\'echet differentiable with second Fr\'echet derivative $\mathrm{D}^2f\colon X\to \mathcal{L}(X)$, the proximity operator $\mathrm{prox}_{f}\colon X\to X$ is Fr\'echet differentiable with Fr\'echet derivative $\mathrm{D}\mathrm{prox}_{f}\colon X\to \mathcal{L}(X)$, for every $x\in X$ given via\vspace{-1mm}
        \begin{align}\label{lem:prop_proximity.iv}
            \mathrm{D}\mathrm{prox}_{f}(x)=\big(\mathbbone_X+\mathrm{D}^2f(\mathrm{prox}_{f}(x))\big)^{-1}\quad\text{ in }\mathcal{L}(X)\,.
        \end{align}
        \end{itemize}
    \end{lemma}

    \begin{proof}
        \emph{ad \hspace{-0.1mm}(\hyperlink{lem:prop_proximity.i}{i}).} \hspace{-0.1mm}See \hspace{-0.1mm}\cite[Thm.\ \hspace{-0.1mm}14.3(ii)]{BauschkeCombettes2017}.\quad\hspace{-2mm} \emph{ad \hspace{-0.1mm}(\hyperlink{lem:prop_proximity.ii}{ii}).} \hspace{-0.1mm}See \hspace{-0.1mm}\cite[Prop.\ \hspace{-0.1mm}12.28]{BauschkeCombettes2017}.\quad\hspace{-2mm}  \emph{ad \hspace{-0.1mm}(\hyperlink{lem:prop_proximity.iii}{iii}).} \hspace{-0.1mm}See \hspace{-0.1mm}Remark~\hspace{-0.1mm}\ref{rem:newton_differentiable}(\hyperlink{rem:newton_differentiable.iii}{iii}).
 
%Although the result is well known and often used implicitly in the literature, explicit proofs are rarely given.  For completeness, we present a self-contained argument based on the implicit function theorem.

\emph{ad (\hyperlink{lem:prop_proximity.iv}{iv}).} By the definition of  $\operatorname{prox}_f\colon X\to X$ (\textit{cf}.\ Definition \ref{def:proximity}), for every $x\in X$,~the~element $\operatorname{prox}_f(x)\in X$ is the unique root of the mapping $F(x,\cdot)\colon X\to X$, for every $y\in X$~defined~by
\begin{align}\label{lem:prop_proximity.iv.1}
F(x,y) \coloneqq y - x + \mathrm{D} f(y)\quad\text{ in }X\,.
\end{align}
This is a consequence of the first-order optimality condition associated with the minimization problem in \eqref{eq:proximity} and the assumed continuous Fr\'echet differentiability of the functional $f\colon X\to \mathbb{R}$. 

Due to the assumed twice continuous Fr\'echet differentiability of  the functional $f\colon X\to \mathbb{R}$,
the mapping $F\colon  X \times X \to X$, defined via  \eqref{lem:prop_proximity.iv.1}, is continuously Fr\'echet differentiable with Fr\'echet derivative $\mathrm{D}F\colon X\times X\to \mathcal{L}(X\times X;X)$, for every $(x,y)\in X\times X$~given~via
\begin{align*}
    \mathrm{D}F(x,y)=\left[\begin{array}{c}
         \mathrm{D}_x F(x,y)  \\
         \mathrm{D}_y F(x,y)
    \end{array}\right]=\left[\begin{array}{c}
        -\mathbbone_X  \\
        \mathbbone_X + \mathrm{D}^2 f(y)
    \end{array}\right]\quad \text{ in }\mathcal{L}(X\times X;X)\,.
\end{align*}
By the convexity and twice continuous Fr\'echet differentiability of $f\colon X\to \mathbb{R}$, for every $y\in X$, the second Fr\'echet derivative $\mathrm{D}^2 f(y)\in \mathcal{L}(X)$ is bounded, self-adjoint,~positive~semi-definite~and,~thus, $ \mathrm{D}_y F(\cdot,y)\hspace{-0.175em}=\hspace{-0.175em}\mathbbone_X+\mathrm{D}^2 f(y)\hspace{-0.175em}\in\hspace{-0.175em} \mathcal{L}(X)$ \hspace{-0.15mm}an \hspace{-0.15mm}isomorphism. \hspace{-0.15mm}Then, \hspace{-0.15mm}due \hspace{-0.15mm}to \hspace{-0.15mm}$F(x,\operatorname{prox}_f(x))\hspace{-0.175em}=\hspace{-0.175em}0$~\hspace{-0.15mm}in~\hspace{-0.15mm}$X$~\hspace{-0.15mm}for~\hspace{-0.15mm}all~\hspace{-0.15mm}${x\hspace{-0.175em}\in\hspace{-0.175em} X}$,\linebreak by the implicit function~theorem~(\textit{cf}.~\mbox{\cite[p.~133]{HildebrandtGraves1927}}), for every $x\in X$, 
there exist open sets $U_x,V_x\subseteq X$ with $x\in U_x$ and $\operatorname{prox}_f(x)\in V_x$ and a continuously Fr\'echet differentiable~function~$y_x\colon U_x\to V_x$ such that $F(\widehat{x},y_x(\widehat{x}))=0$ in $X$ for all $\widehat{x}\in U_x$, which, due to the uniqueness of the root of the mapping $F(\widehat{x},\cdot)\colon X\to X$ for all $\widehat{x}\in X$, satisfies $y_x(\widehat{x})=\operatorname{prox}_f(\widehat{x})$ for all $\widehat{x}\in U_x$.  In other words, the proximity operator is continuously Fr\'echet differentiable in $ U_x$. Beyond that,  for every $\widehat{x}\in U_x$, we have that $\mathrm{D}y_x(\widehat{x})=-(\mathrm{D}_y F(\widehat{x},y_x(\widehat{x})))^{-1}\mathrm{D}_x F(\widehat{x},y_x(\widehat{x}))=(\mathbbone_X+\mathrm{D}^2 f(y_x(\widehat{x})))^{-1}$ in $\mathcal{L}(X)$, which implies the claimed representation \eqref{lem:prop_proximity.iv} of the Fr\'echet derivative of $\operatorname{prox}_f\colon X\to X$. 
    \end{proof}
    
    Note that point (\hyperlink{lem:prop_proximity.iii}{iii}) in Lemma \ref{lem:prop_proximity} asserts only \emph{point-wise} Newton differentiability of proximity operators, not \emph{global} (or even \emph{local}) Newton differentiability. However,~global~Newton~differentiability is typically satisfied in practice and can be verified~on~a~\textit{\mbox{case-by-case}~basis}~(\textit{cf}.~\mbox{\cite[Chap.~6]{Beck2017}}).

    \begin{examples}[for proximity operators]\label{expl:proximity} Let $\gamma>0$ be a proximity parameter and $d\in \mathbb{N}$. 
        \begin{itemize}[noitemsep,topsep=2pt,leftmargin=!,labelwidth=\widthof{(ii)}]
            \item[(i)] \hypertarget{expl:proximity.i}{} The proximity operator $ \operatorname{prox}_{\gamma f}\colon \mathbb{R}^d\to \mathbb{R}^d$ (\textit{cf}.\ Figure \ref{fig:proximity_operators}(\textit{left}))~of~the~Euclidean~distance~$f\coloneqq \vert \cdot\vert\in \Gamma_0(\mathbb{R}^d)$, for every $t\in \mathbb{R}^d$, is given via\vspace{-0.5mm} 
        \begin{align}\label{expl:proximity.1}
            \smash{\operatorname{prox}_{\gamma f}(t)=(\vert t\vert-\gamma)_+\tfrac{t}{\vert t\vert }\,,}
            %\begin{cases}
            %    \frac{1}{2\gamma} \vert t\vert^2& \text{ if }\vert t\vert\leq \gamma\,,\\
            %     \vert t\vert-\frac{\gamma}{2}& \text{ if }\vert t\vert> \gamma\,.\\
            %\end{cases}
        \end{align}
        and, thus, not only Newton differentiable point-wise at every $t\in \mathbb{R}^d$, %(\textit{cf}.~Lemma~\ref{lem:prop_proximity}(\hyperlink{lem:prop_proximity.iii}{iii})),
        but Newton differentiable in  $\mathbb{R}^d$ with Newton derivative $\partial_{\mathrm{N}} \smash{\operatorname{prox}_{\gamma f}}\colon \mathbb{R}^d\to  2^{\mathbb{R}^d}$, for every $t\in \mathbb{R}^d $ given via\footnote{By $\mathbbone\coloneqq(\delta_{ij})_{i,j\in \{1,\ldots,d\}},\mathbbzero\in \mathbb{R}^{d\times d}$, we denote the \emph{identity} and \emph{zero matrix}, respectively.}\vspace{-0.5mm} 
        \begin{align}\label{expl:proximity.2}
            \partial_{\mathrm{N}} \operatorname{prox}_{\gamma f}(t)=\begin{cases}
                \mathbbone-\tfrac{\gamma}{\vert t\vert}(\mathbbone-\tfrac{t\otimes t}{\vert t\vert^2})&\text{ if }\vert t\vert >\gamma\,,\\[-0.25mm]
                [0,1]\tfrac{t\otimes t}{\vert t\vert^2}&\text{ if }\vert t\vert =\gamma\,,\\[-0.25mm]
                \mathbbzero&\text{ if }\vert t\vert <\gamma\,;
            \end{cases}
        \end{align}
         \item[(ii)] \hypertarget{expl:proximity.ii}{} The proximity operator $\operatorname{prox}_{\gamma f}\colon \mathbb{R}^d\to \mathbb{R}^d$ (\textit{cf}.\ Figure \ref{fig:proximity_operators}(\textit{right})) of the functional $f\coloneqq \frac{1}{2}\vert \cdot\vert^2$ $+I_{\smash{K_1^d(0)}}\in \Gamma_0(\mathbb{R}^d)$, for every $t\in \mathbb{R}^d$, is given via\vspace{-0.5mm} 
        \begin{align}\label{expl:proximity.3}
             \smash{\operatorname{prox}_{\gamma f}(t)= 
               \min\{\tfrac{1}{1+\gamma},\tfrac{1}{\vert t\vert}\}t\,,}
        \end{align}
        and, thus, not only Newton differentiable point-wise at every $t\in \mathbb{R}^d$,  %(\textit{cf}.~Lemma~\ref{lem:prop_proximity}(\hyperlink{lem:prop_proximity.iii}{iii})),~
        but Newton differentiable in  $\mathbb{R}^d$ with Newton derivative $\partial_{\mathrm{N}}  \smash{\operatorname{prox}_{\gamma f}}\colon \mathbb{R}^d\to  2^{\mathbb{R}^d}$, for every $t\in \mathbb{R}^d$~given~via\vspace{-0.5mm} 
        \begin{align}\label{expl:proximity.4}
            \partial_{\mathrm{N}} \operatorname{prox}_{\gamma f}(t)=\begin{cases}
                \tfrac{1}{\vert t\vert}(\mathbbone-\tfrac{t\otimes t}{\vert t\vert^2})&\text{ if }\vert t\vert >1+\gamma\,,\\[-0.25mm]
               \tfrac{1}{1+\gamma}(\mathbbone- [0,1]\tfrac{t\otimes t}{\vert t\vert^2})&\text{ if }\vert t\vert =1+\gamma\,,\\[-0.25mm]
                \tfrac{1}{1+\gamma}\mathbbone&\text{ if }\vert t\vert <1+\gamma\,.
            \end{cases}
        \end{align}
        \end{itemize}\vspace{-4mm}
    \end{examples}

    \begin{figure}[H]\vspace{-1mm}
        \centering
        \includegraphics[width=0.5\linewidth]{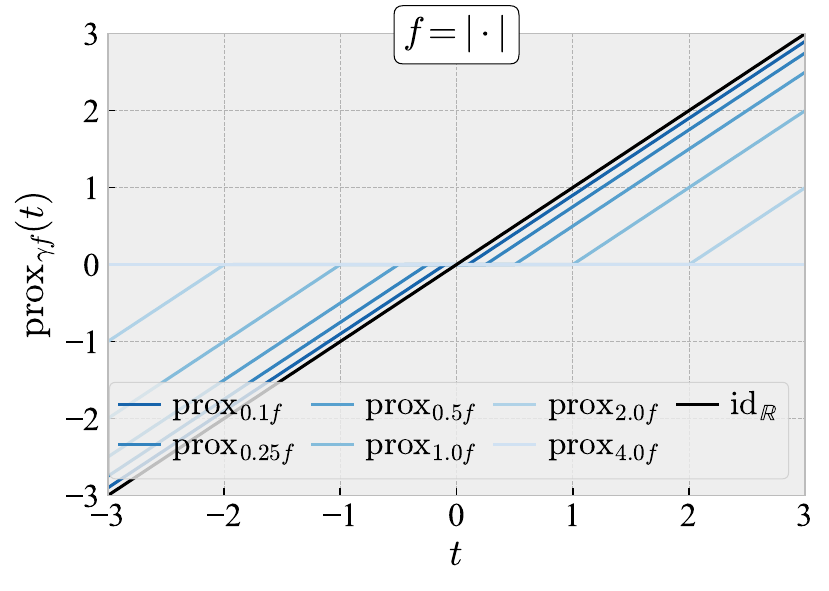}\includegraphics[width=0.5\linewidth]{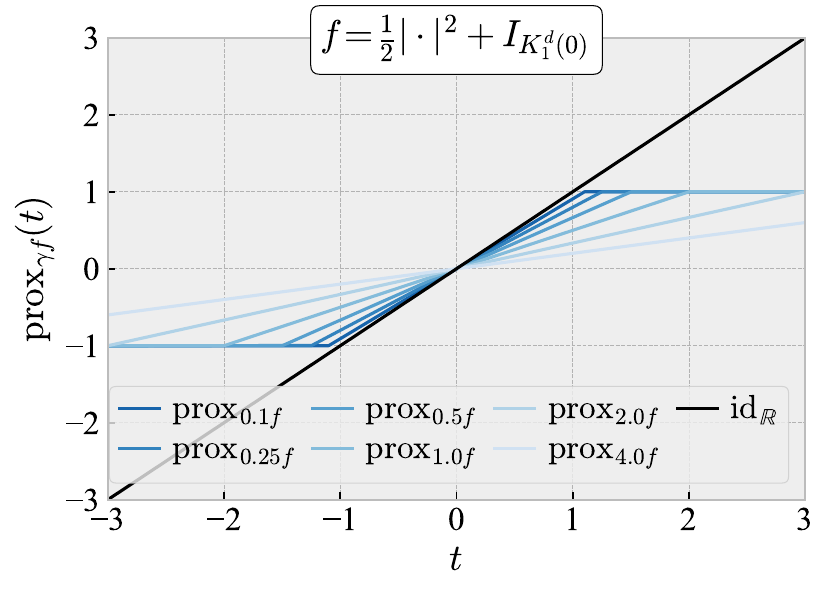}\vspace{-2.5mm}
        \caption{\textit{left:} Euclidean distance $f\hspace{-0.15em}\coloneqq\hspace{-0.15em} \vert \cdot\vert \hspace{-0.15em}\in\hspace{-0.15em} \Gamma_0(\mathbb{R})$; \textit{right:} the functional $f\hspace{-0.15em}\coloneqq\hspace{-0.15em} \frac{1}{2}\vert \cdot\vert^2+I_{\smash{K_1^d(0)}}\hspace{-0.15em}\in \hspace{-0.15em}\Gamma_0(\mathbb{R})$; each with proximity operators $\operatorname{prox}_{\gamma f}\colon \mathbb{R}\to \mathbb{R}$, $\gamma\in \{0.1(\text{{\color{blue1}\raisebox{2pt}{\rule{0.5em}{1pt}}}}), 0.25(\text{{\color{blue2}\raisebox{2pt}{\rule{0.5em}{1pt}}}}), 0.5(\text{{\color{blue3}\raisebox{2pt}{\rule{0.5em}{1pt}}}}), 1(\text{{\color{blue4}\raisebox{2pt}{\rule{0.5em}{1pt}}}}), 2(\text{{\color{blue5}\raisebox{2pt}{\rule{0.5em}{1pt}}}}), 4(\text{{\color{blue6}\raisebox{2pt}{\rule{0.5em}{1pt}}}})\}$, which, consistent with Lemma \ref{lem:resolvent}(\protect\hyperlink{lem:resolvent.ii}{ii}), approximate the identity operator $\operatorname{id}_{\mathbb{R}}\in \mathcal{L}(\mathbb{R}) \text{ ({\color{black}\raisebox{2pt}{\rule{0.5em}{1pt}}})}$.\vspace{-1mm}}
        \label{fig:proximity_operators}
    \end{figure}

    The precise relation between the Moreau envelope and the proximity operators --in particular, with regard to their derivatives-- can be found in the following lemma.\vspace{-0.5mm}\enlargethispage{5mm}

    \begin{lemma}\label{lem:relations}%[Relations between Moreau envelope and proximity operator]
        Let $f\in \Gamma_0(X)$ be a functional. Then, the following statements apply:
        \begin{itemize}[noitemsep,topsep=2pt,leftmargin=!,labelwidth=\widthof{(iii)}]
            \item[(i)] \hypertarget{lem:relations.i}{}For every $x\in X$ and $\gamma>0$, we have that $f^\gamma(x)=f(\mathrm{prox}_{\gamma f}(x))+\tfrac{1}{2\gamma}\|x-\mathrm{prox}_{\gamma f}(x)\|_X^2$;
            %\begin{align*}
            %    f^\gamma(x)=f(\mathrm{prox}_{\gamma f}(x))+\tfrac{1}{2\gamma}\|x-\mathrm{prox}_{\gamma f}(x)\|_X^2\,;
            %\end{align*}
            \item[(ii)]\hypertarget{lem:relations.ii}{} For every  $\gamma>0$, the Moreau envelope $f^\gamma\colon X\to \mathbb{R}$ is Fr\'echet differentiable with $\tfrac{1}{\gamma}$-Lipschitz continuous Fr\'echet derivative $\mathrm{D}f^{\gamma}\colon X\to X$, for every $x\in X$ given via\vspace{-0.5mm}
            \begin{align*}
                \mathrm{D}f^{\gamma}(x)=\tfrac{1}{\gamma}(x-\mathrm{prox}_{\gamma f}(x))\,.
            \end{align*}
            %In particular, the Fr\'echet derivative $\mathrm{D}f^{\gamma}\colon X\to X$ is $\tfrac{1}{\gamma}$-Lipschitz continuous;
            %\item[(iii)] For every $\lambda>0$, the functional $f\in \Gamma_0(X)$ is piece-wise linear-quadratic (in the sense of ??) if and only if is Moreau envelope $f^\gamma\colon X\to \mathbb{R}$ is piece-wise linear-quadratic if and only if the proximity operator is $\mathrm{prox}_{f}\colon X\to X$ is piece-wise linear.
            \item[(iii)]\hypertarget{lem:relations.iii}{} For every  $\lambda,\gamma>0$ and $x\in X$, we have that  
            \begin{align*}
            \operatorname{prox}_{\gamma f^\lambda}(x)=\tfrac{1}{\gamma+\lambda}(\lambda x+\gamma\operatorname{prox}_{(\gamma+\lambda)f}(x))\quad\text{ in }X\,. 
            \end{align*}
            In particular, the proximity operator $ \operatorname{prox}_{\gamma f^\lambda}\colon X\to X$ is $\frac{\lambda}{\gamma+\lambda}$-strongly monotone and, thus,  
            for every $x\in X$ and $\mathtt{L}\in \partial_{\mathrm{N}} \operatorname{prox}_{\gamma f^\lambda}(x)$, there holds $\mathtt{L}\succeq  \frac{\lambda}{\gamma+\lambda}\mathbbone_{X}$\footnote{For  
$\mathtt{A},\mathtt{B}\in \mathbb{R}^{d\times d}$, we write $\mathtt{A}\preceq \mathtt{B}$ (or $\mathtt{A}\succeq \mathtt{B}$) if 
$((\mathtt{B}-\mathtt{A})t)\cdot t\ge 0$ (or $((\mathtt{B}-\mathtt{A})t)\cdot t\le 0$)~for~all~$t\in \mathbb{R}^d$.}.
        \end{itemize}
    \end{lemma}

    \begin{proof}
        \emph{ad \hspace{-0.1mm}(\hyperlink{lem:relations.i}{i}).} \hspace{-0.1mm}See \hspace{-0.1mm}\cite[\hspace{-0.1mm}Rem.\ \hspace{-0.1mm}12.24]{BauschkeCombettes2017}.\quad\hspace{-2mm}  \emph{ad \hspace{-0.1mm}(\hyperlink{lem:relations.ii}{ii}).} \hspace{-0.1mm}See \hspace{-0.1mm}\cite[\hspace{-0.1mm}Prop.\ \hspace{-0.1mm}12.30]{BauschkeCombettes2017}.\quad\hspace{-2mm} \hspace{-0.1mm}\emph{ad \hspace{-0.1mm}(\hyperlink{lem:relations.iii}{iii}).} \hspace{-0.1mm}See \hspace{-0.1mm}\cite[\hspace{-0.1mm}Cor.~\hspace{-0.1mm}6.64]{Beck2017}.
        %See \cite[Rem.\ 12.24 \& Prop.\ 12.29]{BauschkeCombettes2017} and \cite[Prop.\ 12.30]{RockafellarWets1998}.
    \end{proof}

    We  recall two related concepts from monotone operator theory: the resolvent and the Yosida approximation, which provide equivalent regularizations of subdifferentials~of~convex functionals.\enlargethispage{5mm}\vspace{-0.5mm}

    %\newpage
    \begin{definition}[Resolvent and Yosida approximation]
        Let $f\in \Gamma_0(X)$ be a functional with subdifferential $\partial f\colon X\to \smash{2^X}$, where $\smash{2^X}$ denotes the power set of $X$, for every $x\in X$ defined by
        \begin{align*}
            \partial f(x)\coloneqq \begin{cases}
                \{x^*\in X\mid  f-f(x)\ge ( x^*,(\cdot)-x)_X\text{ in }X\}&\text{ if }f(x)<+\infty\,,\\
                \emptyset &\text{ else}\,. 
            \end{cases}
        \end{align*}
        Then, the \emph{Yosida approximation} of $\partial f\colon X\to \smash{2^X}$ is defined by
        $(\partial f)^\gamma\coloneqq \frac{1}{\gamma} (\mathbbone_X-J_{\gamma\partial f})\colon X\to X$, where $\smash{J_{\gamma\partial f}\coloneqq (\mathbbone_X+\gamma\partial f)^{-1}\colon X\to X}$ denotes the \emph{resolvent} of $\partial f\colon X\to \smash{2^X}$. 
    \end{definition}

        The precise relations between the Moreau envelope, proximity operator,  resolvent, and  Yosida approximation are presented in the following lemma.\enlargethispage{2.5mm}\vspace{-0.5mm}
        
    \begin{lemma}\label{lem:resolvent}
        Let $f\in \Gamma_0(X)$ be a functional. Then, the following statements apply:
        \begin{itemize}[noitemsep,topsep=2pt,leftmargin=!,labelwidth=\widthof{(iii)}]
            \item[(i)] \hypertarget{lem:resolvent.i}{} For every $x\in X$ and $\gamma>0$, there holds $ J_{\gamma \partial f}(x)=\operatorname{prox}_{\gamma f}(x)$ and 
            $(\partial f)^\gamma(x)=\mathrm{D} f^\gamma(x)$~in~$X$;
            %\vspace{-0.5mm}
            %\begin{alignat*}{2}
            %    J_{\gamma \partial f}(x)&=\operatorname{prox}_{\gamma f}(x)&\quad\text{ in }X\,,\\
            %    (\partial f)^\gamma(x)&=\nabla f^\gamma(x)&\quad\text{ in }X\,;
            %\end{alignat*}
            \item[(ii)] \hypertarget{lem:resolvent.ii}{} For every $x\in X$, we have that\vspace{-0.5mm}
            \begin{alignat*}{3}
                J_{\gamma \partial f}(x)&\to x&&\quad\text{ in }X&&\quad(\gamma\to 0^+)\,,\\
               (\partial f)^\gamma(x)&\to \operatorname{arg\,min}_{y\in \partial f(x)}{\{\|y\|_X\}}&&\quad\text{ in }X&&\quad (\gamma\to 0^+)\,;
            \end{alignat*} 
             \item[(iii)] \hypertarget{lem:resolvent.iii}{} If the functional $f\in \Gamma_0(X)$ is \emph{$\mu$-strongly convex} for some $\mu>0$, \textit{i.e.}, for every $x,y\in X$ and $x^*\in \partial f(x)$, there holds
         \begin{align*}
             f(y)\ge f(x)+(x^*,y-x)_X+\tfrac{\mu}{2}\|x-y\|_X^2\,,
         \end{align*}
         then the proximity operator $\mathrm{prox}_{f}\colon X\to X$ is $\frac{1}{1+\mu}$-Lipschitz continuous. In particular, for every $x\in X$ and  $\mathtt{L}\in \partial_{\mathrm{N}} \operatorname{prox}_{f}(x)$, there holds $\|\mathtt{L}\|_{\mathcal{L}(X)}\leq \frac{1}{1+\mu}$.
        \end{itemize}
    \end{lemma} 

    \begin{proof}
        \emph{ad (\hyperlink{lem:resolvent.i}{i}).} See \cite[Expl.\ 23.3]{BauschkeCombettes2017}.\quad \emph{ad (\hyperlink{lem:resolvent.ii}{ii}).} See \cite[Cor.\ 23.46 \& Thm.\ 23.48]{BauschkeCombettes2017}.\quad 
        \emph{ad (\hyperlink{lem:resolvent.iii}{iii}).}  See \cite[Prop.\ 23.13]{BauschkeCombettes2017}. %(or \cite[Thm.\ 12.56]{RockafellarWets1998}
    \end{proof}\vspace{-1.5mm}
    
    \begin{figure}[H]\vspace{-2.5mm}
        \centering
        \includegraphics[width=0.515\linewidth]{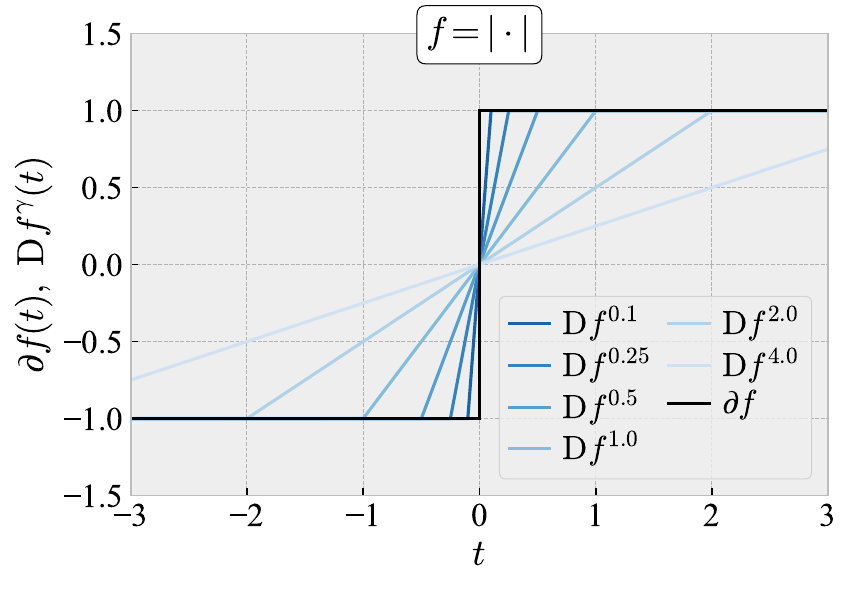}\includegraphics[width=0.5\linewidth]{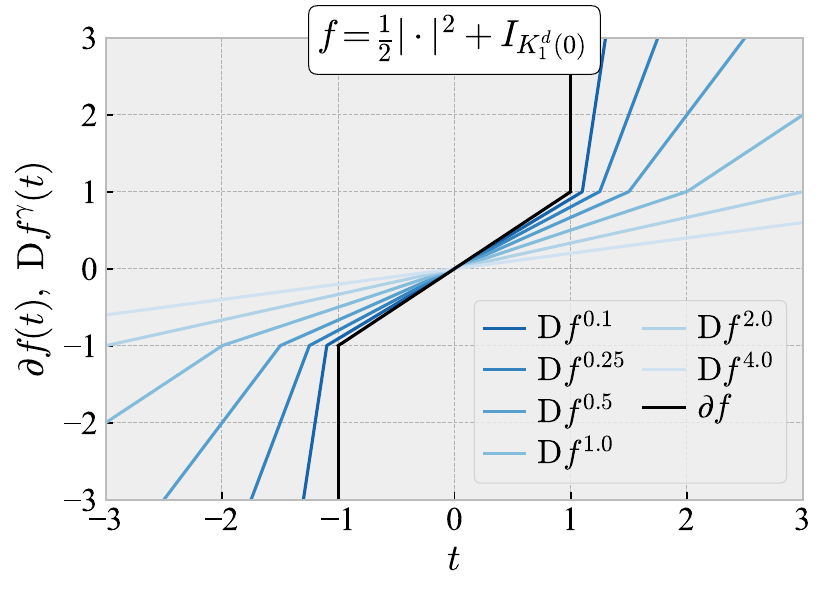}\hspace*{-2.5mm}
        \caption{\textit{left:} Euclidean distance $f\hspace{-0.15em}\coloneqq\hspace{-0.15em} \vert \cdot\vert \hspace{-0.15em}\in\hspace{-0.15em} \Gamma_0(\mathbb{R})$; \textit{right:} the functional $f\hspace{-0.15em}\coloneqq\hspace{-0.15em} \frac{1}{2}\vert \cdot\vert^2+I_{\smash{K_1^d(0)}}\hspace{-0.15em}\in \hspace{-0.15em}\Gamma_0(\mathbb{R})$;\linebreak each with  derivatives of Moreau envelopes ${\mathrm{D} f^\gamma\colon \hspace{-0.15em} \mathbb{R}\hspace{-0.15em}\to \hspace{-0.15em}\mathbb{R}}$,~${\gamma\hspace{-0.15em}\in\hspace{-0.15em} \{0.1(\text{{\color{blue1}\raisebox{2pt}{\rule{0.5em}{1pt}}}}), 0.25(\text{{\color{blue2}\raisebox{2pt}{\rule{0.5em}{1pt}}}}), 0.5(\text{{\color{blue3}\raisebox{2pt}{\rule{0.5em}{1pt}}}}), 1(\text{{\color{blue4}\raisebox{2pt}{\rule{0.5em}{1pt}}}}), 2(\text{{\color{blue5}\raisebox{2pt}{\rule{0.5em}{1pt}}}}), 4(\text{{\color{blue6}\raisebox{2pt}{\rule{0.5em}{1pt}}}})\}}$, which approximate the subdifferential $\partial f\colon \mathbb{R}\to 2^{\mathbb{R}}\text{ ({\color{black}\raisebox{2pt}{\rule{0.5em}{1pt}}})}$ in the sense of Lemma \ref{lem:resolvent}(\protect\hyperlink{lem:resolvent.ii}{ii}).}\vspace{-1mm}
        \label{fig:nabla_moreau}
    \end{figure}\newpage

    \subsection{Continuous function spaces}\vspace{-0.5mm}

    \hspace{5mm}Throughout the paper, if not otherwise specified, let $\Omega\subseteq \mathbb{R}^d$, $d\in \mathbb{N}$, be a bounded polyhedral Lipschitz\hspace{-0.1mm} domain \hspace{-0.1mm}with \hspace{-0.1mm}(topological) \hspace{-0.1mm}boundary \hspace{-0.1mm}$\partial\Omega$ \hspace{-0.1mm}split \hspace{-0.1mm}into \hspace{-0.1mm}a \hspace{-0.1mm}relatively~\hspace{-0.1mm}open~\hspace{-0.1mm}\mbox{Dirichlet}~\hspace{-0.1mm}part~\hspace{-0.1mm}${\Gamma_D\hspace{-0.175em}\subseteq \hspace{-0.175em}\partial\Omega}$\linebreak and a relatively open Neumann part $\Gamma_N\subseteq \partial\Omega$; \textit{i.e.}, we have that $\overline{\Gamma}_D\cup\overline{\Gamma}_N=\partial\Omega$~and~$\Gamma_D\cap \Gamma_N= \emptyset$.\linebreak 
    For $p\in [1,+\infty]$, the space of $p$-integrable (Lebesgue) measurable functions is denoted~by~$L^p(\Omega)$ and equipped with the norm $\|\cdot\|_{p,\Omega}\coloneqq (\int_{\Omega}{\vert \cdot\vert^p\,\mathrm{d}x})^{\smash{\frac{1}{p}}}$ if $p<+\infty$ and $\|\cdot\|_{\infty,\Omega}\coloneqq \textup{ess\,sup}_{x\in \Omega}{\{\vert (\cdot)(x)\vert\}}$.

    Then, for $p,p'\hspace{-0.15em}\in\hspace{-0.15em} [1,+\infty]$, where $\smash{\frac{1}{p}+\frac{1}{p'}}\hspace{-0.15em}=\hspace{-0.15em}1$ with $\frac{1}{\infty}\hspace{-0.15em}\coloneqq\hspace{-0.15em} 0$, we employ the standard~\mbox{function}~spaces\vspace{-0.5mm}
    \begin{align*}
        W^{1,p}(\Omega)&\coloneqq \big\{v\in L^p(\Omega)\mid \nabla v\in (L^p(\Omega))^d\big\}\,,\\ 
        W^{\smash{p'}}(\operatorname{div};\Omega)&\coloneqq \big\{y\in (L^{\smash{p'}}(\Omega))^d\mid \operatorname{div}y\in L^{\smash{p'}}(\Omega)\big\}\,,
    \end{align*}
    where  
    \hspace{-0.1mm}$\nabla\! \coloneqq\!\smash{(\partial_{x_1},\ldots,\partial_{x_d})}$ \hspace{-0.1mm}and 
    \hspace{-0.1mm}$\operatorname{div}\!\coloneqq\!\smash{\sum_{i=1}^d{\partial_{x_i}}}\!$ 
    \hspace{-0.1mm}initially \hspace{-0.1mm}are \hspace{-0.1mm}to \hspace{-0.1mm}be \hspace{-0.1mm}understood~\hspace{-0.1mm}in~\hspace{-0.1mm}a~\hspace{-0.1mm}\mbox{distributional}~\hspace{-0.1mm}sense.

    If we denote by $\textup{tr}(\cdot)\colon W^{1,p}(\Omega)\to W^{\smash{1-\frac{1}{p}},p}(\partial\Omega)$ and ${\textup{tr}((\cdot)\cdot n)\colon W^{\smash{p'}}(\operatorname{div};\Omega)\to (W^{\smash{1-\frac{1}{p}},p}(\partial\Omega))^*}$ the trace   and the normal trace operator, respectively, where ${n\colon\partial\Omega\to\mathbb{S}^{d-1}\coloneqq\{x\in  \mathbb{R}^d\mid \vert x\vert=1\}}$\linebreak is the outward unit normal vector field to $\Omega$,  then, for every $v\in W^{1,p}(\Omega)$ and $y\in W^{\smash{p'}}(\operatorname{div};\Omega)$, there holds the integration-by-parts formula (\textit{cf}.\ \cite[Sec.\ 4.3, (4.12)]{EG21II})\vspace{-0.5mm} 
	\begin{align}\label{eq:pi_cont}
		%\smash{
        \int_{\Omega}{\nabla v\cdot y\,\mathrm{d}x}+\int_{\Omega}{v\operatorname{div}y\,\mathrm{d}x}=\langle \textup{tr}(y\cdot n),\textup{tr}(v)\rangle_{W^{\smash{1-\frac{1}{p}},p}(\partial\Omega)}\,.%}
	\end{align} 
    Then, we introduce the spaces of functions in $\smash{W^{1,p}(\Omega)}$ with vanishing trace on $\Gamma_D$ and~of~vector fields in $W^{p'}(\operatorname{div};\Omega)$ with vanishing normal trace in $\Gamma_N$ (in~a~\mbox{generalized}~sense),~respectively,~\textit{i.e.},\vspace{-0.5mm}
	\begin{align*} 
		W^{1,p}_D(\Omega)&\coloneqq  \big\{v\in 	W^{1,p}(\Omega) \mid \textup{tr}(v)=0\textup{ a.e.\ on }\Gamma_D\big\}\,,\\
		\smash{W^{p'}_N(\operatorname{div};\Omega)}&\coloneqq  \big\{y\in W^{\smash{p'}}(\operatorname{div};\Omega)\mid \langle\textup{tr}(y\cdot n),\textup{tr}(v)\rangle_{\smash{W^{\smash{1-\frac{1}{p}},p}(\partial\Omega)}} =0\text{ for all  }v\in W^{1,p}_D(\Omega)\big\}\,.
	\end{align*}  

    \subsection{Discrete function spaces}\vspace{-0.5mm}\enlargethispage{2.5mm}

    \hspace{5mm}Throughout the paper, if not otherwise specified, let 
    $\{\mathcal{T}_h\}_{h>0}$ be a family of shape-regular triangulations of 
    $\Omega\subseteq \mathbb{R}^d$, ${d\in\mathbb{N}}$, (\textit{cf}.\  {\cite[Def.\ 11.2]{EG21I}}), where
	$h>0$ refers to the \textit{maximal mesh-size}, \textit{i.e.}, $h 
	\coloneqq \max_{T\in \mathcal{T}_h}{\{h_T\coloneqq \textup{diam}(T)\}}$. The \emph{set of 
    nodes}  (of~the~triangulation~$\mathcal{T}_h$)~is~denoted~by~$\mathcal{N}_h$, the  \emph{set of Dirichlet nodes} by $\mathcal{N}_h^D\coloneqq\{\nu\in \mathcal{N}_h\mid \nu\in \Gamma_D\}$, and the \emph{set of~free~nodes}~by~${\mathcal{N}_h^F\coloneqq \mathcal{N}_h\setminus\mathcal{N}_h^D}$.

    For $k\in \{0,1\}$ and $T\in \mathcal{T}_h$, let $\mathbb{P}^k(T)$ denote the space of polynomials of maximal~degree~$k$~on~$T$. Then, 
	the \emph{space of element-wise constant functions} is defined by
	\begin{align*}
		\smash{\mathcal{L}^0(\mathcal{T}_h) \coloneqq  \big\{v_h\in L^\infty(\Omega)\mid v_h|_T\in\mathbb{P}^0(T)\text{ for all }T\in \mathcal{T}_h\big\}\,,}
	\end{align*}
    while 
	the \emph{space of globally continuous \mbox{element-wise} affine functions} is defined by
    \begin{align*}
		\smash{\mathcal{S}^1(\mathcal{T}_h) \coloneqq  \big\{v_h\in C^0(\overline{\Omega})\mid v_h|_T\in\mathbb{P}^1(T)\text{ for all }T\in \mathcal{T}_h\big\}\,.}
	\end{align*}
    Next, we introduce the 
    \emph{space of globally continuous element-wise affine functions vanishing~on~$\Gamma_D$} 
    \begin{align*}
		\smash{\mathcal{S}^1_D(\mathcal{T}_h) \coloneqq \mathcal{S}^1(\mathcal{T}_h)\cap W^{1,1}_D(\Omega)\,,}
	\end{align*}
    along with the (node-wise) \emph{nodal interpolation operator}  $\smash{I_h^{p1}\colon \smash{(\mathbb{R}\cup\{+\infty\})^{\smash{\mathcal{N}_h}}}\to \mathcal{S}^1_D(\mathcal{T}_h)\cup\{+\infty\}}$, where $\smash{(\mathbb{R}\cup\{+\infty\})^{\smash{\mathcal{N}_h}}}\coloneqq \{v_h \mid v_h\colon\mathcal{N}_h\to \mathbb{R}\cup\{+\infty\}\}$, for every $v_h\in \smash{(\mathbb{R}\cup\{+\infty\})^{\smash{\mathcal{N}_h}}}$ defined by\vspace{-0.5mm}
    \begin{align}\label{def:nodal_interpolation_operator}
        I_h^{p1}v_h\coloneqq\sum_{\nu\in \smash{\mathcal{N}_h^F}}{v_h(\nu)\varphi_\nu}\,,
    \end{align}
    where $(\varphi_\nu)_{\nu\in \mathcal{N}_h}$ denotes the shape basis of $\smash{\mathcal{S}^1(\mathcal{T}_h)}$, which may be used to approximate~the~$L^2(\Omega)$-inner product through the \emph{mass-lumped inner product}, for every $v_h,w_h\hspace{-0.15em}\in\hspace{-0.15em} \smash{(\mathbb{R}\hspace{-0.1em}\cup\hspace{-0.1em}\{+\infty\})^{\smash{\mathcal{N}_h}}}$~\mbox{defined}~by
    \begin{align}\label{def:mass_lumping_inner_product}
        (v_h,w_h)_h\coloneqq \int_{\Omega}{I_h^{p1}\{v_hw_h\}\,\mathrm{d}x}=\sum_{\nu\in  \smash{\mathcal{N}_h^F}}{\biggl(\beta_\nu\coloneqq\smash{\int_{\Omega}{\varphi_\nu\,\mathrm{d}x}}\biggr)v_h(\nu)w_h(\nu)}\in \mathbb{R}\cup\{+\infty\}\,.
    \end{align}
    Furthermore, we canonically extend the nodal interpolation operator \eqref{def:nodal_interpolation_operator} and the mass-lumped inner product \eqref{def:mass_lumping_inner_product} to functions $v,w\in C^0(\overline{\Omega})$ via  $\smash{I_h^{p1}v\coloneqq I_h^{p1}(v|_{\mathcal{N}_h})}$~and~${(v,w)_h\coloneqq (v|_{\mathcal{N}_h},w|_{\mathcal{N}_h})_h}$.% and denote by $\smash{\|\cdot\|_h\coloneqq ((\cdot,\cdot)_h)^{\smash{\frac{1}{2}}}}$ the induced norm.

	\newpage
    \section{Continuous (Fenchel) duality for integral functionals}\label{sec:continuous_duality}

    \hspace{5mm}In this section, we introduce a (Fenchel) duality framework for a broad class of convex integral functionals. For a detailed presentation, we refer to the textbook \cite{EkelandTemam1999} (see also \cite{Rockafellar1968,Rockafellar1971,RockafellarWets1998}). 

    To this end, let $\phi\colon \Omega\times \mathbb{R}^d\to \mathbb{R}\cup\{+\infty\}$ and $\psi\colon \Omega\times \mathbb{R}\to \mathbb{R}\cup\{+\infty\}$ be energy densities~such~that the following assumptions are satisfied.\enlargethispage{12.5mm}
    
    \begin{assumption}[on $\phi$ and $\psi$]\label{ass:energy_densities}
    	The energy densities 
        $\phi\colon \Omega\times \mathbb{R}^d\to \mathbb{R}\cup\{+\infty\}$ and $\psi\colon \Omega\times \mathbb{R}\to \mathbb{R}\cup\{+\infty\}$ are such that for an integrability exponent $p\in [1,\infty)$ the following conditions~are~satisfied:
    	\begin{itemize}[noitemsep,topsep=2pt,leftmargin=!,labelwidth=\widthof{(iii)}]
    		\item[(i)]\hypertarget{ass:energy_densities.i}{} \emph{Convex normal integrands} (\textit{cf}.\ \cite[Sec.\ 2]{Rockafellar1968}): %the energy densities $\phi\colon \Omega\times \mathbb{R}^d\to \mathbb{R}\cup\{+\infty\}$ and $\psi\colon \Omega\times \mathbb{R}\to \mathbb{R}\cup\{+\infty\}$ are convex normal integrands:
    		\begin{itemize}[noitemsep,topsep=2pt,leftmargin=!,labelwidth=\widthof{(i.b)}]
    			\item[(i.a)]\hypertarget{ass:energy_densities.i.a}{} $\phi$ is $ \mathcal{M}(\mathrm{d}x;\Omega)\otimes \mathcal{B}(\mathbb{R}^d)$-measurable\footnote{Here, $\mathcal{M}(\mathrm{d}x;\Omega)$ denotes the \emph{Lebesgue $\sigma$-algebra} on $\Omega$ and $\mathcal{B}(\mathbb{R}^d)$ the \emph{Borel $\sigma$-algebra} on $\mathbb{R}^d$.} and for a.e.\ $x\in \Omega$, we have that $\phi(x,\cdot)\in \Gamma_0(\mathbb{R}^d)$; 
    			\item[(i.b)]\hypertarget{ass:energy_densities.i.b}{} $\psi$ is $\mathcal{M}(\mathrm{d}x;\Omega)\otimes \mathcal{B}(\mathbb{R})$-measurable and for a.e.\ $x\in \Omega$, we have that  $\psi(x,\cdot)\in \Gamma_0(\mathbb{R})$.
    		\end{itemize}
    		\item[(ii)] \hypertarget{ass:energy_densities.ii}{}\emph{Lower compactness property} (\textit{cf}.\ \cite[Sec.\ 2]{Ioffe1977I}):  
    		\begin{itemize}[noitemsep,topsep=2pt,leftmargin=!,labelwidth=\widthof{(ii.b)}]
    			\item[(ii.a)] \hypertarget{ass:energy_densities.ii.a}{} 
                If $\{y_n\}_{n\in \mathbb{N}}\subseteq (L^p(\Omega))^d$ is a sequence
                such that $y_n\to y$ in $(L^p(\Omega))^d$ $(n\to \infty)$ and 
                $\sup_{n\in \mathbb{N}}{\{\int_{\Omega}{\phi(\cdot,y_n)\,\mathrm{d}x}\}}\hspace{-0.1em}<\hspace{-0.1em}+\infty$, then $\{(\phi(\cdot,y_n))_-\}_{n\in \mathbb{N}}\hspace{-0.1em}\subseteq\hspace{-0.1em}  L^1(\Omega)$\footnote{Here, $(\cdot)_-\coloneqq -\min\{0,\cdot\}\colon\mathbb{R}\to\mathbb{R}_{\ge 0}$.} is uniformly~\mbox{integrable}; 
    			\item[(ii.b)] 
                \hypertarget{ass:energy_densities.ii.b}{}
                If $\{v_n\}_{n\in \mathbb{N}}\subseteq L^p(\Omega)$ is a sequence 
                such that $v_n\to v$ in $L^p(\Omega)$ $(n\to \infty)$ and $\sup_{n\in \mathbb{N}}{\{\int_{\Omega}{\psi(\cdot,v_n)\,\mathrm{d}x}\}}\hspace{-0.1em}<\hspace{-0.1em}+\infty$, then $\{(\psi(\cdot,v_n))_-\}_{n\in \mathbb{N}} \hspace{-0.1em}\subseteq \hspace{-0.1em}L^1(\Omega)$~is~uniformly~integrable.
    		\end{itemize}
            \item[(iii)]
            \hypertarget{ass:energy_densities.iii}{}
            \emph{Non-triviality:} there exists a function $u^{\star}\in W^{1,p}_D(\Omega)$ such that $\phi(\cdot,\nabla u^{\star}),\psi(\cdot,u^{\star})\in L^1(\Omega)$.            
    	\end{itemize} 
    \end{assumption}

    Then, if Assumption \ref{ass:energy_densities} is satisfied, the \emph{primal energy functional} $I\colon \smash{W_D^{1,p}(\Omega)}\to \mathbb{R}\cup\{+\infty\}$, for every $v\in \smash{W_D^{1,p}(\Omega)}$, is defined by\vspace*{-0.5mm}
    \begin{align}\label{eq:primal}
        I(v)\coloneqq \int_{\Omega}{\phi(\cdot,\nabla v)\,\mathrm{d}x}+\int_{\Omega}{\psi(\cdot,v)\,\mathrm{d}x}\,,\\[-6mm]\notag
    \end{align}
    and we refer to the problem that seeks to minimize the primal energy functional \eqref{eq:primal}~as~the~\emph{primal problem}. In order to refrain from imposing further restrictive assumptions on the energy densities $\phi\colon \Omega\times \mathbb{R}^d\to \mathbb{R}\cup\{+\infty\}$ and $\psi\colon \Omega\times \mathbb{R}\to \mathbb{R}\cup\{+\infty\}$ (in addition to Assumption \ref{ass:energy_densities}) that would just be sufficient for the existence of a minimizer $u\in W^{1,p}_D(\Omega)$ of the primal~energy~functional~\eqref{eq:primal}, in this section, we assume the   
    existence of such a minimizer and refer to it  as the \emph{primal solution}.

    \begin{remark}[on Assumption \ref{ass:energy_densities}] Assumption \ref{ass:energy_densities} precisely ensures that $I\in \Gamma_0(W^{1,p}_D(\Omega))$:
        \begin{itemize}[noitemsep,topsep=2pt,leftmargin=!,labelwidth=\widthof{(ii)}]
             \item[(i)] Assumption \ref{ass:energy_densities}(\hyperlink{ass:energy_densities.iii}{iii}) ensures that the primal energy functional \eqref{eq:primal} is proper;
            \item[(ii)] Assumption \ref{ass:energy_densities}(\hyperlink{ass:energy_densities.i}{i}),(\hyperlink{ass:energy_densities.ii}{ii}) ensure that the primal energy functional \eqref{eq:primal} is well-defined, convex, and lower semi-continuous (\textit{cf}.\ \cite[Thm.\ 1]{Ioffe1977I}).
           
            %\item[(iii)] In total, Assumption \ref{ass:wellposedness} ensures that $I\in \Gamma_0(W^{1,p}_D(\Omega))$.
        \end{itemize}
    \end{remark}
    
    Denote by $\phi^*\colon \Omega\times \mathbb{R}^d\to \mathbb{R}\cup\{+\infty\}$ and $\psi^*\colon \Omega\times \mathbb{R}\to \mathbb{R}\cup\{+\infty\}$ the (Fenchel) conjugates (with respect to the second argument) of the energy densities $\phi\colon \Omega\times \mathbb{R}^d\to \mathbb{R}\cup\{+\infty\}$ and $\psi\colon \Omega\times \mathbb{R}\to \mathbb{R}\cup\{+\infty\}$, respectively. Then, if the following additional assumption on the validity of a convex conjugation formula is satisfied,
    a (Fenchel) dual problem (in the sense of \cite[Rem.~4.2, p. 60/61]{EkelandTemam1999}) to the minimization of \eqref{eq:primal} likewise is given via the maximization~of~an~integral~functional.
    
    \begin{assumption}[Convex conjugation formula]\label{ass:energy_densities_conjugates}
    	The (Fenchel) conjugate $\psi^*\colon \Omega\times \mathbb{R}\to \mathbb{R}\cup\{+\infty\}$ is  such that\vspace{-1mm} %the following conditions are satsified:
    	%\begin{itemize}[noitemsep,topsep=2pt,leftmargin=!,labelwidth=\widthof{(ii)}]
    		%\item[(i)] \hypertarget{ass:energy_densities_conjugates.i}{} \emph{integrability from below:} %for the (Fenchel) conjugates $\phi\colon \Omega\times \mathbb{R}^d\to \mathbb{R}\cup\{+\infty\}$ and $\psi\colon \Omega\times \mathbb{R}\to \mathbb{R}\cup\{+\infty\}$, the following integrability conditions are satisfied:
    		%\begin{itemize}[noitemsep,topsep=2pt,leftmargin=!,labelwidth=\widthof{(i.b)}]
    		%	\item[(i.a)] \hypertarget{ass:energy_densities_conjugates.i.a}{}For every $y^*\in (L^{p'}(\Omega))^d$, we have that $(\phi^*(\cdot,y^*))^-\in L^1(\Omega)$;
    		%	\item[(i.b)] \hypertarget{ass:energy_densities_conjugates.i.b}{}For every $v^*\in L^{p'}(\Omega)$, we have that $(\psi^*(\cdot,v^*))^-\in L^1(\Omega)$.
    		%\end{itemize} 
    		%\item[(ii)] \hypertarget{ass:energy_densities_conjugates.ii}{} \emph{convex conjugation:} 
            for every $v^*\in (W^{1,p}_D(\Omega))^*$, there holds the \emph{convex conjugation formula}
    		\begin{align}\label{eq:convex_conjugation_formula}
    			\sup_{v\in W^{1,p}_D(\Omega)}{\left\{\langle v^*,v\rangle_{\smash{W^{1,p}_D(\Omega)}}-\int_{\Omega}{\psi(\cdot,v)\,\mathrm{d}x}\right\}}=
                \begin{cases}
                    \displaystyle\int_{\Omega}{\psi^*(\cdot,w^*)\,\mathrm{d}x}&\text{if }
                    \left\{\begin{aligned} 
                     &(\textup{id}_{\smash{W^{1,p}_D(\Omega)}})^*w^*=v^*%\text{ in }(W^{1,p}_D(\Omega))^*
                     \\&\text{for }w^*\in L^{p'}(\Omega)\,,   
                    \end{aligned}\right.\\
                    +\infty&\text{else}\,.
                \end{cases}
                %\int_{\Omega}{\psi^*(\cdot,v^*)\,\mathrm{d}x}\,.
    		\end{align} 
    	%\end{itemize} 
    \end{assumption}

    \begin{remark}[on Assumption \ref{ass:energy_densities_conjugates}]\label{rem:energy_densities_conjugates}
        Assumption \ref{ass:energy_densities} implicitly also ensures that the integral on the right-hand side of the convex conjugation formula \eqref{eq:convex_conjugation_formula} is well-defined:
    	\begin{itemize}[noitemsep,topsep=2pt,leftmargin=!,labelwidth=\widthof{(ii)}]
    		\item[(i)] 
            Due to Assumption \ref{ass:energy_densities}(\hyperlink{ass:energy_densities.i}{i}), the (Fenchel) conjugates $\phi^*\colon \Omega\times \mathbb{R}^d\to \mathbb{R}\cup\{+\infty\}$ and $\psi^*\colon \Omega\times \mathbb{R}\to \mathbb{R}\cup\{+\infty\}$ are convex normal integrands~as~well (\textit{cf}.\ \cite[Lem.\ 5]{Rockafellar1968});%,so that it is not necessary assume the latter;

            \item[(ii)] Due to Assumption \ref{ass:energy_densities}(\hyperlink{ass:energy_densities.i}{i}),(\hyperlink{ass:energy_densities.iii}{iii}), the (Fenchel) conjugates $\phi^*\colon \Omega\times \mathbb{R}^d\to \mathbb{R}\cup\{+\infty\}$~and~$\psi^*\colon \Omega\times \mathbb{R}\to \mathbb{R}\cup\{+\infty\}$ have the lower compactness property~as~well, where~$p$~is~replaced~with~$p'$,
            which follows from the Fenchel--Young inequality (\textit{cf}.\ \cite[Prop.\ 51.2]{Zeidler1985III}).\vspace{-1mm}% (\textit{cf}.\ \cite[Prop.\ 5.1, p.\ 21]{EkelandTemam1999}).\vspace{-1mm}
    	\end{itemize}
    \end{remark}\newpage

    The following lemma provides a general sufficient condition on the energy density $\psi\colon\Omega\times \mathbb{R}\to \mathbb{R}\cup\{+\infty\}$ that ensures the validity of the convex conjugation formula \eqref{eq:convex_conjugation_formula} in Assumption \ref{ass:energy_densities_conjugates}.

    \begin{lemma}[Sufficient condition for Assumption \ref{ass:energy_densities_conjugates}]\label{lem:sufficient_for_conjugation}
        Let $\psi\colon \Omega\times\mathbb{R}\to \mathbb{R}\cup\{+\infty\}$ be an energy density such that the extended functional $\overline{F}\colon L^p(\Omega)\to \mathbb{R}$, for every $v\in L^p(\Omega)$ defined by
        \begin{align*}
            \overline{F}(v)\coloneqq \int_{\Omega}{\psi(\cdot,v)\,\mathrm{d}x}\,,
        \end{align*}
        is well-defined, continuous, and bounded (\textit{i.e.}, maps bounded sets in $L^p(\Omega)$ into bounded~sets~in~$\mathbb{R}$). Then, the convex conjugation formula \eqref{eq:convex_conjugation_formula} in Assumption \ref{ass:energy_densities_conjugates} applies.
    \end{lemma}

    \begin{remark}[on Lemma \ref{lem:sufficient_for_conjugation}]
        By the Krasnoselskii theorem (\textit{cf}.\ \cite[Prop.\ 1.1]{Krasnoselskii1964}), the assumptions of Lemma \ref{lem:sufficient_for_conjugation} are %, \textit{e.g.}, 
        satisfied if $\psi\colon \Omega\times\mathbb{R}\to \mathbb{R}$ is a Carath\'eodory mapping (\textit{i.e.}, $\psi(x,\cdot)\in  C^0(\mathbb{R})$ for a.e.\  $x\in \Omega$ and $\psi(\cdot,s)\in L^0(\Omega)$ for all $s\in \mathbb{R}$)~and,~for~\mbox{every}~${v\in L^p(\Omega)}$,~there~holds~${\psi(\cdot,v)\in  L^1(\Omega)}$.
    \end{remark}

    \begin{proof}[Proof (of Lemma \ref{lem:sufficient_for_conjugation}).] Let $v^*\hspace{-0.1em}\in \hspace{-0.1em}(W^{1,p}_D(\Omega))^*$ be fixed, but arbitrary. Then, 
        we distinguish~\mbox{between} the cases $v^*\in (\textup{id}_{\smash{W^{1,p}_D(\Omega)}})^*(L^{p'}(\Omega))$ and $v^*\notin (\textup{id}_{\smash{W^{1,p}_D(\Omega)}})^*(L^{p'}(\Omega))$:

    \emph{$\bullet$ Case $v^*\in (\textup{id}_{\smash{W^{1,p}_D(\Omega)}})^*(L^{p'}(\Omega))$.} In this case, there exists a function $w^*\in L^{p'}(\Omega)$~such~that $v^*\hspace{-0.15em}=\hspace{-0.15em}(\textup{id}_{\smash{W^{1,p}_D(\Omega)}})^*w^*$ \hspace{-0.1mm}in \hspace{-0.1mm}$(W^{1,p}_D(\Omega))^*$. \hspace{-0.1mm}Then, \hspace{-0.1mm}by \hspace{-0.1mm}the \hspace{-0.1mm}density \hspace{-0.1mm}of \hspace{-0.1mm}$W^{1,p}_D(\Omega)$ \hspace{-0.1mm}in \hspace{-0.1mm}$L^p(\Omega)$ \hspace{-0.1mm}and %together with 
   \hspace{-0.1mm}the~\hspace{-0.1mm}\mbox{continuity}~\hspace{-0.1mm}of $\overline{F}\colon L^p(\Omega)\to \mathbb{R}$, the convex conjugation formula for integral functionals (\textit{cf}.\ \cite[Thm.\ 2]{Rockafellar1968})~yields
    \begin{align*}
            \sup_{v\in W^{1,p}_D(\Omega)}{\left\{\langle v^*,v\rangle_{\smash{W^{1,p}_D(\Omega)}}-\int_{\Omega}{\psi(\cdot,v)\,\mathrm{d}x}\right\}}&=\sup_{v\in L^p(\Omega)}{\left\{\int_{\Omega}{w^*v\,\mathrm{d}x}-\int_{\Omega}{\psi(\cdot,v)\,\mathrm{d}x}\right\}}
            \\[-0.5mm]&=\int_{\Omega}{\psi^*(\cdot,w^*)\,\mathrm{d}x}\,,
        \end{align*}
        which is the claimed convex conjugation formula \eqref{eq:convex_conjugation_formula} in the case $\smash{v^*\in (\textup{id}_{\smash{W^{1,p}_D(\Omega)}})^*(L^{p'}(\Omega))}$.

    \emph{$\bullet$ Case $v^*\hspace{-0.15em}\notin\hspace{-0.15em} (\textup{id}_{\smash{W^{1,p}_D(\Omega)}})^*(L^{p'}(\Omega))$.} In this case, there exists a sequence $\{v_n\}_{n\in \mathbb{N}}\hspace{-0.15em}\subseteq \hspace{-0.15em}W^{1,p}_D(\Omega)$~boun\-ded in $L^p(\Omega)$ such that $\langle v^*,v_n\rangle_{\smash{W^{1,p}_D(\Omega)}}\to +\infty$ $(n\to \infty)$. By the~boundedness~of~$\overline{F}\colon L^p(\Omega)\to \mathbb{R}$, the sequence $\{\overline{F}(v_n)\}_{n\in \mathbb{N}}\subseteq \mathbb{R}$ is bounded as well, so that
   %     \begin{align*}
   %         \langle v^*,v_n\rangle_{\smash{W^{1,p}_D(\Omega)}}\to +\infty\quad (n\to \infty)\,,\\
   %        \sup_{n\in \mathbb{N}}{\big\{\|v_n\|_{p,\Omega}\big\}}<+\infty\,.
   %     \end{align*}
   % Since $\overline{F}\colon L^p(\Omega)\to \mathbb{R}$ maps bounded sets into bounded sets, we infer that
   % \begin{align*}
   %          \sup_{n\in \mathbb{N}}{\big\{\overline{F}(v_n)\big\}}<+\infty\,.
   %     \end{align*}
   %     In summary, we have that
        \begin{align*}
            \sup_{v\in W^{1,p}_D(\Omega)}{\left\{\langle v^*,v\rangle_{\smash{W^{1,p}_D(\Omega)}}-\int_{\Omega}{\psi(\cdot,v)\,\mathrm{d}x}\right\}}\ge \langle v^*,v_n\rangle_{\smash{W^{1,p}_D(\Omega)}}-\overline{F}(v_n)\to +\infty\,,
        \end{align*}
        which is the claimed convex conjugation formula \eqref{eq:convex_conjugation_formula} in the case $\smash{v^*\notin (\textup{id}_{\smash{W^{1,p}_D(\Omega)}})^*(L^{p'}(\Omega))}$.
    \end{proof}

	Then, if  Assumptions \ref{ass:energy_densities} and  \ref{ass:energy_densities_conjugates} are satisfied, a (Fenchel) dual problem (in the sense of \cite[Rem. 4.2, p. 60/61]{EkelandTemam1999}) to the minimization of \eqref{eq:primal} is given via the maximization of the \emph{dual energy functional} $D\colon W^{p'}_N(\operatorname{div};\Omega)\to \mathbb{R}\cup\{-\infty\}$, for every $y\in W^{p'}_N(\operatorname{div};\Omega)$ defined by\vspace{-0.5mm}
	\begin{align}\label{eq:dual}
		D(y)\coloneqq -\int_{\Omega}{\phi^*(\cdot,y)\,\mathrm{d}x}-\int_{\Omega}{\psi^*(\cdot,\operatorname{div}y)\,\mathrm{d}x}\,,
	\end{align}
	which is established along with the existence of a maximizer $z\in \smash{W_N^{p'}(\operatorname{div};\Omega)}$, called \emph{dual solution}, a strong duality relation, and primal optimality inclusions in the following proposition. 
	
	\begin{proposition}[Identification of dual problem, strong duality, and primal optimality inclusions]\label{prop:duality}
	 		Let Assumptions \ref{ass:energy_densities} and \ref{ass:energy_densities_conjugates} be satisfied. Then, the following statements apply:
	 		\begin{itemize}[noitemsep,topsep=2pt,leftmargin=!,labelwidth=\widthof{(ii)}]
	 			\item[(i)] \hypertarget{prop:duality.i}{} A (Fenchel) dual problem to the minimization of \eqref{eq:primal} is given via the maximization~of~\eqref{eq:dual};
	 			\item[(ii)] \hypertarget{prop:duality.ii}{} If, in addition, $(y\mapsto \int_{\Omega}{\phi(\cdot,y)\,\mathrm{d}x})\colon (L^p(\Omega))^d\to \mathbb{R}\cup\{+\infty\}$ is continuous at $\nabla u^{\star}\in (L^p(\Omega))^d$, then a dual solution exists, a \emph{strong duality relation} applies, \textit{i.e.}, we have that $I(u)=D(z)$, which is precisely equivalent to the \emph{primal  optimality inclusions}
	 			\begin{subequations}\label{prop:duality.2} 
	 			\begin{alignat}{2}
	 				z&\in \partial_t \phi(\cdot,\nabla u)&&\quad\text{ a.e.\ in }\Omega\,,\label{prop:duality.2.1}\\
	 				\operatorname{div}z&\in \partial_s\psi(\cdot,u )&&\quad\text{ a.e.\ in }\Omega\,,\label{prop:duality.2.2} 
	 			\end{alignat}
	 		\end{subequations}
            where the subdifferentials $\partial_t$, $\partial_s$ are formed with respect to the second arguments, respectively.
	 		\end{itemize}
	\end{proposition}
	
	\begin{proof}
		\emph{ad (\hyperlink{prop:duality.i}{i}).} To begin with, we introduce the functionals $G\hspace{-0.1em}\in\hspace{-0.1em} \smash{\Gamma_0((L^p(\Omega))^d)}$ and~$\smash{F\hspace{-0.1em}\in\hspace{-0.1em} \Gamma_0(W^{1,p}_D(\Omega))}$, for every $y\in (L^p(\Omega))^d$ and $v\in \smash{W^{1,p}_D(\Omega)}$, respectively, defined by\vspace{-0.5mm}
		\begin{align}\label{prop:duality.0.1}
			G(y)\coloneqq \int_{\Omega}{\phi(\cdot,y)\,\mathrm{d}x}\,,\qquad F(v)\coloneqq \int_{\Omega}{\psi(\cdot,v)\,\mathrm{d}x}\,.\\[-6mm]\notag
		\end{align}
		Then, \hspace{-0.1mm}according \hspace{-0.1mm}to \hspace{-0.1mm}\cite[Rem.\ \hspace{-0.1mm}4.2, \hspace{-0.1mm}p.\ \hspace{-0.1mm}60/61]{EkelandTemam1999}, \hspace{-0.1mm}the \hspace{-0.1mm}(Fenchel) \hspace{-0.1mm}dual \hspace{-0.1mm}problem \hspace{-0.1mm}to \hspace{-0.1mm}the \hspace{-0.1mm}minimization~\hspace{-0.1mm}of~\hspace{-0.1mm}\eqref{eq:primal} is given via the maximization of $D\colon  (L^{p'}(\Omega))^d\to \mathbb{R}\cup\{-\infty\}$,~for~every~${y\in (L^{p'}(\Omega))^d}$~given~via\vspace{-0.5mm}
		\begin{align}\label{prop:duality.3}
			D(y)\coloneqq-G^*(y)-F^*(\nabla^* y)\,,\\[-6mm]\notag
		\end{align}
		where $\nabla^*\colon \smash{(L^{p'}(\Omega))^d}\to (W^{1,p}_D(\Omega))^*$ is the adjoint operator to the gradient operator $\nabla\colon W^{1,p}_D(\Omega)\to (L^p(\Omega))^d$. Thus, it only remains to establish that the claimed integral representation \eqref{eq:dual} applies:
        \begin{itemize}[noitemsep,topsep=2pt,leftmargin=!,labelwidth=\widthof{$\bullet$}]
            \item[$\bullet$] Due to Assumption \ref{ass:energy_densities}(\hyperlink{ass:energy_densities.i}{i}.\hyperlink{ass:energy_densities.i.a}{a}),(\hyperlink{ass:energy_densities.ii}{ii}.\hyperlink{ass:energy_densities.ii.a}{a}) and Remark \ref{rem:energy_densities_conjugates}, the convex conjugation formula for integral functionals (\textit{cf}.\ \cite[Thm.\ 2]{Rockafellar1968}) is applicable and, for every $y\in (L^{p'}(\Omega))^d$,  yields\vspace{-0.75mm}
            \begin{align}\label{prop:duality.4}
                G^*(y)\coloneqq \int_{\Omega}{\phi^*(\cdot,y)\,\mathrm{d}x}\,.\\[-6mm]\notag
            \end{align}

             \item[$\bullet$] Due to Assumption \ref{ass:energy_densities_conjugates}, for every $y\in \smash{(L^{p'}(\Omega))^d}$, we have that\vspace{-0.5mm}
             \begin{align}\label{prop:duality.5}
                \begin{aligned}
                 F^*(-\nabla^* y)&=\sup_{v\in W^{1,p}_D(\Omega)}{\left\{\langle -\nabla^* y,v\rangle_{\smash{W^{1,p}_D(\Omega)}}-\int_{\Omega}{\psi(\cdot,v)\,\mathrm{d}x}\right\}}
                 \\&=\begin{cases}
                    \displaystyle\int_{\Omega}{\psi^*(\cdot,w^*)\,\mathrm{d}x}&\text{ if }
                    \left\{\begin{aligned} 
                     &(\textup{id}_{\smash{W^{1,p}_D(\Omega)}})^*w^*=-\nabla^* y%\text{ in }(W^{1,p}_D(\Omega))^*
                     \\&\text{for }w^*\in L^{p'}(\Omega)\,,   
                    \end{aligned}\right.\\
                    +\infty&\text{ else}\,,
                \end{cases}
                \\&=\begin{cases}
                    \displaystyle\int_{\Omega}{\psi^*(\cdot,\operatorname{div}y)\,\mathrm{d}x}&\text{ if }y\in W^{p'}_N(\operatorname{div};\Omega)\,,\\
                    +\infty&\text{ else}\,,
                \end{cases}
                \end{aligned}\\[-6mm]\notag
             \end{align}
            so that $\textup{dom}(-D)\subseteq \smash{W^{p'}_N(\operatorname{div};\Omega)}$, where we used in the last step that $(\textup{id}_{\smash{W^{1,p}_D(\Omega)}})^*w^*=-\nabla^* y$ in $(W^{1,p}_D(\Omega))^*$ for some $w^*\in L^{p'}(\Omega)$ is equivalent to that $y\in \smash{W^{p'}_N(\operatorname{div};\Omega)}$~with~${\operatorname{div}y=w^*}$~a.e.~in~$\Omega$.\vspace{0.5mm}
        \end{itemize}
         Using  \eqref{prop:duality.4} and \eqref{prop:duality.5} in \eqref{prop:duality.3},  we conclude that the claimed integral representation \eqref{eq:dual} applies.\vspace{1mm}
		
		\emph{ad (\hyperlink{prop:duality.ii}{ii}).} By Assumption \ref{ass:energy_densities}(\hyperlink{ass:energy_densities.iii}{iii}), we have that $G(\nabla u^{\star})<+\infty$, $F(u^{\star})<+\infty$, and, by the additional assumption of this proposition, $G\colon (L^p(\Omega))^d\to \mathbb{R}\cup\{+\infty\}$ is continuous~at~${\nabla u^{\star}\in (L^p(\Omega))^d}$. Therefore, the Fenchel duality theorem (\textit{cf}.\ \cite[Rem.\ 4.1, eqs.\ (4.21), p.\ 61]{EkelandTemam1999}) yields the existence of a dual solution $z\in \textup{dom}(-D)\subseteq \smash{W^{p'}_N(\operatorname{div};\Omega)}$ as well as that a strong duality relation applies. %, \textit{i.e.}. we have that $I(u)=D(z)$. 
        According to \cite[Rem.\ 4.1, eqs.\ (4.22),(4.23), p.\ 61]{EkelandTemam1999}, the strong duality relation is equivalent to\vspace{-0.5mm} %that
        \begin{subequations}\label{prop:duality.6}
        \begin{align}\label{prop:duality.6.1}
            \int_{\Omega}{(\phi^*(\cdot,z)-z\cdot\nabla u+\phi(\cdot,\nabla u))\,\mathrm{d}x}&=0\,,\\
            \int_{\Omega}{(\psi^*(\cdot,\operatorname{div}z)-\operatorname{div}z\, u+\psi(\cdot, u))\,\mathrm{d}x}&=0\,.\label{prop:duality.6.2}
        \end{align}\\[-3mm]\notag
        \end{subequations}
        By the Fenchel--Young inequality (\textit{cf}.\ \cite[Prop.\ 51.2]{Zeidler1985III}), the integrands of the integrals in \eqref{prop:duality.6.1} and \eqref{prop:duality.6.2} are point-wise a.e.\ non-negative, so that \eqref{prop:duality.6} is precisely equivalent to\vspace{-0.5mm}
        \begin{subequations}\label{prop:duality.7}
        \begin{alignat}{2}\label{prop:duality.7.1}
           \phi^*(\cdot,z)-z\cdot\nabla u+\phi(\cdot,\nabla u)&=0&&\quad\text{ a.e.\ in }\Omega\,,\\
           \psi^*(\cdot,\operatorname{div}z)-\operatorname{div}z\, u+\psi(\cdot, u)&=0&&\quad\text{ a.e.\ in }\Omega\,,\label{prop:duality.7.2}
        \end{alignat}\\[-5mm]\notag
        \end{subequations}
        which, by the equality condition in the Fenchel--Young inequality (\textit{cf}.\ \cite[Prop.\ 51.2]{Zeidler1985III}),~in~turn, is equivalent to the claimed primal optimality inclusions \eqref{prop:duality.2}.\enlargethispage{5mm}\vspace{-0.25mm}
	\end{proof}

    \begin{remark}[Proximity operators of \eqref{prop:duality.0.1}]\label{rem:proximity_integral_functionals} If Assumptions \ref{ass:energy_densities} and \ref{ass:energy_densities_conjugates} are satisfied and $p=2$, due \hspace{-0.1mm}to \hspace{-0.1mm}\cite[Thm.\ \hspace{-0.1mm}14.60]{RockafellarWets1998}, \hspace{-0.1mm}for \hspace{-0.1mm}the \hspace{-0.1mm}functionals \hspace{-0.1mm}\eqref{prop:duality.0.1}, \hspace{-0.1mm}for \hspace{-0.1mm}every \hspace{-0.1mm}$y\hspace{-0.15em}\in\hspace{-0.15em} (L^2(\Omega))^d$~\hspace{-0.1mm}and~\hspace{-0.1mm}${v\hspace{-0.15em}\in \hspace{-0.15em}L^2(\Omega)}$,~\hspace{-0.1mm}we~\hspace{-0.1mm}have~\hspace{-0.1mm}that\vspace{-4.5mm}
    \begin{subequations} 
    \begin{alignat}{2}
        \operatorname{prox}_{G}(y)(x)&=\operatorname{prox}_{\phi(x,\cdot)}(y(x))&&\quad\text{ for a.e.\ }x\in \Omega\,,\\
        \operatorname{prox}_{F}(v)(x)&=\operatorname{prox}_{\psi(x,\cdot)}(v(x))&&\quad\text{ for a.e.\ }x\in \Omega\,.
    \end{alignat} 
    \end{subequations}
    \end{remark}
	\newpage

    \section{Discrete duality for first-order continuous Lagrange approximations}\label{sec:discrete_duality}

    \hspace{5mm}In this section, we transfer the convex duality framework from Section \ref{sec:continuous_duality} to a discrete setting; more precisely, to a first-order continuous Lagrange approximation of the primal problem \eqref{eq:primal} and  a zeroth-order discontinuous Lagrange approximation~of~the~dual~\mbox{problem}~\eqref{eq:dual}.~In~this~\mbox{connection}, for the rest of the paper, we %employ 
    consider the following two Hilbert spaces
    \begin{align*}
        V_h\coloneqq \mathcal{S}^1_D(\mathcal{T}_h)\,,\qquad Y_h\coloneqq (\mathcal{L}^0(\mathcal{T}_h))^d\,,
    \end{align*} 
    where \hspace{-0.175mm}$(\cdot,\hspace{-0.075em}\cdot)_{V_h}\hspace{-0.2em}\in\hspace{-0.175em} \{(\cdot,\hspace{-0.075em}\cdot)_{L^2(\Omega)},(\cdot,\hspace{-0.075em}\cdot)_h\}$ \hspace{-0.175mm}and \hspace{-0.175mm}$(\cdot,\hspace{-0.075em}\cdot)_{Y_h}\hspace{-0.175em}\coloneqq\hspace{-0.175em}(\cdot,\hspace{-0.075em}\cdot)_{(L^2(\Omega))^d}$; \hspace{-0.175mm}\textit{i.e.}, \hspace{-0.175mm}$\|\cdot\|_{V_h}\hspace{-0.175em}\coloneqq\hspace{-0.175em} \smash{(\cdot,\hspace{-0.075em}\cdot)_{V_h}^{1/2}}$~\hspace{-0.175mm}and~\hspace{-0.175mm}${\|\hspace{-0.175em}\cdot\hspace{-0.175em}\|_{Y_h}\hspace{-0.15em}\coloneqq\hspace{-0.15em} \smash{(\cdot,\hspace{-0.075em}\cdot)_{Y_h}^{1/2}}}\hspace{-0.15em}$.
    
    Whereas the canonical discretization of the energy density $\phi\colon \Omega\times \mathbb{R}^d\to \mathbb{R}\cup\{+\infty\}$ is an element-wise constant one, the discretization of the energy density $\psi\colon \Omega\times \mathbb{R}\to \mathbb{R}\cup\{+\infty\}$ is not uniquely prescribed and, hence,  permits %substantially greater 
    more flexibility. The main requirements on the discretization of the energy density $\psi\colon \Omega\times \mathbb{R}\to \mathbb{R}\cup\{+\infty\}$ are that a discrete convex conjugation formula (\textit{cf}.~Lemma \ref{lem:discrete_convex_conjugation}) applies and that its subdifferential maps $V_h$~into~its~power~set~(\textit{cf}.~Remark~\ref{lem:subdiff_in_Vh}).
In this section, we consider two discretization strategies for the energy density~$\psi\colon \Omega\times \mathbb{R}\to \mathbb{R}\cup\{+\infty\}$ that satisfy these criteria: the first employs mass lumping and is applicable to a broad~class~of energy densities~$\psi_h$;~the~second avoids mass lumping but is, thus, restricted to at most~quadratic~$\psi_h$.

 In this connection, for the rest of the paper, in order to consider both discretization strategies simultaneously, we will employ a discrete operator $I_{V_h}\colon \psi_h(\cdot,V_h)\to \{L^1(\Omega),+\infty\}$~and a discrete inner product $(\cdot,\cdot)_{V_h}$ associated with $V_h$, which, depending on the employed discretization~strategy, serve as placeholders for the following two concrete choices:

 \begin{itemize}[noitemsep,topsep=2pt,leftmargin=!,labelwidth=\widthof{$\bullet$}]
     \item[$\bullet$] \emph{Mass lumping (ML):}\hypertarget{ML}{} In this case, we employ $I_{V_h}=I_h^{p1}$ and $(\cdot,\cdot)_{V_h}=(\cdot,\cdot)_h$;
     \item[$\bullet$] \emph{No mass lumping (NML):}\hypertarget{NML}{} In this case, we employ $I_{V_h}=\operatorname{id}_{\mathbb{R}}$ and $(\cdot,\cdot)_{V_h}=(\cdot,\cdot)_{L^2(\Omega)}$.
 \end{itemize}

    %In doing so, we want to present two ways for how to discretize the lower-order part in the primal energy functional \eqref{eq:primal}.
    
    Then, depending on whether mass lumping is employed or not, let $\phi_h\colon \Omega\times \mathbb{R}^d\to \mathbb{R}\cup\{+\infty\}$ and $\psi_h\colon \Omega\times \mathbb{R}\to \mathbb{R}\cup\{+\infty\}$ be approximations of $\phi\colon \Omega\times \mathbb{R}^d\to \mathbb{R}\cup\{+\infty\}$ and $\psi\colon \Omega\times \mathbb{R}\to \mathbb{R}\cup\{+\infty\}$, respectively, such that 
    the following  assumptions are satisfied.
    \begin{assumption}[on $\phi_h$ and $\psi_h$]\label{ass:energy_densities_discrete}
    	The discrete energy densities 
        $\phi_h\colon \Omega\times \mathbb{R}^d\to \mathbb{R}\cup\{+\infty\}$ and $\psi_h\colon \Omega\times \mathbb{R}\to \mathbb{R}\cup\{+\infty\}$ are such that the following conditions are satisfied:
    	\begin{itemize}[noitemsep,topsep=2pt,leftmargin=!,labelwidth=\widthof{(iii)}]
    		\item[(i)]\hypertarget{ass:energy_densities_discrete.i}{} \emph{Assumptions on $\phi_h$:} 
            \begin{itemize}[noitemsep,topsep=2pt,leftmargin=!,labelwidth=\widthof{(i.b)}]
                \item[(i.a)] \hypertarget{ass:energy_densities_discrete.i.a}{} For every $t\in \mathbb{R}^d$, we have that $\phi_h(\cdot,t)\in \mathcal{L}^0(\mathcal{T}_h)$;
                \item[(i.b)] \hypertarget{ass:energy_densities_discrete.i.b}{} For a.e.\ $x\in \Omega$, we have that $\phi_h(x,\cdot)\in \Gamma_0(\mathbb{R}^d)$.
            \end{itemize}    
    		\item[(ii)] \hypertarget{ass:energy_densities_discrete.i}{}\emph{Assumptions on $\psi_h$:}  
    		\begin{itemize}[noitemsep,topsep=2pt,leftmargin=!,labelwidth=\widthof{(ii.b)}]
    			\item[(ii.a)] \hypertarget{ass:energy_densities_discrete.ii.a}{} \emph{Mass lumping:} If (\hyperlink{ML}{ML}) applies, 
                for every $\nu\in \mathcal{N}_h^F$, we have that
                $\psi_h(\nu,\cdot)\in \Gamma_0(\mathbb{R})$;
    			\item[(ii.b)] 
                \hypertarget{ass:energy_densities_discrete.ii.b}{}
              \emph{No mass lumping:} If (\hyperlink{NML}{NML}) applies, there exists a constant $a_h\hspace{-0.1em}\ge \hspace{-0.1em}0$~as~well~as~functions $b_h\in V_h$
                and $c_h\in L^1(\Omega)$ such that for a.e.\ $x\in \Omega$ and every $s\in \mathbb{R}$, we have that
                \begin{align*}
                    \psi_h(x,s)=a_hs^2+b_h(x)s +c_h(x)\,.
                \end{align*}
    		\end{itemize}   
            \item[(iii)]
            \hypertarget{ass:energy_densities_discrete.iii}{}
            \emph{Non-triviality:} there exists a function $u_h^{\star}\in V_h$ such that $\phi_h(\cdot,\nabla u_h^{\star}),I_{V_h}\{\psi_h(\cdot,u_h^{\star})\}\in L^1(\Omega)$.~~        
    	\end{itemize} 
    \end{assumption}

    Then, if Assumption \ref{ass:energy_densities_discrete} is satisfied, the \emph{discrete primal energy functional} $I_h\colon V_h\to \mathbb{R}\cup\{+\infty\}$, for every $v_h\in V_h$, is defined by
    \begin{align}\label{eq:primal_discrete}
        I_h(v_h)\coloneqq \int_{\Omega}{\phi_h(\cdot,\nabla  v_h)\,\mathrm{d}x}+\int_{\Omega}{I_{V_h}\{\psi_h(\cdot,v_h)\}\,\mathrm{d}x}\,,
    \end{align}
    and we refer to the problem that seeks to minimize the discrete primal energy functional~\eqref{eq:primal_discrete}~as~the \emph{discrete primal problem}. In order to refrain from imposing further restrictive assumptions on the discrete energy densities $\phi_h\colon \Omega\times \mathbb{R}^d\to \mathbb{R}\cup\{+\infty\}$ and $\psi_h\colon \Omega\times \mathbb{R}\to \mathbb{R}\cup\{+\infty\}$ (in~addition~to Assumption \ref{ass:energy_densities_discrete}) that would just be sufficient for the existence of a minimizer $u_h\in V_h$~of~the discrete primal energy functional \eqref{eq:primal_discrete}, in this section, we assume the existence of such~a~minimizer~and refer to it as the \emph{discrete primal solution}.

    \begin{remark}\label{lem:subdiff_in_Vh}
        Assumption \ref{ass:energy_densities_discrete}(\hyperlink{ass:energy_densities_discrete.ii}{ii}) also ensures that the inclusion $ I_{V_h}\{\partial_s \psi_h(\cdot,V_h)\}\subseteq 2^{V_h}$ applies; more precisely, for every $v_h\in V_h$, there holds $I_{V_h}\{\partial_s \psi_h(\cdot,v_h)\}\subseteq V_h$. 
    \end{remark}

    \hspace{-1mm}In \hspace{-0.175mm}order \hspace{-0.175mm}to \hspace{-0.175mm}identify \hspace{-0.175mm}a \hspace{-0.175mm}(Fenchel) \hspace{-0.175mm}dual \hspace{-0.175mm}problem \hspace{-0.175mm}(in \hspace{-0.175mm}the \hspace{-0.175mm}sense \hspace{-0.175mm}of \cite[Rem.\ \hspace{-0.175mm}4.2, \hspace{-0.175mm}p.\ \hspace{-0.175mm}60/61]{EkelandTemam1999})~\hspace{-0.175mm}to~\hspace{-0.175mm}the~\hspace{-0.175mm}minimi\-zation of \eqref{eq:primal_discrete}, we introduce the \textit{discrete divergence operator (with~respect~to $(\cdot,\cdot)_{V_h}$)} ${\operatorname{div}_h\colon \hspace{-0.1em} Y_h\hspace{-0.1em}\to \hspace{-0.1em}V_h}$, for every $y_h\in Y_h$ and $v_h\in V_h$ defined by\vspace{-0.25mm}
    \begin{align}\label{def:discrete_divergence}
        (\operatorname{div}_hy_h,v_h)_{V_h}\coloneqq -\int_{\Omega}{y_h\cdot\nabla v_h\,\mathrm{d}x}\,.\\[-6.25mm]\notag
    \end{align} 
    Then, if Assumption \ref{ass:energy_densities_discrete}(\hyperlink{ass:energy_densities_discrete.ii.b}{ii}) is satisfied, a discrete counterpart of the convex conjugation~formula \eqref{prop:duality.5} in Assumption \ref{ass:energy_densities_conjugates} applies.\enlargethispage{1mm}

    \begin{lemma}[Discrete convex conjugation formula]\label{lem:discrete_convex_conjugation}
    Let Assumption \ref{ass:energy_densities_discrete}(\hyperlink{ass:energy_densities_discrete.ii.b}{ii}) be satisfied. Then, for every $y_h\in Y_h$, there holds the \emph{discrete convex conjugation formula}\vspace{-0.5mm}
    \begin{align}\label{lem:discrete_convex_conjugation.1}
        \sup_{v_h\in V_h}{\left\{\int_{\Omega}{I_{V_h}\{\operatorname{div}_hy_hv_h-\psi_h(\cdot,v_h)\}\,\mathrm{d}x}\right\}}=\int_{\Omega}{I_{V_h}\{\psi_h^*(\cdot,\operatorname{div}_hy_h)\}\,\mathrm{d}x}\,.\\[-6mm]\notag
    \end{align}
    \end{lemma}

    \begin{proof}\let\qed\relax
    We distinguish the cases with and without mass lumping (\textit{i.e.}, Assumption \ref{ass:energy_densities_discrete}(\hyperlink{ass:energy_densities_discrete.ii.a}{ii.a}/\hyperlink{ass:energy_densities_discrete.ii.b}{b})):

        $\bullet$ \emph{Case 1: (\hyperlink{ML}{ML}).} In this case, for every $y_h\in Y_h$, we have that\vspace{-1mm}
        \begin{align*}
        \sup_{v_h\in V_h}{\left\{\int_{\Omega}{I_h^{p1}{\{\operatorname{div}_hy_hv_h-\psi_h(\cdot,v_h)\}}\,\mathrm{d}x}\right\}}
            &=\sup_{v_h\in V_h}{\left\{\sum_{\nu\in \mathcal{N}_h^F}{\beta_\nu(\operatorname{div}_hy_h(\nu)v_h(\nu)-\psi_h(\nu,v_h(\nu)))}\right\}}
            \\&=\sum_{\nu\in \smash{\mathcal{N}_h^F}}{\beta_\nu\sup_{s\in \mathbb{R}}{\{\operatorname{div}_hy_h(\nu)s-\psi_h(\nu,s)\}}}
            \\&=\sum_{\nu\in \smash{\mathcal{N}_h^F}}{\beta_\nu\psi_h^*(\nu,\operatorname{div}_hy_h(\nu))}
            \\&=\int_{\Omega}{I_h^{p1}{\{\psi_h^*(\cdot,\operatorname{div}_h y_h)\}}\,\mathrm{d}x}\,,\\[-5.5mm]\notag
        \end{align*}
        which is the claimed discrete convex conjugation formula  \eqref{lem:discrete_convex_conjugation.1} in the case with mass lumping. % (\textit{cf}.\ Assumption \ref{ass:energy_densities_discrete}(\hyperlink{ass:energy_densities_discrete.i.b}{i})).

          $\bullet$ \emph{Case 2: (\hyperlink{NML}{NML}).} We distinguish the cases $a_h=0$ and $a_h> 0$: 

          $\bullet$ \emph{Case 2.1: ($a_h=0$).} In this case, the (Fenchel) conjugate (with respect to~the~second~argument) $\psi_h^*\colon \Omega\times\mathbb{R}\to \mathbb{R}\cup\{+\infty\}$, for a.e.\ $x\in \Omega$ and every $s^*\in \mathbb{R}$, is given via\vspace{-0.5mm}
        \begin{align*}
            \psi_h^*(x,s^*)= \begin{cases}
                -c_h(x)&\text{ if }s^*=b_h(x)\,,\\
                +\infty& \text{ else}\,.
            \end{cases}\\[-6mm]\notag
        \end{align*}
        In particular, for every $y_h\in Y_h$, due to $\operatorname{div}_hy_h-b_h\in V_h$, we have that\vspace{-0.5mm}
        \begin{align*}
            \sup_{v_h\in V_h}{\left\{\int_{\Omega}{(\operatorname{div}_hy_hv_h-\psi_h(\cdot,v_h))\,\mathrm{d}x}\right\}}&=\sup_{v_h\in V_h}{\left\{\int_{\Omega}{(\operatorname{div}_hy_h-b_h)v_h\,\mathrm{d}x}\right\}}-\int_{\Omega}{c_h\,\mathrm{d}x}
            \\&=\begin{cases}
                \displaystyle-\int_{\Omega}{c_h\,\mathrm{d}x}&\text{ if }\operatorname{div}_hy_h=b_h\text{ in }\Omega\,,\\
                +\infty& \text{ else}\,,
            \end{cases}
            \\&=\int_{\Omega}{\psi_h^*(\cdot,\operatorname{div}_hy_h)\,\mathrm{d}x}\,.\\[-6mm]\notag
        \end{align*}

        $\bullet$ \emph{Case 2.2: ($a_h>0$).} In this case, the (Fenchel) conjugate (with respect to~the~second~argument) $\psi_h^*\colon \Omega\times\mathbb{R}\to \mathbb{R}\cup\{+\infty\}$, for a.e.\ $x\in \Omega$ and every $s^*\in \mathbb{R}$, is given via\vspace{-0.5mm}
        \begin{align*}
            \smash{\psi_h^*(x,s^*)= \tfrac{1}{4a_h}(s^*-b_h(x))^2-c_h(x)}\,.\\[-6mm]\notag
        \end{align*}
        In particular, for every $y_h\in Y_h$, due to $\operatorname{div}_hy_h-b_h\in V_h$, we have that 
        \begin{align*}
            \sup_{v_h\in V_h}{\left\{\int_{\Omega}{(\operatorname{div}_hy_hv_h-\psi_h(\cdot,v_h))\,\mathrm{d}x}\right\}}&=\sup_{v_h\in V_h}{\left\{\int_{\Omega}{(\operatorname{div}_hy_h-a_hv_h-b_h)v_h\,\mathrm{d}x}\right\}}
            -\int_{\Omega}{c_h\,\mathrm{d}x}\,.
        \end{align*}\newpage
        \noindent Using that, by the direct method in the calculus of variations, the supremum is attained~at~some $v_h^*\hspace{-0.1em}\coloneqq \hspace{-0.1em}v_h^*(y_h)\hspace{-0.15em}\in\hspace{-0.15em} V_h$, \hspace{-0.1mm}which, \hspace{-0.1mm}by \hspace{-0.1mm}the \hspace{-0.1mm}Fr\'echet \hspace{-0.1mm}differentiability \hspace{-0.1mm}of \hspace{-0.1mm}the \hspace{-0.1mm}functional, \hspace{-0.1mm}for \hspace{-0.1mm}every~\hspace{-0.1mm}${v_h\hspace{-0.15em}\in\hspace{-0.15em} V_h}$,~\hspace{-0.1mm}\mbox{satisfies} $\int_{\Omega}{(\operatorname{div}_hy_h-2a_hv_h^*-b_h)v_h\,\mathrm{d}x}=0$
        %\begin{align*}
        %    \int_{\Omega}{(\operatorname{div}_hy_h-2a_hv_h^*-b_h)v_h\,\mathrm{d}x}=0\,,
        %\end{align*}
        and, consequently, taking into account that ${\operatorname{div}_hy_h-b_h\in V_h}$, is given via $v_h^*=\smash{\frac{1}{2a_h}}(\operatorname{div}_hy_h-b_h)\in V_h$, we find that\vspace{-0.5mm}
        \begin{align*}
            \sup_{v_h\in V_h}{\left\{\int_{\Omega}{(\operatorname{div}_hy_hv_h-\psi_h(\cdot,v_h))\,\mathrm{d}x}\right\}}&=\int_{\Omega}{(\operatorname{div}_hy_h-a_hv_h^*-b_h)v_h^*\,\mathrm{d}x}
            -\int_{\Omega}{c_h\,\mathrm{d}x}
            \\&=\int_{\Omega}{\bigl(\tfrac{1}{4a_h}(\operatorname{div}_hy_h-b_h)^2-c_h\bigr)\,\mathrm{d}x}
            \\&=\int_{\Omega}{\psi_h^*(\cdot,\operatorname{div}_hy_h)\,\mathrm{d}x}\,,\\[-6mm]\notag
        \end{align*}
        which is the claimed discrete convex conjugation formula \eqref{lem:discrete_convex_conjugation.1} in the case without  mass lumping.~$\qedsymbol$ % (\textit{cf}.\ Assumption \ref{ass:energy_densities_discrete}(\hyperlink{ass:energy_densities_discrete.ii.b}{ii})).
    \end{proof}

    Then, if  Assumption \ref{ass:energy_densities_discrete} is satisfied, the (Fenchel) dual problem (in the sense of \cite[Rem.~4.2, p. 60/61]{EkelandTemam1999}) to the minimization of \eqref{eq:primal_discrete} is given via the maximization of the \emph{discrete dual energy functional} $D_h\colon Y_h^*\cong Y_h\to \mathbb{R}\cup\{-\infty\}$, for every $y_h\in Y_h$ defined by\vspace{-0.75mm}
    \begin{align}\label{eq:dual_discrete}
        D_h(y_h)\coloneqq -\int_{\Omega}{\phi_h^*(\cdot,y_h)\,\mathrm{d}x}-\int_{\Omega}{I_{V_h}\{\psi_h^*(\cdot,\operatorname{div}_h y_h)\}\,\mathrm{d}x}\,,\\[-6mm]\notag
    \end{align}
    which is established along with the existence of a maximizer $z_h\in Y_h$, called \emph{discrete~dual~solution}, a discrete strong duality relation, and discrete primal optimality inclusions in the following proposition.\enlargethispage{6mm}\vspace{-0.5mm}

    \begin{proposition}[Identification of dual problem, strong duality, and primal optimality inclusions]\label{prop:discrete_duality}
        Let Assumption \ref{ass:energy_densities_discrete} be satisfied. Then, the following statements apply:
        \begin{itemize}[noitemsep,topsep=2pt,leftmargin=!,labelwidth=\widthof{(ii)}]
            \item[(i)] \hypertarget{prop:discrete_duality.i}{} The \hspace{-0.1mm}(Fenchel) \hspace{-0.1mm}dual \hspace{-0.1mm}problem \hspace{-0.1mm}to \hspace{-0.1mm}the \hspace{-0.1mm}minimization \hspace{-0.1mm}of \hspace{-0.1mm}\eqref{eq:primal_discrete} \hspace{-0.1mm}is \hspace{-0.1mm}given \hspace{-0.1mm}via \hspace{-0.1mm}the~\hspace{-0.1mm}\mbox{maximization}~\hspace{-0.1mm}of~\hspace{-0.1mm}\eqref{eq:dual_discrete};
            \item[(ii)] \hypertarget{prop:discrete_duality.ii}{} If, in addition, $(y_h\mapsto \int_{\Omega}{\phi_h(\cdot,y_h)\,\mathrm{d}x})\colon Y_h\to \mathbb{R}\cup\{+\infty\}$ is continuous at $\nabla u_h^{\star}\in Y_h$,~then~a discrete dual solution exists, a \emph{discrete strong duality relation} applies, \textit{i.e.}, we have that $I_h(u_h)=D_h(z_h)$, which is precisely equivalent to the \emph{discrete~primal~optimality~inclusions}\vspace{-0.5mm}
            \begin{subequations}\label{prop:discrete_duality.1}
           \begin{alignat}{2}\label{prop:discrete_duality.1.1}
        z_h&\in \partial_t \phi_h(\cdot,\nabla u_h)&&\quad \text{ a.e.\ in }\Omega\,,\\
        \operatorname{div}_hz_h&\in I_{V_h}{\{\partial_s\psi_h(\cdot,u_h)\}}&&\quad \text{ in }\Omega\,.\label{prop:discrete_duality.1.2}
        \end{alignat}
         \end{subequations}
        \end{itemize}
    \end{proposition}

    \begin{proof}
        \emph{ad (\hyperlink{prop:discrete_duality.i}{i}).}  To begin with, we introduce the functionals $G_h\in \Gamma_0(Y_h)$ and $F_h\in \Gamma_0(V_h)$, for every $y_h\in Y_h$ and $v_h\in V_h$, respectively, defined by\vspace{-0.5mm}
        \begin{align}\label{prop:discrete_duality.0.1}
            G_h(y_h)\coloneqq \int_{\Omega}{\phi_h(\cdot,y_h)\,\mathrm{d}x}\,,\qquad
            F_h(v_h)\coloneqq\int_{\Omega}{I_{V_h}{\{\psi_h(\cdot,v_h)\}}\,\mathrm{d}x}\,.\\[-6mm]\notag
        \end{align}
        Then, \hspace{-0.1mm}a \hspace{-0.1mm}(Fenchel) \hspace{-0.1mm}dual \hspace{-0.1mm}problem \hspace{-0.1mm}(in \hspace{-0.1mm}the \hspace{-0.1mm}sense \hspace{-0.1mm}of \hspace{-0.1mm}\cite[Rem.\ \hspace{-0.1mm}4.2, \hspace{-0.1mm}p.\ \hspace{-0.1mm}60/61]{EkelandTemam1999}) \hspace{-0.1mm}to \hspace{-0.1mm}the~\hspace{-0.1mm}\mbox{minimization}~\hspace{-0.1mm}of~\hspace{-0.1mm}\eqref{eq:primal_discrete}  is given via the maximization of $D_h\colon Y_h\to \mathbb{R}\cup\{-\infty\}$, for every $y_h\in Y_h$~defined~by\vspace{-0.5mm}
        \begin{align}\label{prop:discrete_duality.0.2}
            D_h(y_h)\coloneqq -G_h^*(y_h)-F_h^*(-\nabla_h^*y_h)\,,\\[-6mm]\notag
        \end{align}
        where \hspace{-0.1mm}$\nabla_h^*\colon \hspace{-0.12em}Y_h\hspace{-0.12em}\to\hspace{-0.12em} V_h^*$ \hspace{-0.1mm}is \hspace{-0.1mm}the \hspace{-0.1mm}adjoint \hspace{-0.1mm}operator \hspace{-0.1mm}to \hspace{-0.1mm}the \hspace{-0.1mm}restricted \hspace{-0.1mm}gradient \hspace{-0.1mm}operator \hspace{-0.1mm}${\nabla_h\hspace{-0.12em}\coloneqq \hspace{-0.12em}\nabla|_{V_h}\colon\hspace{-0.12em} V_h\hspace{-0.12em}\to \hspace{-0.12em}Y_h}$. 
        Thus, it only remains to establish that the claimed integral representation \eqref{eq:dual_discrete} applies:\vspace{-0.5mm}

        \begin{itemize}[noitemsep,topsep=2pt,leftmargin=!,labelwidth=\widthof{$\bullet$}]
            \item[$\bullet$]  Since \hspace{-0.15mm}$\phi_h(\cdot,t)\hspace{-0.15em}\in\hspace{-0.15em} \mathcal{L}^0(\mathcal{T}_h)$ \hspace{-0.15mm}for \hspace{-0.15mm}all \hspace{-0.15mm}$t\hspace{-0.15em}\in\hspace{-0.15em} \mathbb{R}^d$ \hspace{-0.15mm}(\textit{cf}.~Assumption~\ref{ass:energy_densities_discrete}(\hyperlink{ass:energy_densities_discrete.i}{i}.\hyperlink{ass:energy_densities_discrete.i.a}{a})), 
            \hspace{-0.15mm}denoting \hspace{-0.15mm}by \hspace{-0.15mm}${x_T\hspace{-0.15em}\coloneqq \hspace{-0.15em}\frac{1}{d+1}\sum_{\nu \in \mathcal{N}_h}{\nu}\hspace{-0.15em}\in\hspace{-0.15em} T}$ the barycenter of an element $T\in \mathcal{T}_h$,
            for every $y_h\in Y_h$,~we~find~that\vspace{-0.5mm}
            \begin{align}\label{prop:discrete_duality.2}
            \begin{aligned} 
                G_h^*(y_h)&=\sup_{\widehat{y}_h\in Y_h}{\left\{\int_{\Omega}{y_h\cdot \widehat{y}_h\,\mathrm{d}x}-\int_{\Omega}{\phi_h(\cdot,\widehat{y}_h)\,\mathrm{d}x}\right\}}
                \\[-0.25mm]&=\sum_{T\in \mathcal{T}_h}{\vert T\vert\sup_{t\in \mathbb{R}^d}{\{y_h(x_T)\cdot t-\phi_h(x_T,t)\}}}
                \\[-0.25mm]&=\sum_{T\in \mathcal{T}_h}{\vert T\vert\phi_h^*(x_T,y_h(x_T))}
                \\[-0.25mm]&=\int_{\Omega}{\phi_h^*(\cdot,y_h)\,\mathrm{d}x}\,;
            \end{aligned}\\[-6mm]\notag
            \end{align}
            
            \item[$\bullet$] By the definition of the discrete divergence \eqref{def:discrete_divergence} and Lemma \ref{lem:discrete_convex_conjugation}, for every $y_h\in Y_h$,~we~find~that\vspace{-0.5mm}
            \begin{align}\label{prop:discrete_duality.3}
            \begin{aligned} 
                F_h^*(-\nabla_h^* y_h)&=\sup_{v_h\in V_h}{\left\{\int_{\Omega}{-y_h\cdot \nabla v_h\,\mathrm{d}x}-\int_{\Omega}{I_{V_h}\{\psi_h(\cdot,v_h)\}\,\mathrm{d}x}\right\}}
                \\&=\sup_{v_h\in V_h}{\left\{\int_{\Omega}{I_{V_h}\{\operatorname{div}_hy_h v_h-\psi_h(\cdot,v_h)\}\,\mathrm{d}x}\right\}}
                \\&=\int_{\Omega}{I_{V_h}\{\psi_h^*(\cdot,\operatorname{div}_hy_h)\}\,\mathrm{d}x}\,.
            \end{aligned}\\[-6mm]\notag
            \end{align} 
        \end{itemize} 
        Using  \eqref{prop:discrete_duality.2} and \eqref{prop:discrete_duality.3} in \eqref{prop:discrete_duality.0.2},  we conclude that the claimed integral representation \eqref{eq:dual_discrete} applies.

        \emph{ad (\hyperlink{prop:discrete_duality.ii}{ii}).} 
        By Assumption \ref{ass:energy_densities_discrete}(\hyperlink{ass:energy_densities_discrete.iii}{iii}), we have that $G_h(\nabla u_h^{\star})<+\infty$, $F_h(u_h^{\star})<+\infty$, and, by the additional assumption of this proposition, $G_h\colon Y_h\to \mathbb{R}\cup\{+\infty\}$ is continuous~at~${\nabla u_h^{\star}\in Y_h}$. Therefore, the Fenchel duality theorem (\textit{cf}.\ \cite[Rem.\ 4.1, eqs.\ (4.21), p.\ 61]{EkelandTemam1999}) yields the existence of a discrete dual solution $z_h\in Y_h$ as well as that a discrete strong duality relation applies. %, \textit{i.e.}. we have that $I(u)=D(z)$. 
        According to \cite[Rem.\ 4.1, eqs.\ (4.22),(4.23), p.\ 61]{EkelandTemam1999}, the discrete strong duality relation~is~equivalent to\vspace{-0.5mm} %that
        \begin{subequations}\label{prop:discrete_duality.6}
        \begin{align}\label{prop:discrete_duality.6.1}
            \int_{\Omega}{(\phi_h^*(\cdot,z_h)-z_h\cdot\nabla u_h+\phi_h(\cdot,\nabla u_h))\,\mathrm{d}x}&=0\,,\\
            \int_{\Omega}{I_{V_h}\{\psi^*_h(\cdot,\operatorname{div}_hz_h)-\operatorname{div}_hz_h u_h+\psi_h(\cdot, u_h)\}\,\mathrm{d}x}&=0\,.\label{prop:discrete_duality.6.2}
        \end{align}\\[-3mm]\notag
        \end{subequations}
        By the Fenchel--Young inequality (\textit{cf}.\ \cite[Prop.\ 51.2]{Zeidler1985III}), the integrands of the integrals in \eqref{prop:discrete_duality.6.1} and \eqref{prop:discrete_duality.6.2}~are~point-wise (a.e.) non-negative, so that \eqref{prop:discrete_duality.6} is precisely equivalent to\vspace{-0.5mm}
        \begin{subequations}\label{prop:discrete_duality.7}
        \begin{alignat}{2}\label{prop:discrete_duality.7.1}
           \phi^*_h(\cdot,z_h)-z_h\cdot\nabla u_h+\phi_h(\cdot,\nabla u_h)&=0&&\quad\text{ a.e.\ in }\Omega\,,\\
           I_{V_h}\{\psi^*_h(\cdot,\operatorname{div}_hz_h)-\operatorname{div}_hz_h\, u_h+\psi_h(\cdot, u_h)\}&=0&&\quad\text{ in }\Omega\,,\label{prop:discrete_duality.7.2}
        \end{alignat}\\[-4.5mm]\notag
        \end{subequations}
        which, by the equality condition in the Fenchel--Young inequality (\textit{cf}.\ \cite[Prop.\ 51.2]{Zeidler1985III}),~in~turn, is equivalent to the claimed discrete primal optimality inclusions \eqref{prop:discrete_duality.1}.
    \end{proof} 

    \begin{remark}[Discrete \emph{a posteriori} error identity]\label{rem:primal-dual_gap_identity}
         Inspired by \cite{BartelsGudiKaltenbach2025,AntilBartelsKaltenbachKhandelwal2025,DieningStorn2025,BartelsKaltenbach2026},
        the~discrete~strong duality relation in Proposition \ref{prop:discrete_duality}(\hyperlink{prop:discrete_duality.ii}{ii}) implies a discrete \emph{a posteriori} error identity, which~may~be~used in a stopping criterion enabling explicit error control or in line-search strategies~in~\mbox{iterative}~\mbox{solvers}:
        \begin{itemize}[noitemsep,topsep=2pt,leftmargin=!,labelwidth=\widthof{$\bullet$}]
            \item[$\bullet$] \emph{Estimator:} The \emph{discrete primal-dual gap estimator} $\eta_{\textup{gap},h}^2\colon V_h\times Y_h\to [0,+\infty]$, for every $v_h\in V_h$ and $y_h\in Y_h$ defined by $\eta_{\textup{gap},h}^2(v_h,y_h)\coloneqq I_h(v_h)-D_h(y_h)$, measures~the~discrete~\mbox{primal-dual}~gap; 
        
        \item[$\bullet$]  \emph{Primal error:} The \emph{discrete optimal convexity measure} $\rho_{I_h}^2\colon V_h\to [0,+\infty]$, for every $v_h\in V_h$ defined by $\rho_{I_h}^2(v_h)\coloneqq I_h(v_h)-I_h(u_h)$, 
        measures the strong convexity of the discrete primal energy functional \hspace{-0.1mm}\eqref{eq:primal_discrete} \hspace{-0.1mm}at \hspace{-0.1mm}a \hspace{-0.1mm}discrete \hspace{-0.1mm}primal \hspace{-0.1mm}solution \hspace{-0.1mm}$u_h\hspace{-0.175em}\in\hspace{-0.175em} V_h$ \hspace{-0.1mm}and \hspace{-0.1mm}serves \hspace{-0.1mm}as~\hspace{-0.1mm}a~\hspace{-0.1mm}\mbox{measure}~\hspace{-0.1mm}for~\hspace{-0.1mm}the~\hspace{-0.1mm}\mbox{primal}~\hspace{-0.1mm}\mbox{error};
        \item[$\bullet$]  \emph{Dual error:} The \emph{discrete optimal convexity measure} $\rho_{-D_h}^2\colon Y_h\to [0,+\infty]$, for every $y_h\in Y_h$ defined by $\rho_{-D_h}^2(y_h)\coloneqq -D_h(y_h)+D_h(z_h)$, 
        measures the strong concavity of the discrete dual energy \hspace{-0.1mm}functional \hspace{-0.1mm}\eqref{eq:dual_discrete} \hspace{-0.1mm}at \hspace{-0.1mm}a \hspace{-0.1mm}discrete \hspace{-0.1mm}dual \hspace{-0.1mm}solution \hspace{-0.1mm}$z_h\hspace{-0.15em}\in \hspace{-0.15em} Y_h$ \hspace{-0.1mm}and \hspace{-0.1mm}serves~\hspace{-0.1mm}as~\hspace{-0.1mm}a~\hspace{-0.1mm}\mbox{measure}~\hspace{-0.1mm}for~\hspace{-0.1mm}the~\hspace{-0.1mm}dual~\hspace{-0.1mm}\mbox{error}.\enlargethispage{2.5mm}
        \end{itemize} 
        
        If Assumption \ref{ass:energy_densities_discrete} is satisfied, by Proposition \ref{prop:discrete_duality}(\hyperlink{prop:discrete_duality.ii}{ii}), a discrete strong duality relation~applies, which, for every $v_h\in V_h$ and $y_h\in Y_h$, implies the \emph{discrete \emph{a posteriori} error identity}\vspace{-0.5mm}
        \begin{align}\label{rem:primal-dual_gap_identity.1}
             \smash{\rho_{I_h}^2(v_h)+
             \rho_{-D_h}^2(y_h)=\eta_{\textup{gap},h}^2(v_h,y_h)\,.}\\[-6mm]\notag
        \end{align}
        By the integral representations of the discrete primal energy functional \eqref{eq:primal_discrete} and dual energy functional \eqref{eq:dual_discrete} as well as the definition of the discrete divergence operator \eqref{def:discrete_divergence},~the~discrete~primal-dual gap estimator, for every $v_h\in V_h$ and $y_h\in Y_h$, admits the integral representation\vspace{-0.5mm}
        \begin{align}\label{rem:primal-dual_gap_identity.1}
        \begin{aligned}
\eta_{\textup{gap},h}^2(v_h,y_h)&=\int_{\Omega}{(\phi_h^*(\cdot,y_h)-y_h\cdot \nabla v_h+\phi_h(\cdot,\nabla v_h))\,\mathrm{d}x}
            \\&\quad +\int_{\Omega}{I_{V_h}\{\psi_h^*(\cdot,\operatorname{div}_hy_h)-\operatorname{div}_hy_hv_h+\psi_h(\cdot,v_h)\}\,\mathrm{d}x}\,,
        \end{aligned}\\[-6mm]\notag
        \end{align}
        which \hspace{-0.15mm}reveals \hspace{-0.15mm}that \hspace{-0.15mm}it \hspace{-0.15mm}precisely \hspace{-0.15mm}measures \hspace{-0.15mm}the \hspace{-0.15mm}violation \hspace{-0.15mm}of \hspace{-0.15mm}the \hspace{-0.15mm}discrete \hspace{-0.15mm}primal \hspace{-0.15mm}optimality~\hspace{-0.15mm}\mbox{inclusions}~\hspace{-0.15mm}\eqref{prop:discrete_duality.1}, the starting point of the $\operatorname{prox}$-based semi-smooth Newton method proposed~in~the~next~section.
    \end{remark}

   As in Remark~\ref{rem:proximity_integral_functionals}, the proximity
operators of the discrete energy functionals
\eqref{prop:discrete_duality.0.1} inherit a local decoupling structure
and admit point-wise representations in terms of the proximity~operators
(with respect to the second argument) of the underlying discrete energy
densities.

    \begin{proposition}[Proximity operators of the discrete energy functionals \eqref{prop:discrete_duality.0.1}]\label{prop:proximity_of_integral_functionals}
        Let Assumption~\ref{ass:energy_densities_discrete} be satisfied. Then, the following statements apply:
        \begin{itemize}[noitemsep,topsep=2pt,leftmargin=!,labelwidth=\widthof{(ii)}]
            \item[(i)] \hypertarget{prop:proximity_of_integral_functionals.i}{} The proximity operators $\operatorname{prox}_{G_h}\colon \hspace{-0.1em}Y_h\hspace{-0.1em}\to\hspace{-0.1em} Y_h$ and $\operatorname{prox}_{\phi_h(x,\cdot)}\colon \hspace{-0.1em}\mathbb{R}^d\hspace{-0.1em}\to\hspace{-0.1em}\mathbb{R}^d$, $x\hspace{-0.1em}\in\hspace{-0.1em} \Omega$,~for~every~$y_h\hspace{-0.1em}\in\hspace{-0.1em} Y_h$, are related via
            \begin{align}\label{prop:proximity_of_integral_functionals.i}
               \operatorname{prox}_{G_h}(y_h)(x)= \operatorname{prox}_{\phi_h(x,\cdot)}(y_h(x))\quad \text{ for a.e.\ }x\in \Omega\,.
            \end{align}
            
            \item[(ii)] \hypertarget{prop:proximity_of_integral_functionals.ii}{} The proximity operators $\operatorname{prox}_{F_h}\colon V_h\to V_h$ and $\operatorname{prox}_{\psi_h(x,\cdot)}\colon \mathbb{R}\to\mathbb{R}$, $x\in \Omega$, for every $v_h\in V_h$ are related via
            \begin{align}\label{prop:proximity_of_integral_functionals.ii}
               \operatorname{prox}_{F_h}(v_h)(x)&= I_{V_h}\{\operatorname{prox}_{\psi_h}(v_h)\}(x)\quad \text{ for all }x\in \Omega\,,
            \end{align}
            where $\operatorname{prox}_{\psi_h}(v_h)(x)\coloneqq \operatorname{prox}_{\psi_h(x,\cdot)}(v_h(x))$ for all $x\in \Omega$.
        \end{itemize}
    \end{proposition}

    \begin{proof}
        \emph{ad (\hyperlink{prop:proximity_of_integral_functionals.i}{i}).} For every $y_h\in Y_h$, we have that
        \begin{align*}
            \operatorname{prox}_{G_h}(y_h)&=\underset{\widehat{y}_h\in Y_h}{\operatorname{arg\,min}}{\big\{G_h(\widehat{y}_h)+\tfrac{1}{2}\|y_h-\widehat{y}_h\|_{2,\Omega}^2\big\}}
            \\&=\underset{\widehat{y}_h\in Y_h}{\operatorname{arg\,min}}{\bigg\{\int_{\Omega}{(\phi_h(\cdot,\widehat{y}_h)+\tfrac{1}{2}\vert y_h-\widehat{y}_h\vert^2)\,\mathrm{d}x}\bigg\}}
            \\&=\sum_{T\in \mathcal{T}_h} {\chi_T \underset{t\in \mathbb{R}^d}{\operatorname{arg\,min}}{\{\phi_h(x_T,t)+\tfrac{1}{2}\vert y_h(x_T)-t\vert^2\}}}
            \\&=\sum_{T\in \mathcal{T}_h} {\chi_T\operatorname{prox}_{\phi_h(x_T,\cdot)}(y_h(x_T))} \,,
            %\\&=\operatorname{prox}_{\phi_h(\cdot,\cdot)}(y_h)\,.
        \end{align*}
        which implies the claimed local  relation \eqref{prop:proximity_of_integral_functionals.i}.

         \emph{ad (\hyperlink{prop:proximity_of_integral_functionals.ii}{ii}).} We distinguish the cases with and without mass lumping (\textit{i.e.}, Assumption \ref{ass:energy_densities_discrete}(\hyperlink{ass:energy_densities_discrete.ii.a}{ii.a}/\hyperlink{ass:energy_densities_discrete.ii.b}{b})):

        $\bullet$ \emph{Case 1: (\hyperlink{ML}{ML}).} In this case, for every $v_h\in V_h$, we have that
        \begin{align*}
            \operatorname{prox}_{F_h}(v_h)&=\underset{\widehat{v}_h\in V_h}{\operatorname{arg\,min}}{\big\{F_h(\widehat{v}_h)+\tfrac{1}{2}\|v_h-\widehat{v}_h\|_h^2\big\}}
            \\&=\underset{\widehat{v}_h\in V_h}{\operatorname{arg\,min}}{\bigg\{\int_{\Omega}{I_h^{p1}\{\psi_h(\cdot,\widehat{v}_h)+\tfrac{1}{2}\vert v_h-\widehat{v}_h\vert^2\}\,\mathrm{d}x}\bigg\}}
            \\&=\sum_{\nu\in \mathcal{N}_h^F} {\varphi_\nu \underset{s\in \mathbb{R}}{\operatorname{arg\,min}}{\{\psi_h(\nu,s)+\tfrac{1}{2}\vert v_h(\nu)-s\vert^2\}}}
            \\&=\sum_{\nu\in \mathcal{N}_h^F} {\varphi_\nu \operatorname{prox}_{\psi_h(\nu,\cdot)}(v_h(\nu))} \,,
            %\\&=I_h^{p1}\operatorname{prox}_{\psi_h(\cdot,\cdot)}(v_h)\,,
        \end{align*}
        which implies the claimed local relation \eqref{prop:proximity_of_integral_functionals.ii} in the case with mass lumping.

        $\bullet$ \emph{Case 2: (\hyperlink{NML}{NML}).} In this case, for every $v_h\in V_h$, we have that
        \begin{align}\label{prop:proximity_of_integral_functionals.3}
            \begin{aligned} 
            \operatorname{prox}_{F_h}(v_h)&=\underset{\widehat{v}_h\in V_h}{\operatorname{arg\,min}}{\big\{F_h(\widehat{v}_h)+\tfrac{1}{2}\|v_h-\widehat{v}_h\|_{2,\Omega}^2\big\}}
            \\&=\underset{\widehat{v}_h\in V_h}{\operatorname{arg\,min}}{\bigg\{\int_{\Omega}{(a_h\vert\widehat{v}_h\vert^2+b_h\widehat{v}_h+c_h +\tfrac{1}{2}\vert v_h-\widehat{v}_h\vert^2)\,\mathrm{d}x}\bigg\}}
            \\&=\tfrac{1}{1+2a_h}(v_h-b_h)
            \\&=\operatorname{prox}_{\psi_h}(v_h)\,,
             \end{aligned}
        \end{align}
        which implies the claimed local relation \eqref{prop:proximity_of_integral_functionals.ii} in the case without mass lumping.
    \end{proof}

    The following result shows that the Newton derivative selections of the proximity
operators associated with the discrete integral functionals \eqref{prop:discrete_duality.0.1} inherit the
local decoupling structure established in
Proposition~\ref{prop:proximity_of_integral_functionals}. In particular,
they can be characterized element-wise or node-wise in terms of the
Newton derivatives of the corresponding point-wise proximity operators.

    \begin{corollary}[Newton derivatives of proximity operators of the discrete integral~\mbox{functionals}~\eqref{prop:discrete_duality.0.1}]\label{prop:derivative_proximity_of_integral_functionals}
        Let Assumption \ref{ass:energy_densities_discrete} be satisfied. Then, the following statements apply:
        \begin{itemize}[noitemsep,topsep=2pt,leftmargin=!,labelwidth=\widthof{(ii)}]
            \item[(i)] \hypertarget{prop:derivative_proximity_of_integral_functionals.i}{} For every $y_h\in Y_h$, a mapping $\mathtt{J}_{\smash{\operatorname{prox}_{\smash{G_h}}}}\colon B_{\varepsilon_h}^{Y_h}(y_h)\to \mathcal{L}(Y_h)$, $\varepsilon_h>0$, is a Newton derivative~selection  of the proximity operator $\operatorname{prox}_{G_h}\colon Y_h\to  Y_h$ at $y_h\in Y_h$~if~and~only~if,~for~every~$T\in \mathcal{T}_h$, the mapping
            $\mathtt{J}_{\smash{\operatorname{prox}_{\smash{\phi_h(x_T,\cdot)}}}}\colon B_{\varepsilon_h}^d(y_h(x_T))\to \mathbb{R}^{d\times d}$, for every $t\in B_{\varepsilon_h}^d(y_h(x_T))$ defined by 
            \begin{align}\label{prop:derivative_proximity_of_integral_functionals.1}
                \mathtt{J}_{\operatorname{prox}_{\smash{\phi_h(x_T,\cdot)}}}(t)\coloneqq \mathtt{J}_{\smash{\operatorname{prox}_{\smash{G_h}}}}(t\chi_T)\,,
            \end{align}
            is \hspace{-0.1mm}a \hspace{-0.1mm}Newton \hspace{-0.1mm}derivative \hspace{-0.1mm}selection \hspace{-0.1mm}of \hspace{-0.1mm}the \hspace{-0.1mm}proximity \hspace{-0.1mm}operator \hspace{-0.1mm}$\operatorname{prox}_{\phi_h(x_T,\cdot)}\colon\hspace{-0.15em}\mathbb{R}^d\hspace{-0.15em}\to\hspace{-0.15em} \mathbb{R}^d$~\hspace{-0.1mm}at~\hspace{-0.1mm}${y_h(x_T)\hspace{-0.15em}\in\hspace{-0.15em} \mathbb{R}^d}$;
            
            \item[(ii)] \hypertarget{prop:derivative_proximity_of_integral_functionals.ii}{} For every $v_h\in V_h$, a mapping $\mathtt{J}_{\smash{\operatorname{prox}_{\smash{F_h}}}}\colon B_{\varepsilon_h}^{V_h}(v_h)\to \mathcal{L}(V_h)$, $\varepsilon_h>0$, is a Newton~\mbox{derivative}~selection  of the proximity operator $\operatorname{prox}_{F_h}\colon V_h\to  V_h$ at $v_h\in V_h$ if and only if,~for~every~$\nu\in \mathcal{N}_h^F$, the mapping
            $\mathtt{J}_{\smash{\operatorname{prox}_{\smash{\psi_h(\nu,\cdot)}}}}\colon B_{\varepsilon_h}^1(v_h(\nu))\to \mathbb{R}$, for every $s\in B_{\varepsilon_h}^1(v_h(\nu))$~defined~by 
            \begin{align}\label{prop:derivative_proximity_of_integral_functionals.2}
                \mathtt{J}_{\smash{\operatorname{prox}_{\smash{\psi_h(\nu,\cdot)}}}}(s)\coloneqq \mathtt{J}_{\smash{\operatorname{prox}_{\smash{F_h}}}}(s\varphi_\nu)\,,
            \end{align}
            is a Newton derivative selection of the proximity operator $\operatorname{prox}_{\psi_h(\nu,\cdot)}\colon\mathbb{R}\to \mathbb{R}$ at $v_h(\nu)\in \mathbb{R}$.
        \end{itemize}
    \end{corollary}

    \begin{proof}
        \emph{ad (\hyperlink{prop:derivative_proximity_of_integral_functionals.i}{i}).} 
        Using Proposition \ref{prop:proximity_of_integral_functionals}(\hyperlink{prop:proximity_of_integral_functionals.i}{i}),
        for every $\mathtt{L}_h\in (\mathcal{L}^0(\mathcal{T}_h))^{d\times d}$ and $\tau_h\in Y_h$, we find that
        \begin{align*}%\label{prop:derivative_proximity_of_integral_functionals.3}
            &\|\operatorname{prox}_{G_h}(y_h+\tau_h)-\operatorname{prox}_{G_h}(y_h)-\mathtt{L}_h\tau_h\|_{Y_h}^2\\
            &\quad=\sum_{T\in \mathcal{T}_h}{\vert T\vert\vert \operatorname{prox}_{\phi_h(x_T,\cdot)}(y_h(x_T)+\tau_h(x_T))-\operatorname{prox}_{\phi_h(x_T,\cdot)}(y_h(x_T))-\mathtt{L}_h(x_T)\tau_h(x_T)\vert^2}\,,
        \end{align*}
        %and $\|\tau_h\|_{Y_h}^2=\sum_{T\in \mathcal{T}_h}{\vert T\vert \vert \tau_h(x_T)\vert^2}$, 
        which readily implies the assertion.
        %Therefore, if $\mathtt{J}_{\smash{\operatorname{prox}_{\smash{G_h}}}}\colon B_{\varepsilon_h}^{Y_h}(y_h)\to \mathcal{L}(Y_h)$, $\varepsilon_h>0$, is a Newton derivative~selection  of the proximity operator $\operatorname{prox}_{G_h}\colon Y_h\to  Y_h$ at $y_h\in Y_h$, choosing $\tau_h=t\chi_T\in Y_h$, for arbitrary $t\in \mathbb{R}^d$ and $T\in \mathcal{T}_h$, and $\mathtt{L}_h=\mathtt{J}_{\smash{\operatorname{prox}_{\smash{G_h}}}}(y_h+t\chi_T)$ in \eqref{prop:derivative_proximity_of_integral_functionals.3}, we find that
        %\begin{align*}
        %    \lim_{h\to 0^+}{\big\{\tfrac{1}{\vert T\vert \vert t\vert^2}\vert \operatorname{prox}_{\phi_h(x_T,\cdot)}(y_h(x_T)+t)-\operatorname{prox}_{\phi_h(x_T,\cdot)}(y_h(x_T))-\mathtt{J}_{\smash{\operatorname{prox}_{\smash{G_h}}}}(y_h(x_t)+t)t\vert\big\}}=0\,.
        %\end{align*}

        \emph{ad (\hyperlink{prop:proximity_of_integral_functionals.ii}{ii}).} We distinguish the cases with and without mass lumping (\textit{i.e.}, Assumption \ref{ass:energy_densities_discrete}(\hyperlink{ass:energy_densities_discrete.ii.a}{ii.a}/\hyperlink{ass:energy_densities_discrete.ii.b}{b})):

        $\bullet$ \emph{Case 1: (\hyperlink{ML}{ML}).} Using Proposition \ref{prop:proximity_of_integral_functionals}(\hyperlink{prop:proximity_of_integral_functionals.ii}{ii}),
        for every $\ell_h\in V_h$ and $w_h\in V_h$,~we~find that
        \begin{align*}
            &\|\operatorname{prox}_{F_h}(v_h+w_h)-\operatorname{prox}_{F_h}(v_h)-\ell_hw_h\|_{V_h}^2\\
            &\quad=\sum_{\nu\in \mathcal{N}_h^F}{\beta_\nu\vert \operatorname{prox}_{\psi_h(\nu,\cdot)}(v_h(\nu)+w_h(\nu))-\operatorname{prox}_{\psi_h(\nu,\cdot)}(v_h(\nu))-\ell_h(\nu)w_h(\nu)\vert^2}\,,
        \end{align*}
        which readily implies the assertion.

        $\bullet$ \emph{Case 2: (\hyperlink{NML}{NML}).} Using Proposition~\ref{prop:proximity_of_integral_functionals}(\hyperlink{prop:proximity_of_integral_functionals.ii}{ii}) (more precisely, the representation formula~\eqref{prop:proximity_of_integral_functionals.3}), for every $\ell_h\in V_h$ and $w_h\in V_h$, we find that
    \begin{align}\label{prop:derivative_proximity_of_integral_functionals.3}
        \|\operatorname{prox}_{F_h}(v_h+w_h)-\operatorname{prox}_{F_h}(v_h)-\ell_hw_h\|_{V_h}^2
        &= \big\|\bigl(\tfrac{1}{1+2a_h}\mathbbone_{\mathcal{L}(V_h)}-\ell_h\bigr)w_h\big\|_{V_h}^2\, .
    \end{align}
    On the other hand, by the standard equivalence of $\|\cdot\|_{V_h}$ to a dof-based norm (\textit{cf}.\ \cite[Prop.~12.5]{EG21I}), for every $\ell_h\in V_h$ and $w_h\in V_h$, we have that
     \begin{align}\label{prop:derivative_proximity_of_integral_functionals.4}
        \begin{aligned} 
        &\big\|\bigl(\tfrac{1}{1+2a_h}\mathbbone_{\mathcal{L}(V_h)}-\ell_h\bigr)w_h\big\|_{V_h}^2\\&\quad\sim \sum_{\nu\in \mathcal{N}_h^{F}}{h_\nu^d\big\vert \tfrac{1}{1+2a_h}w_h(\nu)-\ell_h(\nu) w_h(\nu)\big\vert^2}
        \\&\quad=\sum_{\nu\in \mathcal{N}_h^F}{\beta_\nu\vert \operatorname{prox}_{\psi_h(\nu,\cdot)}(v_h(\nu)+w_h(\nu))-\operatorname{prox}_{\psi_h(\nu,\cdot)}(v_h(\nu))-\ell_h(\nu)w_h(\nu)\vert^2}\,,
        \end{aligned}
    \end{align}
    where $h_\nu\hspace{-0.1em}\coloneqq\hspace{-0.1em} \textup{diam}(\operatorname{supp}\varphi_\nu)\hspace{-0.1em}\sim\hspace{-0.1em} \beta_\nu^{\frac{1}{d}}$ for all $\nu\hspace{-0.1em}\in\hspace{-0.1em} \mathcal{N}_h^{F}$ and the implicit constants in $\sim$ do not~depend~on~$h$. As a consequence, by combining \eqref{prop:derivative_proximity_of_integral_functionals.3} and \eqref{prop:derivative_proximity_of_integral_functionals.4}, we conclude the assertion analogously to the case with mass~lumping.
    \end{proof}
     
     \section{A $\operatorname{prox}$-based semi-smooth Newton method}\label{sec:semi-smooth_newton}

     \hspace{5mm}In \hspace{-0.1mm}this \hspace{-0.1mm}section, \hspace{-0.1mm}we \hspace{-0.1mm}derive \hspace{-0.1mm}the \hspace{-0.1mm}$\operatorname{prox}$-based \hspace{-0.1mm}semi-smooth \hspace{-0.1mm}Newton \hspace{-0.1mm}method. \hspace{-0.1mm}In \hspace{-0.1mm}doing \hspace{-0.1mm}so,~\hspace{-0.1mm}for~\hspace{-0.1mm}fixed,~\hspace{-0.1mm}but arbitrary proximity parameters  $\gamma_1,\gamma_2>0$, the 
     discrete energy  densities~${\phi_h\colon \Omega \times  \mathbb{R}^d\to\mathbb{R}\cup\{+\infty\}}$ and $\psi_h\colon \Omega\times \mathbb{R}\to \mathbb{R}\cup\{+\infty\}$ such that Assumption \ref{ass:energy_densities_discrete} is satisfied, for Newton derivative selections at $\nabla u_h+\gamma_1 z_h\in Y_h$ and $u_h+\gamma_2 \operatorname{div}_hz_h\in V_h$
     (the existence of which is ensured~by~Lemma~\ref{lem:prop_proximity}(\hyperlink{lem:prop_proximity.iii}{iii})) of the proximity operators $\operatorname{prox}_{\gamma_1 G_h}\colon  Y_h\to  Y_h$ and $\operatorname{prox}_{\gamma_2 F_h}\colon  V_h\to  V_h$ of the  
     discrete energy functionals \eqref{prop:discrete_duality.0.1}, respectively, we employ the notation%abbreviations
     \begin{subequations}\label{eq:abbreviations}
     \begin{align}
         %\mathtt{J}_{\gamma_1 G_h}&\coloneqq 
         \mathtt{J}_{\operatorname{prox}_{\gamma_1 G_h}}&\colon  B_{\varepsilon_h}^{Y_h}(\nabla u_h+\gamma_1 z_h)\to \mathcal{L}(Y_h)\,,\\
        % \mathtt{J}_{\gamma_2 F_h}&\coloneqq 
         \mathtt{J}_{\operatorname{prox}_{\gamma_2 F_h}}&\colon B_{\varepsilon_h}^{V_h}(u_h+\gamma_2 \operatorname{div}_hz_h)\to \mathcal{L}(V_h)\,.
     \end{align}
     \end{subequations}
     Here, $\varepsilon_h>0$ is a possibly mesh-dependent radius.\enlargethispage{2.5mm}
     
    Then, the $\operatorname{prox}$-based semi-smooth Newton method is derived through~the~following~steps: 
    To begin with, the discrete primal optimality inclusions \eqref{prop:discrete_duality.1} equivalently may~be~rewritten~as 
    \begin{subequations}\label{sec:semi-smooth_newton.1}
    \begin{alignat}{2}
       \nabla u_h+ \gamma_1z_h&\in (\operatorname{id}_{\mathbb{R}^d}+\gamma_1\partial_t \phi_h)(\cdot,\nabla u_h)&&\quad \text{ a.e.\ in }\Omega\,,\label{sec:semi-smooth_newton.1.1}\\
        u_h+\gamma_2\operatorname{div}_hz_h&\in I_{V_h}\{(\operatorname{id}_{\mathbb{R}}+ \gamma_2\partial_s \psi_h)(\cdot,u_h)\}&&\quad \text{ in }\Omega\,.\label{sec:semi-smooth_newton.1.2}
    \end{alignat}
    \end{subequations}
    Then, by Lemma \ref{lem:resolvent}(\hyperlink{lem:resolvent.i}{i}),  the inclusions in \eqref{sec:semi-smooth_newton.1} are each equivalent to the \emph{discrete proximal optimality conditions}
    \begin{subequations}\label{sec:semi-smooth_newton.2}
    \begin{alignat}{2}\label{sec:semi-smooth_newton.2.1}
       \nabla u_h&=\operatorname{prox}_{\gamma_1 \phi_h}(\cdot,\nabla u_h+ \gamma_1z_h)&&\quad \text{ a.e.\ in }\Omega\,,\\
        u_h&=I_{V_h}\{\operatorname{prox}_{\gamma_2 \psi_h}(\cdot,u_h+\gamma_2\operatorname{div}_hz_h)\}&&\quad \text{  in }\Omega\,,\label{sec:semi-smooth_newton.2.2}
    \end{alignat} 
    \end{subequations}
    respectively.
    Thus, based on Proposition \ref{prop:proximity_of_integral_functionals}, the discrete primal optimality~\mbox{inclusions}~\eqref{prop:discrete_duality.1}~equiva\-lently may be rewritten as a root-finding problem for the nonlinear mapping $\mathtt{F}_h\colon Y_h\times V_h\to  Y_h\times V_h$, for every $(y_h,v_h)\in Y_h\times V_h$ defined by
    \begin{align}\label{sec:semi-smooth_newton.3}
        \mathtt{F}_h(y_h,v_h)\coloneqq \left[\begin{array}{c}
            \nabla v_h- \operatorname{prox}_{\gamma_1 G_h}(\nabla v_h+ \gamma_1y_h)  \\ 
             v_h- \operatorname{prox}_{\gamma_2 F_h}(v_h+\gamma_2\operatorname{div}_hy_h)
        \end{array}\right]\quad\text{ in }Y_h\times V_h\,.
    \end{align}
    %Then, by Assumption \ref{ass:wellposedness}(i), 
    Due to the point-wise Newton differentiability of the proximity operators $\operatorname{prox}_{\gamma_1 G_h}\colon  Y_h\to  Y_h$~and $\operatorname{prox}_{\gamma_2 F_h}\colon  V_h\to  V_h$ (\textit{cf}.\ Lemma~\ref{lem:prop_proximity}(\hyperlink{lem:prop_proximity.iii}{iii})), the mapping
    \eqref{sec:semi-smooth_newton.3}~is~point-wise~Newton~\mbox{differentiable}~and, %employing the abbreviations \eqref{eq:abbreviations}, 
    a Newton derivative selection $\mathtt{J}_{\mathtt{F}_h}\colon B_{\varepsilon_h}^{Y_h\times V_h}(z_h,u_h)\to \mathcal{L}(Y_h\times V_h)$ at $(z_h,u_h)\in  Y_h\times V_h$, for every $(y_h,v_h)\in  B_{\varepsilon_h}^{Y_h\times V_h}(z_h,u_h)$, is defined by
     \begin{align}\label{sec:semi-smooth_newton.4} 
     \mathtt{J}_{\mathtt{F}_h}(y_h,v_h)\coloneqq \left[\begin{array}{cc}
     -\gamma_1\mathtt{J}_{\operatorname{prox}_{\gamma_1G_h}}(a_h)  & (\mathbbone_{\mathcal{L}(Y_h)} - \mathtt{J}_{\operatorname{prox}_{\gamma_1G_h}}(a_h))\nabla_h
    \\
    -\gamma_2\mathtt{J}_{\operatorname{prox}_{\gamma_2F_h}}(b_h)\operatorname{div}_h  &   \mathbbone_{\mathcal{L}(V_h)}- \mathtt{J}_{\operatorname{prox}_{\gamma_2F_h}}(b_h)   
        \end{array}\right]\quad\text{ in }\mathcal{L}(Y_h\times V_h)\,,
    \end{align} 
    where $a_h\coloneqq \nabla v_h+ \gamma_1y_h\in Y_h$, $b_h\coloneqq v_h+\gamma_2\operatorname{div}_hy_h\in V_h$ and, again, $\nabla_h\coloneqq \nabla|_{V_h}\colon V_h\to Y_h$~denotes the restricted gradient operator.
    
    These observations motivate the formulation of the following $\operatorname{prox}$-based 
    semi-smooth Newton method:

    \begin{algorithm}[$\operatorname{prox}$-based semi-smooth Newton method]\label{alg:semi-smooth_newton}
        Let $\gamma_1,\gamma_2>0$ be proximity parameters, let $\varepsilon_{\mathtt{abs}}^h,\varepsilon_{\mathtt{rel}}^h>0$ be  stopping parameters, let $(z_h^0,u_h^0)\in Y_h\times V_h$ be an initial iterate, and let $k_{\mathtt{max}}^h\in \mathbb{N}\cup\{+\infty\}$ be a maximum number of iterations. Then, for $k=0,\ldots,k_{\mathtt{max}}^h$, perform the following iteration loop:
        \begin{itemize}[noitemsep,topsep=2pt,leftmargin=!,labelwidth=\widthof{(ii)}]

             %\item[(i)] \hypertarget{alg:semi-smooth_newton.i}{} If $\mathtt{F}_h(z^k,u^k)=0_{Y_h\times V_h}$ in $Y_h\times V_h$, then \textup{STOP}; otherwise, continue with Step (\hyperlink{semi-smooth_newton.ii}{ii});

             %\item[(i)] %Check if $\mathtt{F}_h\colon Y_h\times V_h\to Y_h\times V_h$ is Newton differentiable at $(z^k,u^k)\in Y_h\times V_h$ and 
             %Compute a Newton derivative selection $ \mathtt{J}_{\mathtt{F}_h}(z_h^k,u_h^k)\in \mathcal{L}(Y_h\times V_h)$;
        
            \item[(i)] \hypertarget{alg:semi-smooth_newton.i}{} Compute the \emph{primal-dual update direction} $(\delta z_h^k,\delta u_h^k)\in Y_h\times V_h$ such that
            \begin{align}\label{alg:semi-smooth_newton.1}
                \mathtt{J}_{\mathtt{F}_h}(z_h^k,u_h^k)(\delta z_h^k,\delta u_h^k)=-\mathtt{F}_h(z_h^k,u_h^k)\quad \text{ in }Y_h\times V_h\,,
            \end{align}
            and the updated iterate $(z_h^{k+1},u_h^{k+1})\coloneqq(z_h^k,u_h^k)+\alpha_k(\delta z_h^k,\delta u_h^k)\in Y_h\times V_h$, where~${\alpha_k>0}$~is~a (possibly variable) step size;
            \item[(ii)] \hypertarget{alg:semi-smooth_newton.ii}{} If $\|\mathtt{F}_h(z_h^{k+1},u_h^{k+1})\|_{Y_h\times V_h}<\max\{ \varepsilon_{\mathtt{abs}}^h,\varepsilon_{\mathtt{rel}}^h\|\mathtt{F}_h(z_h^0,u_h^0)\|_{Y_h\times V_h}\}$, then \textup{STOP}; otherwise, $k\to k+1$ and continue with Step (\hyperlink{alg:semi-smooth_newton.i}{i}).
        \end{itemize}
    \end{algorithm}\pagebreak

    \begin{remark}[Dual invariance in Algorithm \ref{alg:semi-smooth_newton}]
        Instead of the discrete primal optimality inclusions \eqref{prop:discrete_duality.1}, one could start from the equivalent \emph{discrete dual optimality inclusions}
    \begin{subequations}\label{prop:dual_discrete_duality.1}
           \begin{alignat}{2}\label{prop:dual_discrete_duality.1.1}
        \nabla u_h&\in \partial_t \phi_h^*(\cdot,z_h)&&\quad \text{ a.e.\ in }\Omega\,,\\
       u_h &\in I_{V_h}{\{\partial_s\psi_h^*(\cdot,\operatorname{div}_hz_h)\}}&&\quad \text{ in }\Omega\,.\label{prop:dual_discrete_duality.1.2}
        \end{alignat}
         \end{subequations}
        % which can be regarded as the  optimality inclusions of the discrete dual~problem~\eqref{eq:dual_discrete}. 
        The \hspace{-0.1mm}discrete \hspace{-0.1mm}dual \hspace{-0.1mm}optimality \hspace{-0.1mm}inclusions \hspace{-0.1mm}\eqref{prop:dual_discrete_duality.1}, \hspace{-0.1mm}in \hspace{-0.1mm}turn, \hspace{-0.1mm}for \hspace{-0.1mm}$\gamma_1, \gamma_2\hspace{-0.1em}>\hspace{-0.1em}0$, \hspace{-0.1mm}equivalently~\hspace{-0.1mm}may~\hspace{-0.1mm}be~\hspace{-0.1mm}\mbox{rewritten}~\hspace{-0.1mm}as\vspace{-4.5mm}
         \begin{subequations}\label{sec:dual_semi-smooth_newton.1}
    \begin{alignat}{2}
      z_h+ \gamma_1\nabla u_h&\in (\operatorname{id}_{\mathbb{R}^d}+\gamma_1\partial_t \phi_h^*)(\cdot,z_h)&&\quad \text{ a.e.\ in }\Omega\,,\label{sec:dual_semi-dual_smooth_newton.1.1}\\
     \operatorname{div}_hz_h+  \gamma_2 u_h&\in I_{V_h}\{(\operatorname{id}_{\mathbb{R}}+ \gamma_2\partial_s \psi_h^*)(\cdot,\operatorname{div}_hz_h)\}&&\quad \text{ in }\Omega\,.\label{sec:dual_semi-smooth_newton.1.2}
    \end{alignat}
    \end{subequations}
    Then, by Lemma \ref{lem:resolvent}(\hyperlink{lem:prop_proximity.i}{i}),  the inclusions in \eqref{sec:dual_semi-smooth_newton.1} are each equivalent to the proximal optimality conditions
    \begin{subequations}\label{sec:dual_semi-smooth_newton.2}
    \begin{alignat}{2}\label{sec:dual_semi-smooth_newton.2.1}
      z_h&=\operatorname{prox}_{\gamma_1 \phi_h^*}(\cdot,z_h+\gamma_1\nabla u_h)&&\quad \text{ a.e.\ in }\Omega\,,\\
    \operatorname{div}_hz_h&=I_{V_h}\{\operatorname{prox}_{\gamma_2 \psi_h^*}(\cdot,\operatorname{div}_hz_h+\gamma_2u_h)\}&&\quad \text{  in }\Omega\,,\label{sec:dual_semi-smooth_newton.2.2}
    \end{alignat} 
    \end{subequations}
    respectively. However, by the Moreau decomposition (\textit{cf}.\ Lemma \ref{lem:prop_proximity}(\hyperlink{lem:relations.i}{i})), the identities in \eqref{sec:dual_semi-smooth_newton.2} are each equivalent to
     \begin{subequations}\label{sec:semi-smooth_newton.2}
    \begin{alignat}{2}\label{sec:semi-smooth_newton.2.1}
       \nabla u_h&=\operatorname{prox}_{\smash{\frac{1}{\gamma_1}} \phi_h}(\cdot,\nabla u_h+ \smash{\tfrac{1}{\gamma_1}}z_h)&&\quad \text{ a.e.\ in }\Omega\,,\\
        u_h&=I_{V_h}\{\operatorname{prox}_{\smash{\frac{1}{\gamma_2}} \psi_h}(\cdot,u_h+\smash{\tfrac{1}{\gamma_2}}\operatorname{div}_hz_h)\}&&\quad \text{  in }\Omega\,,\label{sec:semi-smooth_newton.2.2}
    \end{alignat} 
    \end{subequations}
    respectively, so that, up to replacing the regularization parameters %$\gamma_1,\gamma_2>0$
    by their~reciprocals, %~$\frac{1}{\gamma_1},\frac{1}{\gamma_2}>0$, 
    we~arrive~at the same mapping $\mathtt{F}_h\colon \hspace{-0.1em} Y_h\hspace{-0.1em}\times\hspace{-0.1em} V_h\hspace{-0.1em}\to\hspace{-0.1em}  Y_h\hspace{-0.1em}\times \hspace{-0.1em}V_h$, defined by \eqref{sec:semi-smooth_newton.3}, reflecting the hidden dual~invariance~in Algorithm \ref{alg:semi-smooth_newton}  based on the discrete strong duality relation $I_h(u_h)\hspace{-0.15em}=\hspace{-0.15em}D_h(z_h)$~(\textit{cf}.~\mbox{Proposition}~\ref{prop:discrete_duality}(\hyperlink{prop:discrete_duality.ii}{ii})).
    \end{remark}

    \begin{remark}[Alternative stopping criteria]\label{rem:stopping}
        Instead of the \emph{nonlinear residual-based stopping~cri\-terion} $\|\mathtt{F}_h(z_h^{k^*+1},u_h^{k^*+1})\|_{Y_h\times V_h}<\max\{ \varepsilon_{\mathtt{abs}}^h,\varepsilon_{\mathtt{rel}}^h\|\mathtt{F}_h(z_h^0,u_h^0)\|_{Y_h\times V_h}\}$ for some~${k^*\in \mathbb{N}}$~in~Step~(\hyperlink{alg:semi-smooth_newton.ii}{ii}), the following stopping criteria (or combinations thereof) may be employed in~Algorithm~\ref{alg:semi-smooth_newton}:
        \begin{itemize}[noitemsep,topsep=2pt,leftmargin=!,labelwidth=\widthof{(ii)}]
            \item[(i)] \hypertarget{rem:stopping.i}{}\emph{Incremental stopping criterion:} %in Step (\hyperlink{alg:semi-smooth_newton.i}{i}), we \textup{STOP} 
            if $\|(\delta z_h^{k^*},\delta u_h^{k^*})\|_{Y_h\times V_h}<\max\{ \varepsilon_{\mathtt{abs}}^h,\varepsilon_{\mathtt{rel}}^h\|( z_h^{0},u_h^{0})\|_{Y_h\times V_h}\}$ for some $k^*\in \mathbb{N}$ in Step (\hyperlink{alg:semi-smooth_newton.ii}{ii}), then \textup{STOP};
            \item[(ii)] \hypertarget{rem:stopping.ii}{}\emph{Duality-based stopping criterion:}   if $\eta_{\textup{gap},h}^2(u_h^{k^*+1}, z_h^{k^*+1})<\max\{\varepsilon_{\mathtt{abs}}^h,\varepsilon_{\mathtt{rel}}^h\eta_{\textup{gap},h}^2(u_h^0,z_h^0)\}$ for some $k^*\in \mathbb{N}$ in Step (\hyperlink{alg:semi-smooth_newton.ii}{ii}), then \textup{STOP}. 
        \end{itemize}
    \end{remark}

    \begin{remark}[Line-search strategies]\label{rem:linesearch}
In Algorithm~\ref{alg:semi-smooth_newton}, the  step sizes $\alpha_k>0$, $k=0,\ldots,k_{\mathtt{max}}^h$, in Step~(\hyperlink{alg:semi-smooth_newton.i}{i}) can be chosen by a
line-search strategy that is consistent with the residuals listed in Remark~\ref{rem:stopping}. We point out three typical variants:
\begin{itemize}[noitemsep,topsep=2pt,leftmargin=!,labelwidth=\widthof{(iii)}]
\item[(i)] \emph{Residual-based Armijo--Goldstein line-search:} If the nonlinear residual-based~\mbox{criterion}~of Remark~\ref{rem:stopping} is used, we
consider $\Phi_h \coloneqq  \tfrac{1}{2}\| \mathtt{F}_h(\cdot,\cdot) \|_{Y_h\times V_h}^2\colon Y_h\times V_h\to \mathbb{R}_{\ge 0}$, 
where~$\mathtt{F}_h\colon Y_h\times V_h$ $\to Y_h\times V_h $ is  
defined by~\eqref{sec:semi-smooth_newton.3}. Then, for $\beta,\sigma\in (0,1)$, choose the maximal
${\alpha_k\in  \{\beta^{\ell}\mid \ell\in  \mathbb{N}_0\}}$ such that
\begin{align*}
    \Phi_h(z_h^k+\alpha_k \delta z_h^k,\,u_h^k+\alpha_k \delta u_h^k)
  \le (1- 2\sigma \alpha_k)\Phi_h(z_h^k,u_h^k)\,;
\end{align*}  

\item[(ii)] \emph{Increment-based damping:}
If the incremental criterion of Remark~\ref{rem:stopping}(\hyperlink{rem:stopping.i}{i}) is used, then, for $\beta,\theta \in (0,1)$, the step size $\alpha_k>0$ is accepted whenever
\begin{align*}
    \| (\delta z_h^{k+1},\delta u_h^{k+1})\|_{Y_h\times V_h}
  \le \theta \| (\delta z_h^{k},\delta u_h^{k})\|_{Y_h\times V_h}\,;
\end{align*}
otherwise, $\alpha_k \gets \beta \alpha_k$;  

\item[(iii)] \emph{Duality-based line-search:}
If the duality-based criterion of Remark~\ref{rem:stopping}(\hyperlink{rem:stopping.ii}{ii}) is used,
we consider $\Phi_h \coloneqq \eta_{\mathrm{gap},h}^2\colon V_h\times Y_h \to \mathbb{R}_{\ge 0}$, 
where $ \eta_{\mathrm{gap},h}^2\colon V_h\times Y_h\to [0,+\infty]$ is given via \eqref{rem:primal-dual_gap_identity.1}. Then, for $\beta\in (0,1)$, choose the maximal
${\alpha_k\in  \{\beta^{\ell}\mid \ell\in  \mathbb{N}_0\}}$~such~that
\begin{align*}
    \Phi_h(u_h^k+\alpha_k \delta u_h^k,z_h^k+\alpha_k \delta z_h^k)
  \le \Phi_h(u_h^k,z_h^k)\,.
\end{align*}  
\end{itemize}   
\end{remark}\newpage

\begin{remark}[Gradient-flow initializations]
\label{rem:gradient_flow_initializations}
In the case that the energy density $\psi_h\colon \Omega\times \mathbb{R}\to \mathbb{R}\cup\{+\infty\}$ is at most quadratic (\textit{cf}.\ Assumption \ref{ass:energy_densities_discrete}(\hyperlink{ass:energy_densities_discrete.ii.b}{ii.b})), depending on the growth properties of the energy density $\phi_h\colon \Omega \times \mathbb{R}^d\to \mathbb{R}\cup\{+\infty\}$ reliable initial iterates for the $\operatorname{prox}$-based semi-smooth Newton method (\textit{cf}.\ Algorithm \ref{alg:semi-smooth_newton}) may be generated by means of performing a few steps of a
semi-implicit discretized $L^2$-gradient flow. More precisely, with respect to the growth properties of the energy density $\phi_h\colon \Omega \times \mathbb{R}^d\to \mathbb{R}\cup\{+\infty\}$, we distinguish the following two cases:

\medskip
$\bullet$ \emph{Case 1: subquadratic growth.} If there exists a mapping $\widehat\varphi_h\colon \Omega\times \mathbb{R}_{\ge 0}\to \mathbb{R}$ such that 
\begin{align*}
     \varphi_h(x,\cdot)=\widehat\varphi_h(x,|\cdot|)\quad \text{ and }\quad\widehat\varphi_h(x,\cdot)\in C^1(\mathbb{R}_{\ge 0})
    \qquad\text{for a.e. } x\in\Omega\,,
\end{align*}
we employ the following semi-implicit discretized primal $L^2$-gradient flow
algorithm:

\begin{algorithm}[Semi-implicit discretized primal $L^2$-gradient flow]\upshape
\label{alg:primal_gradient_flow}
Let $\varepsilon_{\mathtt{init}}^h>0$ be a stopping parameter, let $\tau>0$ be a step size, 
 let $\widetilde u_h^0\in V_h$ be an initial iterate, and let $\ell_{\mathtt{max}}^h\in\mathbb N\cup\{+\infty\}$ be a maximum number of iterations. Then, for 
$\ell=0,\ldots,\ell_{\mathtt{max}}^h$, perform the~following~iteration~loop:
\begin{itemize}[noitemsep,topsep=2pt,leftmargin=!,labelwidth=\widthof{(ii)}]
\item[(i)] \hypertarget{alg:primal_gradient_flow.i}{} Compute $\widetilde u_h^{\ell+1}\in V_h$ such that for every $v_h\in V_h$, there holds
\begin{align*}\int_{\Omega}{\mathrm{d}_\tau \widetilde u_h^{\ell+1}v_h\,\mathrm{d}x}
    +
    \int_\Omega
     \tfrac{\mathrm{D}_r\widehat\varphi_h(\cdot,|\nabla \widetilde u_h^\ell|)}{|\nabla \widetilde u_h^\ell|}
    \nabla \widetilde u_h^{\ell+1}\cdot \nabla v_h\,\mathrm{d}x
   +
    \int_\Omega
    I_{V_h}\bigl\{
    \mathrm{D}_s\psi_h(\cdot,\widetilde u_h^{\ell+1})v_h
    \bigr\}\,\mathrm{d}x=0\,,
\end{align*}
where $\mathrm{d}_\tau \widetilde u_h^{\ell+1}
    \coloneqq
    \frac{1}{\tau}
    (\widetilde u_h^{\ell+1}-\widetilde u_h^\ell)\in V_h$;

\item[(ii)] \hypertarget{alg:primal_gradient_flow.ii}{} If $\|\mathtt{F}_h(\widetilde z_h^{\ell+1},
                  \widetilde u_h^{\ell+1})\|_{Y_h\times V_h}
   < \varepsilon_{\mathtt{init}}^h$, where $\widetilde z_h^{\ell+1}
    \coloneqq \mathrm{D}_r\widehat\varphi_h(\cdot,|\nabla \widetilde u_h^\ell|)|\nabla \widetilde u_h^\ell|^{-1}  
    \nabla \widetilde u_h^{\ell+1}
    \in Y_h$, then $\mathrm{STOP}$; otherwise, $\ell\mapsto \ell+1$ and continue with Step (\hyperlink{alg:primal_gradient_flow.i}{i}).
\end{itemize}
\end{algorithm}

If the mapping $(r\mapsto\mathrm{D}_r\widehat\varphi_h(x,r)r^{-1})\colon \mathbb{R}_{\ge 0}\to \mathbb{R}_{>0}$ is well-defined, non-increasing,~and~continuous (\textit{cf}.\ \cite[Prop.\ 5.2]{Bartels2021}), Algorithm \ref{alg:primal_gradient_flow} is globally well-posed, unconditionally strongly stable, and terminates after a finite number of iterations.

\medskip
$\bullet$ \emph{Case 2: superquadratic growth.} If there exists a mapping $\widehat\varphi_h^*\colon \Omega\times \mathbb{R}_{\ge 0}\to \mathbb{R}$ such that 
\begin{align*}
     \varphi_h^*(x,\cdot)=\widehat\varphi_h^*(x,|\cdot|)\quad \text{ and }\quad\widehat\varphi_h^*(x,\cdot)\in C^1(\mathbb{R}_{\ge 0})
    \qquad\text{for a.e. } x\in\Omega\,,
\end{align*}
we employ the following semi-implicit discretized dual $L^2$-gradient flow
algorithm:

\begin{algorithm} [Semi-implicit discretized dual $L^2$-gradient flow]\upshape
\label{alg:dual_gradient_flow}
Let $\varepsilon_{\mathtt{init}}^h>0$ be a stopping parameter, let $\tau>0$ be a step size, 
 let $\widetilde z_h^0\in Y_h$ be an initial iterate, and let $\ell_{\mathtt{max}}^h\in\mathbb N\cup\{+\infty\}$ be a maximum number of iterations. Then, for 
$\ell=0,\ldots,\ell_{\mathtt{max}}^h$, perform~the~following~iteration~loop:
\begin{itemize}[noitemsep,topsep=2pt,leftmargin=!,labelwidth=\widthof{(ii)}]
\item[(i)] \hypertarget{alg:dual_gradient_flow.i}{} Compute $\widetilde z_h^{\ell+1}\in Y_h$ such that for every $y_h\in Y_h$, there holds
\begin{align*}
    \int_{\Omega}{\mathrm{d}_\tau \widetilde{z}_h^{\ell+1}\cdot y_h\,\mathrm{d}x} 
    +
    \int_\Omega
    \tfrac{\mathrm{D}_r\widehat\varphi_h^*(\cdot,|\widetilde z_h^\ell|)}{|\widetilde z_h^\ell|} 
    \widetilde z_h^{\ell+1}\cdot y_h\,\mathrm{d}x
    +
    \int_\Omega
    I_{V_h}\bigl\{
    \mathrm{D}_s\psi_h^*(\cdot,\operatorname{div}_h\widetilde z_h^{\ell+1})
    \operatorname{div}_h y_h
    \bigr\}\,\mathrm{d}x=0\,,
\end{align*}
where $\mathrm{d}_\tau \widetilde z_h^{\ell+1}
    \coloneqq
    \frac{1}{\tau}
    (\widetilde z_h^{\ell+1}-\widetilde z_h^\ell)\in Y_h$; 

\item[(ii)] \hypertarget{alg:dual_gradient_flow.ii}{} If $\|\mathtt{F}_h(\widetilde z_h^{\ell+1},
                  \widetilde u_h^{\ell+1})\|_{Y_h\times V_h}
    < \varepsilon_{\mathtt{init}}^h$, where $\widetilde{u}_h^{\ell+1}
    \coloneqq
    I_{V_h}\{
    \mathrm{D}_s\psi_h^*(\cdot,\operatorname{div}_h\widetilde z_h^{\ell+1})\}
    \in V_h$, then $\mathrm{STOP}$; otherwise, $\ell\mapsto \ell+1$ and continue with Step (\hyperlink{alg:dual_gradient_flow.i}{i}).
\end{itemize}
\end{algorithm}

If the mapping $(r\mapsto\mathrm{D}_r\widehat\varphi_h^*(x,r)r^{-1})\colon \mathbb{R}_{\ge 0}\to \mathbb{R}_{>0}$ is well-defined, non-increasing,~and~continuous (\textit{cf}.\ \cite[Prop.\ 5.2]{Bartels2021}), Algorithm \ref{alg:dual_gradient_flow} is globally well-posed, unconditionally strongly stable, and terminates after a finite number of iterations.
\end{remark}\pagebreak

    The following assumptions on the invertibility of the Newton derivative selections \eqref{eq:abbreviations} of the proximity operators $\operatorname{prox}_{\gamma_1 G_h}\colon  Y_h\to  Y_h$ and $\operatorname{prox}_{\gamma_2 F_h}\colon  V_h\to  V_h$ are sufficient, but not necessary (\textit{cf}.\ Subsubsection \ref{subsec:rof_newton}), to ensure the global well-posedness of the~\mbox{$\operatorname{prox}$-based}~\mbox{semi-smooth} Newton method (\textit{cf}.\ Algorithm \ref{alg:semi-smooth_newton}), \textit{i.e.}, if $k_{\mathtt{max}}\hspace{-0.1em}=\hspace{-0.1em}+\infty$, for an arbitrary initial iterate ${(z^0_h,u^0_h)\hspace{-0.1em}\in\hspace{-0.1em} Y_h\hspace{-0.1em}\times\hspace{-0.1em} V_h}$, all iterates $\{(z_h^k,u_h^k)\}_{k\in \mathbb{N}}\subseteq Y_h\times V_h$ are computable.\enlargethispage{5mm}

    \begin{assumption}[for global well-posedness of Algorithm \ref{alg:semi-smooth_newton}]\label{ass:wellposedness}  The discrete energy densities $\phi_h\colon \Omega\times \mathbb{R}^d\to \mathbb{R}\cup\{+\infty\}$ and $\psi_h\colon \Omega\times \mathbb{R}\to \mathbb{R}\cup\{+\infty\}$ are such that Assumption \ref{ass:energy_densities_discrete}  and the following conditions are satisfied:
        \begin{itemize}[noitemsep,topsep=2pt,leftmargin=!,labelwidth=\widthof{(ii)}]
            \item[(i)] \hypertarget{ass:wellposedness.i}{} \emph{Assumptions \hspace{-0.15mm}on \hspace{-0.15mm}$\phi_h$:}
            \hspace{-0.15mm}The \hspace{-0.15mm}proximity \hspace{-0.15mm}operator \hspace{-0.15mm}$\operatorname{prox}_{\gamma_1 G_h}\colon\hspace{-0.15em} Y_h\hspace{-0.15em}\to\hspace{-0.15em}  Y_h$ \hspace{-0.15mm}is~\hspace{-0.15mm}(globally)~\hspace{-0.15mm}Newton~\hspace{-0.15mm}differentia\-ble with a global Newton derivative selection  $\mathtt{J}_{\smash{\operatorname{prox}_{\gamma_1G_h}}}\colon \hspace{-0.1em}Y_h\hspace{-0.1em}\to\hspace{-0.1em} \mathcal{L}(Y_h)$ such that~for~\mbox{every}~${y_h\hspace{-0.1em}\in\hspace{-0.1em} Y_h}$, there holds\footnote{For  
$\mathtt{A},\mathtt{B}\in \mathbb{R}^{d\times d}$, we write $\mathtt{A}\prec \mathtt{B}$ (or $\mathtt{A}\succ \mathtt{B}$) if 
$((\mathtt{A}-\mathtt{B})t)\cdot t< 0$ (or $((\mathtt{A}-\mathtt{B})t)\cdot t> 0$)~for~all~$t\in \mathbb{R}^d$.}
            \begin{align*}
               \mathbbzero_{\mathcal{L}(Y_h)} \prec\mathtt{J}_{\smash{\operatorname{prox}_{\gamma_1G_h}}}(y_h)\prec \mathbbone_{\mathcal{L}(Y_h)}\,;
            \end{align*} 

            \item[(ii)] \hypertarget{ass:wellposedness.ii}{}  \emph{Assumptions on $\psi_h$:}
            \begin{itemize}[noitemsep,topsep=2pt,leftmargin=!,labelwidth=\widthof{(ii.a)}]
                \item[(ii.a)] \hypertarget{ass:wellposedness.ii.a}{} \emph{Mass lumping:} If (\hyperlink{ML}{ML}) applies, the proximity operator  $\operatorname{prox}_{\gamma_2 F_h}\colon V_h\to  V_h$ is (globally) Newton differentiable with a global Newton derivative selection $\mathtt{J}_{\smash{\operatorname{prox}_{\gamma_2 F_h}}}\colon  V_h\to  \mathcal{L}(V_h)$ such that for every $v_h\in  V_h$, there holds
                \begin{align*}
               \mathbbzero_{\mathcal{L}(V_h)} \prec\mathtt{J}_{\smash{\operatorname{prox}_{\gamma_2F_h}}}(v_h)\,;
            \end{align*} 
            
                \item[(ii.b)] \hypertarget{ass:wellposedness.ii.b}{} \emph{No mass lumping:} If (\hyperlink{NML}{NML}) applies, on the basis of Proposition \ref{prop:proximity_of_integral_functionals}\eqref{prop:proximity_of_integral_functionals.3},~as~a~Newton derivative selection, we employ $\mathtt{J}_{\smash{\operatorname{prox}_{\gamma_2F_h}}}\coloneqq \frac{1}{1+2\gamma_2 a_h}\mathbbone_{\mathcal{L}(V_h)}\colon V_h\to  \mathcal{L}(V_h)$.
            \end{itemize}
        \end{itemize}
     \end{assumption}

    One set of sufficient conditions for Assumption \ref{ass:wellposedness}, relying on 
    $C^2$-regularity of the discrete energy densities, is presented in the following remark. Further conditions are given~in~Remark~\ref{rem:convergence}.
    
    \begin{remark}[on Assumption \ref{ass:wellposedness}] 
     \begin{itemize}[noitemsep,topsep=2pt,leftmargin=!,labelwidth=\widthof{(ii)}]
         \item[(i)] \emph{Sufficient conditions for Assumption \ref{ass:wellposedness}(\hyperlink{ass:wellposedness.i}{i}):} If~$\phi_h\in C^2(\mathbb{R}^d)$, then, by Lemma \ref{lem:prop_proximity}(\hyperlink{lem:prop_proximity.iv}{iv}), 
         $\mathrm{D}\!\operatorname{prox}_{\gamma_1\phi_h}(t)=(\mathbbone+\mathrm{D}^2 \gamma_1\phi_h(\operatorname{prox}_{\gamma_1\phi_h}(t)))^{-1}$~for~all $t\in \mathbb{R}^d$. Therefore, %since $\phi_h$ has locally bounded curvature, 
          $\mathrm{D}\!\operatorname{prox}_{\gamma_1\phi_h}(t) \succ \mathbbzero$ for all $t\in \mathbb{R}^d$ and, thus, $\mathtt{J}_{\smash{\operatorname{prox}_{\gamma_1G_h}}}(y_h)\succ \mathbbzero_{\mathcal{L}(Y_h)}$~for~all~${y_h\in Y_h}$.
         
         If, in addition, $\phi_h$ is locally strongly convex, \textit{i.e.}, we have that $\mathrm{D}^2 \phi_h(t)\succ \mathbbzero$ for all $t\in \mathbb{R}^d$, then $\mathrm{D}\!\operatorname{prox}_{\gamma_1\phi_h}(t)\prec \mathbbone$ for all $t\in \mathbb{R}^d$ and, thus, 
            $\mathtt{J}_{\smash{\operatorname{prox}_{\gamma_1G_h}}}(y_h) \prec \mathbbone_{\mathcal{L}(Y_h)}$ for~all~${y_h\in Y_h}$.
            
          \item[(ii)] \emph{Sufficient conditions for Assumption \ref{ass:wellposedness}(\hyperlink{ass:wellposedness.ii}{ii.a}):} If $\psi_h\in C^2(\mathbb{R})$, then, by Lemma \ref{lem:prop_proximity}(\hyperlink{lem:prop_proximity.iv}{iv}), %we have that  
          $\mathrm{D}\!\operatorname{prox}_{\gamma_2\psi_h}(s)> 0$ for all $s\in  \mathbb{R}$ and, thus, $\mathtt{J}_{\smash{\operatorname{prox}_{\gamma_2F_h}}}(v_h)\succ\mathbbzero_{\mathcal{L}(V_h)}$~for~all~${v_h\in V_h}$.

     \end{itemize} 
   \end{remark}

     Finally, given Assumption  \ref{ass:wellposedness}, let us establish that Algorithm~\ref{alg:semi-smooth_newton} is globally well-posed.

    \begin{theorem}[Global well-posedness of Algorithm \ref{alg:semi-smooth_newton}]\label{prop:well-posedness}
        If Assumption \ref{ass:wellposedness} is satisfied, then Algorithm~\ref{alg:semi-smooth_newton} is globally well-posed, \textit{i.e.}, if $k_{\mathtt{max}}^h=+\infty$, given an arbitrary initial iterate $(z_h^0,u_h^0)\in Y_h\times V_h$, all iterates 
        $\{(z_h^{k},u_h^{k})\}_{k\in \mathbb{N}}\subseteq Y_h\times V_h$ are computable.
    \end{theorem} 

    \begin{proof}
        In order to prove the existence of primal-dual update directions $(\delta z_h^{k},\delta u_h^{k})\in Y_h\times V_h$, solving \eqref{alg:semi-smooth_newton.1}, for each $k\in \mathbb{N}$, it is sufficient to establish the invertibility of the Newton derivative selection $\mathtt{J}_{\mathtt{F}_h}(z_h^k,u_h^k)\in \mathcal{L}(Y_h\times V_h)$. 
        To this end, let %$k\in \{0,\ldots,k_{\mathtt{max}}^h\}$~
       $(y_h,v_h)\in Y_h\times V_h$~be~fixed,~but~arbitrary. 
        To~begin~with, employing the abbreviated notations
        \begin{align}\label{prop:well-posedness.1}
            \mathtt{J}_h^1&\coloneqq\mathtt{J}_{\smash{\operatorname{prox}_{\gamma_1 G_h}}}(\nabla v_h+ \gamma_1y_h)\in \mathcal{L}(Y_h)\,,\\
            \mathtt{J}_h^2&\coloneqq \mathtt{J}_{\smash{\operatorname{prox}_{\gamma_2 F_h}}}(v_h+\gamma_2\operatorname{div}_hy_h)\in \mathcal{L}(V_h)\,,
        \end{align}
        on the basis of Assumption \ref{ass:wellposedness}, which guarantees that the operators $\mathbbone_{\mathcal{L}(Y_h)}-\mathtt{J}_h^1\in \mathcal{L}(Y_h)$ and $\mathtt{J}_h^2\in \mathcal{L}(V_h)$ are both positive definite and, consequently, invertible with positive definite inverse, we introduce the invertible block diagonal operator
        \begin{align}
         \mathtt{T}_h\coloneqq\left[\begin{array}{cc}
            -(\mathbbone_{\mathcal{L}(Y_h)} - \mathtt{J}_h^1)^{-1} & \mathbbzero_{\mathcal{L}(V_h)}\\\mathbbzero_{\mathcal{L}(Y_h)}
            &  -\tfrac{1}{\gamma_2}(\mathtt{J}_h^2)^{-1}
        \end{array}\right]\in  \mathcal{L}(Y_h\times V_h)\,.\label{prop:well-posedness.2}
        \end{align}
        By means of this invertible operator, we observe that the invertibility of the Newton derivative selection $\mathtt{J}_{\mathtt{F}_h}(y_h,v_h)\in \mathcal{L}(Y_h\times V_h)$ is equivalent to  the invertibility of the $2\times2$-block matrix
        \begin{align}\label{prop:well-posedness.3}
        \begin{aligned}
            \mathtt{M}_h&\coloneqq\left[\begin{array}{cc}
           \mathtt{A}_h & \mathtt{B}_h\\ \mathtt{C}_h
            & \mathtt{D}_h
        \end{array}\right]\\[-0.5mm]&\coloneqq\left[\begin{array}{cc}
          \smash{\gamma_1(\mathbbone_{\mathcal{L}(Y_h)} - \mathtt{J}_h^1)^{-1}\mathtt{J}_h^1} & -\nabla_h\\ \operatorname{div}_h 
            &\smash{-\tfrac{1}{\gamma_2}(\mathtt{J}_h^2)^{-1}(\mathbbone_{\mathcal{L}(V_h)} - \mathtt{J}_h^2)}
        \end{array}\right]\in  \mathcal{L}(Y_h\times V_h)\,,
        \end{aligned}
        \end{align}
        where the single blocks have the following properties:
        \begin{itemize}[noitemsep,topsep=2pt,leftmargin=!,labelwidth=\widthof{$\bullet$}]
            \item[$\bullet$] the $(1,1)$-block $ \mathtt{A}_h\in  \mathcal{L}(Y_h)$ is positive definite (\textit{cf}.\ Assumption \ref{ass:wellposedness}(\hyperlink{ass:wellposedness.i}{i}));
            \item[$\bullet$] the $(1,2)$-block $ \mathtt{B}_h\in \mathcal{L}(V_h;Y_h)$ is injective;
            \item[$\bullet$] the $(2,1)$-block $ \mathtt{C}_h\in \mathcal{L}(Y_h;V_h)$ is surjective with $\mathtt{C}_h=\mathtt{B}_h^*$;
            \item[$\bullet$] the $(2,2)$-block $ \mathtt{D}_h\in \mathcal{L}(V_h)$ is negative semi-definite (\textit{cf}.\ Assumption \ref{ass:wellposedness}(\hyperlink{ass:wellposedness.ii}{ii})).
        \end{itemize}
        The Newton derivative selection $\mathtt{J}_{\mathtt{F}_h}(y_h,v_h)\hspace{-0.175em}\in\hspace{-0.175em} \mathcal{L}(Y_h\times V_h)$ and the $2\times2$-block~matrix~${\mathtt{M}_h\hspace{-0.175em}\in \hspace{-0.175em} \mathcal{L}(Y_h\hspace{-0.1em}\times\hspace{-0.1em} V_h)}$, defined by \eqref{prop:well-posedness.3}, are related via\vspace{-0.5mm} 
        \begin{align}\label{prop:well-posedness.4}
            \mathtt{M}_h=\mathtt{T}_h\mathtt{J}_{\mathtt{F}_h}(y_h,v_h)\quad\text{ in }\mathcal{L}(Y_h\times V_h)\,.
        \end{align}
        %Using these two operators, we find that
        %\begin{align*}
        %    \textup{det}(\mathtt{J}_{\mathtt{F}_h}(u_h^k,z_h^k))=\textup{det}\left(\left[\begin{array}{cc}
        %   \mathtt{A}_h & -\nabla\\ \operatorname{div}_h 
        %    &-\mathtt{C}_h
        %\end{array}\right]\right)\,.
        %\end{align*}
        By the standard theory of invertibility of block matrices (\textit{cf}.\ \cite[Thm.\ 2.1]{LuShiou2002}), the $2\times 2$-block~matrix $\mathtt{M}_h\in  \mathcal{L}(Y_h\times V_h)$ (and, thus, $\mathtt{J}_{\mathtt{F}_h}(y_h,v_h)\in \mathcal{L}(Y_h\times V_h)$) is invertible if and only if the $(1,1)$-block $\mathtt{A}_h \in \mathcal{L}(Y_h)$ as well as the Schur complement $\mathtt{M}_h/\mathtt{A}_h\coloneqq\mathtt{D}_h-\mathtt{C}_h\mathtt{A}_h^{-1}\mathtt{B}_h\in \mathcal{L}(V_h)$ are~\mbox{invertible};~which is a consequence of the following LDU factorization of the $2\times 2$-block matrix $\mathtt{M}_h\in  \mathcal{L}(Y_h\times V_h)$:
        \begin{align}\label{prop:well-posedness.5}
        \begin{aligned} 
           \mathtt{M}_h &=\left[\begin{array}{cc}
           \mathbbone_{\mathcal{L}(Y_h)}  & \mathbbzero_{\mathcal{L}(V_h)}\\ 
         \mathtt{C}_h\mathtt{A}_h^{\smash{-1}}   & \mathbbone_{\mathcal{L}(V_h)}
        \end{array}\right]\left[\begin{array}{cc}
           \mathtt{A}_h  & \mathbbzero_{\mathcal{L}(V_h)}\\ \mathbbzero_{\mathcal{L}(Y_h)}
            & \mathtt{M}_h/\mathtt{A}_h
        \end{array}\right]\left[\begin{array}{cc}
           \mathbbone_{\mathcal{L}(Y_h)}  & \mathtt{A}_h^{\smash{-1}}\mathtt{B}_h\\ 
           \mathbbzero_{\mathcal{L}(Y_h)} & \mathbbone_{\mathcal{L}(V_h)}
        \end{array}\right]\\[-0.5mm]
        &\eqqcolon \mathtt{L}_h\mathtt{D}_h\mathtt{U}_h \in \mathcal{L}(Y_h\times V_h)\,.
        \end{aligned}
        \end{align}
        On the one hand, inasmuch as the $(1,1)$-block $\mathtt{A}_h \in \mathcal{L}(Y_h)$ is positive definite, it is invertible~as~well. On the other hand, the Schur complement 
        $\mathtt{M}_h/\mathtt{A}_h \in \mathcal{L}(V_h)$ is negative definite~and,~thus,~invertible, since  
        the $(2,2)$-block~$\mathtt{D}_h\in \mathcal{L}(V_h)$ is negative semi-definite and  $\mathtt{C}_h\mathtt{A}_h^{-1}\mathtt{B}_h\in \mathcal{L}(V_h)$ is positive definite,
        which follows from the positive definiteness of 
        $\mathtt{A}_h^{-1} \in   \mathcal{L}(Y_h)$,  %of the $(1,1)$-block, 
        the~injectivity~of~${\mathtt{B}_h\in  \mathcal{L}(V_h;Y_h)}$,~and that, by the definition of the discrete divergence operator \eqref{def:discrete_divergence}, we~have~that~$\mathtt{C}_h=\mathtt{B}_h^*$.
    \end{proof}

    The following assumption ensures the local well-posedness and local super-linear convergence~of Algorithm \ref{alg:semi-smooth_newton}.\vspace{-0.5mm}\enlargethispage{5.5mm}

    \begin{assumption}[for local well-posedness and local super-linear convergence of Algorithm~\ref{alg:semi-smooth_newton}]\label{ass:convergence} The discrete energy densities $\phi_h\colon \Omega\times \mathbb{R}^d\to \mathbb{R}\cup\{+\infty\}$ and $\psi_h\colon \Omega\times \mathbb{R}\to \mathbb{R}\cup\{+\infty\}$ are such that Assumption \ref{ass:energy_densities_discrete} and the following conditions are satisfied: 
        \begin{itemize}[noitemsep,topsep=2pt,leftmargin=!,labelwidth=\widthof{(iii)}]
            \item[(i)] \hypertarget{ass:wellposedness.i}{} \emph{Assumptions on $\phi_h$:}
            There exist  $0<\smash{\underline{\mu}}_h^1\leq \smash{\overline{\mu}}_h^1<1$ such that 
            for every $y_h\in B_{\varepsilon_h}^{Y_h}(\nabla u_h +\gamma_1z_h)$, there holds\vspace{-1.5mm}
            \begin{align*}
                \smash{\smash{\underline{\mu}}_h^1\mathbbone_{\mathcal{L}(Y_h)}\preceq  \mathtt{J}_{\smash{\operatorname{prox}_{\gamma_1 G_h}}}(y_h) \preceq \smash{\overline{\mu}}_h^1\mathbbone_{\mathcal{L}(Y_h)}\,;}
            \end{align*}

            \item[(ii)] \hypertarget{ass:wellposedness.ii}{}  \emph{Assumptions on $\psi_h$:}
            \begin{itemize}[noitemsep,topsep=2pt,leftmargin=!,labelwidth=\widthof{(ii.a)}]
                \item[(ii.a)] \hypertarget{ass:wellposedness.ii.a}{} \emph{Mass lumping:} There exist  $0<\underline{\mu}_h^2\leq 1$ such that for every $v_h\in  B_{\varepsilon_h}^{V_h}(u_h +\gamma_2\operatorname{div}_hz_h)$, there holds\vspace{-1.5mm}
            \begin{align*}
                \smash{\underline{\mu}_h^2\mathbbone_{\mathcal{L}(V_h)}\preceq\mathtt{J}_{\smash{\operatorname{prox}_{\gamma_2 F_h}}}(v_h)}\,;
            \end{align*}
        
                \item[(ii.b)] \hypertarget{ass:wellposedness.ii.b}{} \emph{No mass lumping:} If (\hyperlink{NML}{NML}) applies, on the basis of Proposition \ref{prop:proximity_of_integral_functionals}\eqref{prop:proximity_of_integral_functionals.3},~as~a~Newton derivative selection, we employ $\mathtt{J}_{\smash{\operatorname{prox}_{\gamma_2 F_h}}}\coloneqq \smash{\frac{1}{1+2\gamma_2 a_h}}\mathbbone_{\mathcal{L}(V_h)}\colon V_h\to  \mathcal{L}(V_h)$. 
            \end{itemize}
        \end{itemize} 
     \end{assumption}
     
     \begin{remark}[on Assumption \ref{ass:convergence}] \label{rem:convergence}
     \begin{itemize}[noitemsep,topsep=2pt,leftmargin=!,labelwidth=\widthof{(ii)}]
         \item[(i)] \emph{Sufficient conditions for Assumption \ref{ass:convergence}(\hyperlink{ass:wellposedness.i}{i}):}
         \begin{itemize}[noitemsep,topsep=2pt,leftmargin=!,labelwidth=\widthof{(ii.a)}]
            \item[(i.a)] \emph{Lower bound.} If $\phi_h$ is given via a Moreau envelope of $\phi$ (with respect to the second argument), \textit{i.e.}, for some $\gamma>0$, there holds $\phi_h(x,\cdot)\coloneqq (\phi(x,\cdot))^\gamma$ for a.e.\ $x\in \Omega$, then, by Lemma~\ref{lem:relations}(\hyperlink{lem:relations.iii}{iii}), for every $y_h\in Y_h$, we have that $\mathtt{J}_{\smash{\operatorname{prox}_{\gamma_1 G_h}}}(y_h)\succeq \frac{\gamma}{\gamma+\gamma_1}\mathbbone_{\mathcal{L}(Y_h)}$.

            \item[(i.b)] \emph{Upper bound.} If $\phi_h$ is $\mu$-strongly convex for some $\mu>0$ (with respect to the second argument), by Lemma \ref{lem:resolvent}(\hyperlink{lem:resolvent.iii}{iii}), for every $y_h\hspace{-0.1em}\in\hspace{-0.1em} Y_h$, we have that ${\mathtt{J}_{\smash{\operatorname{prox}_{\gamma_1 G_h}}}(y_h)\hspace{-0.1em}\preceq\hspace{-0.1em} \smash{\frac{1}{1+\mu}}\mathbbone_{\mathcal{L}(Y_h)}}$. 
         \end{itemize}
          \item[(ii)] \emph{Sufficient conditions for Assumption \ref{ass:convergence}(\hyperlink{ass:wellposedness.ii}{ii.a}):}
          If $\psi_h$ is given via a Moreau envelope of $\psi$ (with respect to the second argument), \textit{i.e.}, for some $\gamma>0$, there holds $\psi_h(x,\cdot)\coloneqq (\psi(x,\cdot))^\gamma$ for a.e.\ $x\in \Omega$, then, by Lemma~\ref{lem:relations}(\hyperlink{lem:relations.iii}{iii}), for every $v_h\in V_h$,  we have that $\mathtt{J}_{\smash{\operatorname{prox}_{\gamma_2 F_h}}}(v_h)\succeq \frac{\gamma}{\gamma+\gamma_2}\mathbbone_{\mathcal{L}(V_h)}$.
         
     \end{itemize} 
   \end{remark}

    Finally, given Assumption  \ref{ass:convergence}, let us establish that Algorithm~\ref{alg:semi-smooth_newton} is locally  well-posed~and  locally super-linearly convergent.\enlargethispage{7mm}

    \begin{theorem}[Local well-posedness and local super-linear convergence of Algorithm \ref{alg:semi-smooth_newton}]\label{prop:convergence}
        If Assumption \ref{ass:convergence} is satisfied, then Algorithm~\ref{alg:semi-smooth_newton} (with $\alpha_k=1$ for all $k\in \mathbb{N}$) is locally well-posed and locally super-linearly convergent to a root $(z_h,u_h)\in Y_h\times V_h$.
    \end{theorem} 

    \begin{proof}
        To begin with, we note that, analogously to the proof of Theorem \ref{prop:well-posedness},~we~observe~that, %Algorithm~\ref{alg:semi-smooth_newton} is well-posed and, 
        for every $(y_h,v_h)\in B_{\varepsilon_h}^{Y_h\times V_h}(z_h,u_h)$, the Newton derivative selection $\mathtt{J}_{\mathtt{F}_h}(y_h,v_h)\in \mathcal{L}(Y_h\times V_h)$ is invertible. Apart from that,
        employing the same notation as in the proof of Theorem \ref{prop:well-posedness}, for every $(y_h,v_h)\in B_{\varepsilon_h}^{Y_h\times V_h}(z_h,u_h)$, 
        its inverse can be represented by\vspace{-0.5mm}
        \begin{align}\label{prop:convergence.1}
           \smash{\mathtt{J}_{\mathtt{F}_h}(y_h,v_h)^{-1}=  \mathtt{M}_h^{-1}\mathtt{T}_h\,,}
        \end{align}
        where %the operators 
        $\mathtt{M}_h,\mathtt{T}_h\in  \mathcal{L}(Y_h\times V_h)$  depend on  $(y_h,v_h)\in B_{\varepsilon_h}^{Y_h\times V_h}(z_h,u_h)$ (\textit{cf}.\ \eqref{prop:well-posedness.2},~\eqref{prop:well-posedness.3}),~so~that 
        \begin{align}\label{prop:convergence.2}
            \smash{\|\mathtt{J}_{\mathtt{F}_h}(y_h,v_h)^{-1}\|_{\mathcal{L}(Y_h\times V_h)}\leq \|\mathtt{M}_h^{-1}\|_{\mathcal{L}(Y_h\times V_h)}\|\mathtt{T}_h\|_{\mathcal{L}(Y_h\times V_h)}\,.}
        \end{align}
        
        Therefore, it only remains to establish that the operator norms on the right-hand side of \eqref{prop:convergence.2} can be estimated independently of $(y_h,v_h)\in B_{\varepsilon_h}^{Y_h\times V_h}(z_h,u_h)$:
        \begin{itemize}[noitemsep,topsep=2pt,leftmargin=!,labelwidth=\widthof{$\bullet$}]
            \item[$\bullet$] By Assumption \ref{ass:convergence}, for every $(y_h,v_h)\in B_{\varepsilon_h}^{Y_h\times V_h}(z_h,u_h)$, we have that\vspace{-0.5mm} 
        \begin{align}\label{prop:convergence.3}
            \begin{aligned} 
            \|\mathtt{T}_h\|_{\mathcal{L}(Y_h\times V_h)}&\leq \max\{\|(\mathbbone_{\mathcal{L}(Y_h)}-\mathtt{J}_h^1)^{-1}\|_{\mathcal{L}(Y_h)},\tfrac{1}{\gamma_2}\|(\mathtt{J}_h^2)^{-1}\|_{\mathcal{L}(V_h)}\}
            \\[-0.5mm]&\leq \max\smash{\big\{\tfrac{1}{1-\smash{\overline{\mu}}^1_h},\tfrac{1}{\gamma_2\underline{\mu}^2_h}\big\}}\,.
            \end{aligned}
        \end{align}

            \item[$\bullet$] 
            Due to the representation theorem for inverses of $2\times 2$-block matrices (\textit{cf}.\ \cite[Thm.\ 2.1]{LuShiou2002}), which is based on the LDU factorization \eqref{prop:well-posedness.5}, for every $(y_h,v_h)\in B_{\varepsilon_h}^{Y_h\times V_h}(z_h,u_h)$,~we~have~that\vspace{-0.5mm}
        \begin{align}\label{prop:convergence.4}
           \hspace{-2mm} \begin{aligned} \mathtt{M}_h^{-1}&=\mathtt{U}_h^{-1}\mathtt{D}_h^{-1}\mathtt{L}_h^{-1}
            \\[-0.5mm]&=\left[\begin{array}{cc}
           \mathbbone_{\mathcal{L}(Y_h)}  & -\mathtt{A}_h^{-1}\mathtt{B}_h\\ 
         \mathbbzero_{\mathcal{L}(Y_h)}   & \mathbbone_{\mathcal{L}(V_h)}
        \end{array}\right]\hspace{-1mm}\left[\begin{array}{cc}
           \mathtt{A}_h^{-1}  & \mathbbzero_{\mathcal{L}(V_h)}\\ \mathbbzero_{\mathcal{L}(Y_h)}
            & (\mathtt{M}_h/\mathtt{A}_h)^{-1}
        \end{array}\right]\hspace{-1mm}\left[\begin{array}{cc}
           \mathbbone_{\mathcal{L}(Y_h)}  & \mathbbzero_{\mathcal{L}(V_h)}\\ 
         -\mathtt{C}_h\mathtt{A}_h^{-1}   & \mathbbone_{\mathcal{L}(V_h)}
        \end{array}\right]%\\&=
       % \left[\begin{array}{cc}
       %    \mathtt{A}_h^{-1}- \mathtt{A}_h^{-1}\mathtt{B}_h(\mathtt{M}_h/\mathtt{A}_h)^{-1}\mathtt{C}_h\mathtt{A}_h^{-1}  & - \mathtt{A}_h^{-1}\mathtt{B}_h(\mathtt{M}_h/\mathtt{A}_h)^{-1}\\ -(\mathtt{M}_h/\mathtt{A}_h)^{-1}\mathtt{C}_h\mathtt{A}_h^{-1}
      %      & (\mathtt{M}_h/\mathtt{A}_h)^{-1}
      %  \end{array}\right]
      \in \mathcal{L}(Y_h\times V_h)\,,
        \end{aligned}\hspace{-1mm}
        \end{align}
        so that\vspace{-0.5mm}
        \begin{align}\label{prop:convergence.5}
            \smash{\|\mathtt{M}_h^{-1}\|_{\mathcal{L}(Y_h\times V_h)}\leq \|\mathtt{U}_h^{-1}\|_{\mathcal{L}(Y_h\times V_h)}\|\mathtt{D}_h^{-1}\|_{\mathcal{L}(Y_h\times V_h)}\|\mathtt{L}_h^{-1}\|_{\mathcal{L}(Y_h\times V_h)}\,,}
        \end{align}
        where\vspace{-0.5mm}
\begin{subequations}\label{prop:convergence.6} 
        \begin{alignat}{2}\label{prop:convergence.6.1} 
            &\begin{aligned}
            \|\mathtt{U}_h^{-1}\|_{\mathcal{L}(Y_h\times V_h)}&\leq 1+\|\mathtt{A}_h^{-1}\|_{\mathcal{L}(Y_h)}\|\mathtt{B}_h\|_{\mathcal{L}(V_h;Y_h)}
            \\[-0.5mm]&\leq 1 +\tfrac{2}{\gamma_1\smash{\underline{\mu}}_h^1}\|\mathtt{B}_h\|_{\mathcal{L}(V_h;Y_h)}\,;\end{aligned}&&\quad\Bigg\}\\[-0.5mm]&\label{prop:convergence.6.2} \begin{aligned}
            \|\mathtt{D}_h^{-1}\|_{\mathcal{L}(Y_h\times V_h)}&\leq \max\{\|\mathtt{A}_h^{-1}\|_{\mathcal{L}(Y_h)},\|(\mathtt{M}_h/\mathtt{A}_h)^{-1}\|_{\mathcal{L}(V_h)}\}
            \\&\leq \max\big\{ \tfrac{2}{\gamma_1\smash{\underline{\mu}}_h^1}, \tfrac{\gamma_1\smash{\smash{\overline{\mu}}_h^1}}{\sigma_{\mathtt{min}}^2(\mathtt{B}_h)(1-\smash{\overline{\mu}}_h^1)}\big\}
            \,;\end{aligned}&&\quad\Bigg\}\\[-0.5mm]&\label{prop:convergence.6.3}\begin{aligned}
            \|\mathtt{L}_h^{-1}\|_{\mathcal{L}(Y_h\times V_h)}&\leq 1+\|\mathtt{C}_h\|_{\mathcal{L}(Y_h;V_h)}\|\mathtt{A}_h^{-1}\|_{\mathcal{L}(Y_h)}\\&
            \leq 1 +\tfrac{2}{\gamma_1\smash{\underline{\mu}}_h^1}\|\mathtt{B}_h\|_{\mathcal{L}(V_h;Y_h)}\,.\end{aligned}&&\quad\Bigg\}
        \end{alignat}
         \end{subequations}
         In the second inequality in \eqref{prop:convergence.6.2}, we used that $\vert \lambda_{\mathtt{min}}(\mathtt{M}_h/\mathtt{A}_h)\vert \ge \lambda_{\mathtt{min}}(\mathtt{D}_h)+\lambda_{\mathtt{min}}(\mathtt{C}_h\mathtt{A}_h^{-1}\mathtt{B}_h)$ and %that
        $\lambda_{\mathtt{min}}(\mathtt{C}_h\mathtt{A}_h^{-1}\mathtt{B}_h)\hspace{-0.15em}\ge\hspace{-0.15em} \sigma_{\mathtt{min}}^2(\mathtt{B}_h)\lambda_{\mathtt{min}}(\mathtt{A}_h)$, where $\sigma_{\mathtt{min}}(\mathtt{B}_h)$ is the smallest~singular~value~of~${\mathtt{B}_h\hspace{-0.15em}\in\hspace{-0.15em} \mathcal{L}(V_h;Y_h)}$.
        \end{itemize}

        In summary, noting that  $\|\mathtt{B}_h\|_{\mathcal{L}(V_h;Y_h)},\sigma_{\mathtt{min}}(\mathtt{B}_h)$ are independent of $(y_h,v_h)\in B_{\varepsilon_h}^{Y_h\times V_h}(z_h,u_h)$, from \eqref{prop:convergence.2}--\eqref{prop:convergence.6.3}, it follows that\vspace{-1mm}
        \begin{align*}
            \sup_{(y_h,v_h)\in B_{\varepsilon_h}^{Y_h\times V_h}(z_h,u_h)}{\big\{\|\mathtt{J}_{\mathtt{F}_h}(y_h,v_h)^{-1}\|_{\mathcal{L}(Y_h\times V_h)}\big\}}<+\infty\,.\\[-6mm]
        \end{align*}
        Hence, the local well-posedness and  super-linear convergence result for semi-smooth Newton methods 
        (\textit{cf}.\ Theorem \ref{thm:wellposed_convergent}) is applicable and yields that the $\operatorname{prox}$-based semi-smooth Newton~method (\textit{cf}.\ Algorithm \ref{alg:semi-smooth_newton}) is locally well-posed and super-linearly convergent~to~a~root~${(z_h,u_h)\in Y_h\times V_h}$.
    \end{proof}

    \begin{remark}[Reduced formulation of primal-dual update direction computation]\label{rem:reduced_problem}
        The perturbed saddle point problem \eqref{alg:semi-smooth_newton.1}, which, for every $k\in\{0,\ldots,k_{\mathtt{max}}^h\}$, computes the primal-dual update direction $(\delta z_{h}^{k},\delta u_{h}^{k})\in Y_{h}\times V_{h}$ from~the~most~recent iterate $(z_{h}^{k},u_{h}^{k})\in Y_{h}\times V_{h}$ can be reduced to a linear elliptic problem:

        In fact, for every $k\in\{0,\ldots,k_{\mathtt{max}}^h\}$, employing the abbreviations
        \begin{align*}
            \smash{\mathtt{J}_{h}^{1,k}}&\coloneqq \mathtt{J}_{\smash{\operatorname{prox}_{\gamma_1G_h}}}(\nabla u_{h}^k+\gamma_1z_{h}^k)\in \mathcal{L}(Y_h)\,,\\
            \smash{\mathtt{J}_{h}^{2,k}}&\coloneqq \mathtt{J}_{\smash{\operatorname{prox}_{\gamma_2F_h}}}( u_{h}^k+\gamma_2\operatorname{div}_hz_{h}^k)\in \mathcal{L}(V_h)\,,
        \end{align*}
        the perturbed saddle point system \eqref{alg:semi-smooth_newton.1} reads %(again, abbreviating $a_h\coloneqq \nabla v_h+\gamma y_h\in Y_{h}$ as in \eqref{def:J_F_h_eps})\vspace{-4.5mm}
    \begin{subequations} \label{rem:reduced.1}
    \begin{alignat}{2}\label{rem:reduced.1.1}
       -\gamma_1\smash{\mathtt{J}_{h}^{1,k}} \delta z_{h}^k+(\mathbbone_{\mathcal{L}(Y_h)} - \smash{\mathtt{J}_{h}^{1,k}}) \nabla \delta u_{h}^k&=-\nabla u_{h}^k+\operatorname{prox}_{\gamma_1G_h}(\nabla u_{h}^k+\gamma_1 z_{h}^k)&&\quad\text{ in }Y_h\,,\\
       -\gamma_2 \smash{\mathtt{J}_{h}^{2,k}}\operatorname{div}_h  \delta z_{h}^k+(\mathbbone_{\mathcal{L}(V_h)} - \smash{\mathtt{J}_{h}^{2,k}})  \delta u_{h}^k &= -u_{h}^k+\operatorname{prox}_{\gamma_2F_h}(u_{h}^k+\gamma_2 \operatorname{div}_{h}z_{h}^k)&&\quad\text{ in }V_h\,,\label{rem:reduced.1.2}
    \end{alignat}
    \end{subequations}
    which, using the abbreviations
    \begin{align*}
            f_{h}^{1,k}&\coloneqq- \nabla u_{h}^k+\operatorname{prox}_{\gamma_1G_h}(\nabla u_{h}^k+\gamma_1 z_{h}^k)\in Y_h\,,\\
            f_{h}^{2,k}&\coloneqq -u_{h}^k+\operatorname{prox}_{\gamma_2F_h}(u_{h}^k+\gamma_2 \operatorname{div}_{h}z_{h}^k)\in V_h\,,
        \end{align*}
    and assuming that $ \smash{\mathtt{J}_{h}^{1,k}}\succ \mathbbzero_{\mathcal{L}(Y_h)}$ and $\smash{\mathtt{J}_{h}^{2,k}}\succ \mathbbzero_{\mathcal{L}(V_h)}$,
    %\begin{align*}
    %    0\prec \smash{\mathtt{J}_{h}^{1,k}}\quad\text{and}\quad0< \smash{\mathtt{J}_{h}^{2,k}}\,,
    %\end{align*} 
    may be reformulated as
    \begin{subequations} \label{rem:reduced.2}
    \begin{alignat}{2}\label{rem:reduced.2.1}
      \delta z_{h}^k&=\tfrac{1}{\gamma_1}(\smash{\mathtt{J}_{h}^{1,k}})^{-1}\{(\mathbbone_{\mathcal{L}(Y_h)} - \smash{\mathtt{J}_{h}^{1,k}}) \nabla \delta u_{h}^k-f_h^{1,k}\}&&\quad\text{ in }Y_h\,,\\
      \operatorname{div}_h  \delta z_{h}^k&=I_{V_h}\bigl\{\tfrac{1}{\gamma_2}(\smash{\mathtt{J}_{h}^{2,k}})^{-1}\{(\mathbbone_{\mathcal{L}(V_h)} - \smash{\mathtt{J}_{h}^{2,k}})  \delta u_{h}^k-f_h^{2,k}\}\bigr\} &&\quad\text{ in }V_h\,.\label{rem:reduced.2.2}
    \end{alignat}
    \end{subequations}
    Therefore, for every $k\in\{0,\ldots,k_{\mathtt{max}}^h\}$, inserting \eqref{rem:reduced.2.1} in \eqref{rem:reduced.2.2}, we arrive at the discrete linear elliptic problem
    \begin{align}\label{rem:reduced.3}
        \hspace{-2.5mm}\left.\begin{aligned} 
        -\operatorname{div}_h\bigl(\tfrac{1}{\gamma_1}(\smash{\mathtt{J}_{h}^{1,k}})^{-1}(\mathbbone_{\mathcal{L}(Y_h)} - \smash{\mathtt{J}_{h}^{1,k}}) \nabla \delta u_{h}^k\bigr)&+I_{V_h}\bigl\{\tfrac{1}{\gamma_2}(\smash{\mathtt{J}_{h}^{2,k}})^{-1}(\mathbbone_{\mathcal{L}(V_h)} - \smash{\mathtt{J}_{h}^{2,k}})  \delta u_{h}^k\bigr\}\\&=-\operatorname{div}_h\bigl(\tfrac{1}{\gamma_1}(\smash{\mathtt{J}_{h}^{1,k}})^{-1}f_{h}^{1,k}\bigr)\\&\quad-I_{V_h}\bigl\{\tfrac{1}{\gamma_2}(\smash{\mathtt{J}_{h}^{2,k}})^{-1}f_h^{2,k}\bigr\}
        \end{aligned}\;\;\right\}\;\;\text{ in }V_h\,,\hspace{-1mm}
    \end{align}
    or equivalently, by the definition of the discrete divergence operator \eqref{def:discrete_divergence}, for every $v_h\in V_h$, there holds
    \begin{align}\label{rem:reduced.4}
        \begin{aligned} 
        \int_{\Omega}{\tfrac{1}{\gamma_1}(\smash{\mathtt{J}_{h}^{1,k}})^{-1}(\mathbbone_{\mathcal{L}(Y_h)} - \smash{\mathtt{J}_{h}^{1,k}})\nabla \delta u_{h}^{k}\cdot\nabla v_{h}\,\mathrm{d}x}&+\int_{\Omega}{I_{V_h}\{\tfrac{1}{\gamma_2}(\smash{\mathtt{J}_{h}^{2,k}})^{-1}(\mathbbone_{\mathcal{L}(V_h)} - \smash{\mathtt{J}_{h}^{2,k}})\delta u_{h}^{k} v_{h}\}\,\mathrm{d}x}\\&\quad=\int_{\Omega}{\tfrac{1}{\gamma_1}(\smash{\mathtt{J}_{h}^{1,k}})^{-1}f_h^{1,k}\cdot \nabla v_{h}\,\mathrm{d}x}\\&\qquad+\int_{\Omega}{I_{V_h}\{\tfrac{1}{\gamma_2}(\smash{\mathtt{J}_{h}^{2,k}})^{-1}f_h^{2,k}v_{h}\}\,\mathrm{d}x}\,.
        \end{aligned}\hspace{-2.5mm}
    \end{align}
    In view of these observations, for every $k\in\{0,\ldots,k_{\mathtt{max}}^h\}$, instead of solving the perturbed saddle point problem \eqref{alg:semi-smooth_newton.1} directly, we may compute the primal-dual update direction ${(\delta z_{h}^{k},\delta u_{h}^{k})\hspace{-0.1em}\in \hspace{-0.1em}Y_h\hspace{-0.15em}\times\hspace{-0.15em} V_h}$,  solving \eqref{alg:semi-smooth_newton.1}, in two steps:
    \begin{itemize}[noitemsep,topsep=2pt,leftmargin=!,labelwidth=\widthof{(2)}]
        \item[(1)] Compute \hspace{-0.125mm}the \hspace{-0.125mm}\emph{primal \hspace{-0.125mm}update \hspace{-0.125mm}direction} \hspace{-0.125mm}$\delta u_{h}^{k}\hspace{-0.15em}\in\hspace{-0.15em} V_h$, \hspace{-0.125mm}solving \hspace{-0.125mm}the \hspace{-0.125mm}discrete \hspace{-0.125mm}linear~\hspace{-0.125mm}elliptic~\hspace{-0.1mm}\mbox{problem}~\hspace{-0.125mm}\eqref{rem:reduced.3}; 
    \item[(2)] Compute the \emph{dual update direction} $\delta z_{h}^{k}\in  Y_h$, given via the  representation~formula~\eqref{rem:reduced.2.1}.
    \end{itemize}
    
    If the Newton derivative selections
    $\smash{\mathtt{J}_{h}^{1,k}}\in\mathcal{L}(Y_h)$ and
    $\smash{\mathtt{J}_{h}^{2,k}}\in\mathcal{L}(V_h)$ are self-adjoint, the
    discrete  linear elliptic problem \eqref{rem:reduced.3} may be solved efficiently by direct sparse Cholesky
    factorizations for moderate problem sizes or by preconditioned conjugate
    gradient (CG) methods~for~larger~problem sizes. 
    
    Otherwise, it may be solved efficiently by direct sparse LU factorizations
    for moderate problem
    sizes, by general Krylov subspace methods such as GMRES or BiCGStab for larger problem sizes, or one may even solve the original
    saddle point formulation \eqref{alg:semi-smooth_newton.1} directly.
    \end{remark}

    \section{Applications}\vspace{-0.5mm}\label{sec:applications}

    \hspace{5mm}In this section, we apply the %general 
    $\operatorname{prox}$-based semi-smooth Newton method (\textit{cf}.\ Algorithm~\ref{alg:semi-smooth_newton}) derived in Section~\ref{sec:semi-smooth_newton} to three benchmark problems: the TV-minimization problem (linear growth),~the $p$-Dirichlet problem ($p$-growth, $p\in(1,+\infty)$), and the elasto-plastic torsion~\mbox{problem}~(\mbox{$\infty$-growth}).\vspace{-1.5mm}

    \subsection{TV-minimization problem}\vspace{-0.5mm}\label{subsec:tv-min}

    \hspace{5mm}In this subsection, we consider the \emph{Rudin--Osher--Fatemi (ROF) image denoising~model}~(\textit{cf}.~\cite{RudinOsherFatemi1992}) --or \emph{TV-minimization problem}--, which is formulated as the minimization of an energy functional combining total variation (TV) regularization with a lower-order $L^2$-data fidelity term, promoting piece-wise smooth reconstructions while suppressing noise.\vspace{-1.5mm}\enlargethispage{0.5mm}

    \subsubsection{The continuous problem}\vspace{-0.5mm}

    \hspace{5mm}Given \hspace{-0.1mm}a \hspace{-0.1mm}\textit{noisy \hspace{-0.1mm}image} \hspace{-0.1mm}$g\hspace{-0.1em}\in\hspace{-0.1em} L^2(\Omega)$,  \hspace{-0.1mm}\textit{fidelity \hspace{-0.1mm}parameter} \hspace{-0.1mm}$\alpha\hspace{-0.1em}>\hspace{-0.1em}0$, \hspace{-0.1mm}and \hspace{-0.1mm}full \hspace{-0.1mm}Neumann~\hspace{-0.1mm}boundary~\hspace{-0.1mm}\mbox{condition} (\textit{i.e.}, $\Gamma_N=\partial\Omega$), 
    the TV-minimization problem seeks a \emph{denoised image} $u\in BV(\Omega)\cap L^2(\Omega)$ that minimizes the primal energy functional $I\colon BV(\Omega)\cap L^2(\Omega)\to \mathbb{R}$, for every $v\in BV(\Omega)\cap L^2(\Omega)$ defined by
    \begin{align}\label{eq:rof_primal}
        I(v)\coloneqq \vert \mathrm{D}v\vert(\Omega)+\frac{\alpha}{2}\int_{\Omega}{\vert v-g\vert^2\,\mathrm{d}x}\,,
    \end{align}
   where $\vert \mathrm{D}(\cdot)\vert(\Omega)\colon L^1_{\textup{loc}}(\Omega)\to \mathbb{R}\cup\{+\infty\}$, for every $v \in  L^1_{\textup{loc}}(\Omega)$ defined by 
    \begin{align*}
        \vert \mathrm{D}v\vert(\Omega)\coloneqq  \sup_{\substack{y\in (C_{\mathrm{c}}^\infty(\Omega))^d\;\vcentcolon\;\|y\|_{\infty,\Omega}\leq 1}}{\bigg\{\int_{\Omega}{v\operatorname{div}y\,\mathrm{d}x}\bigg\}}\,,
    \end{align*}
    denotes the \textit{total variation (TV) functional} and $BV(\Omega)\coloneqq  \{v \in L^1(\Omega) \mid \vert \mathrm{D}v\vert(\Omega) <+\infty\}$ the~\textit{space of functions with bounded variation}. In \cite[Thms.\ 10.5 \& 10.6]{Bartels2015}, it has been established~that there exists a unique primal solution $u\in BV(\Omega)\cap L^2(\Omega)$, \textit{i.e.}, a minimizer of \eqref{eq:rof_primal}. 

    A (Fenchel) dual problem (in the sense of \cite[Rem.\ 4.2, p.\ 60/61]{EkelandTemam1999}) to the minimization~of~\eqref{eq:rof_primal} was identified in \cite[Sec.\ 2]{HintermuellerKunisch2004} (see also \cite[Sec.\ 10.1.3, p.\ 305]{Bartels2015}) and 
    is given via the maximization of the dual energy functional $D\colon W^2_N(\operatorname{div};\Omega)\to \mathbb{R}\cup\{-\infty\}$, for every $y\in  W^2_N(\operatorname{div};\Omega)$  defined by
    \begin{align}\label{eq:rof_dual}
        D(y)\coloneqq  -I_{K_1^d(0)}^{\Omega}(y)-\frac{1}{2\alpha}\int_{\Omega}{\vert\!\operatorname{div}y+\alpha g\vert^2\,\mathrm{d}x}+\frac{\alpha}{2}\int_{\Omega}{\vert g\vert^2\,\mathrm{d}x}\,.
    \end{align}
    Here, the indicator functional $I_{K_1^d(0)}^{\Omega}\colon (L^2(\Omega))^d\to \mathbb{R}\cup\{+\infty\}$, for every $\widehat{y}\in (L^2(\Omega))^d$,~is~defined~by
    \begin{align*}
        I_{K_1^d(0)}^{\Omega}(\widehat{y})\coloneqq \begin{cases}
            0&\text{ if }\vert \widehat{y}\vert \leq 1\text{ a.e.\ in }\Omega\,,\\
            +\infty&\text{ else}\,.
        \end{cases}
    \end{align*}
    By \cite[Prop.\ 2.1 \& Thm.\ 2.2]{HintermuellerKunisch2004} (see also \cite[Prop.\ 10.3]{Bartels2015}), there exists~a~(not~necessarily~unique) dual solution $z\in W^2_N(\operatorname{div};\Omega)$, \textit{i.e.}, a maximizer of \eqref{eq:rof_dual}, and a strong~\mbox{duality}~\mbox{relation}~\mbox{applies},~\textit{i.e.}, we have that $I(u)=D(z)$, which is equivalent to the primal  optimality relations (\textit{cf}.~\cite[Prop.~10.4]{Bartels2015})%\vspace{-4mm}
\begin{subequations}\label{eq:rof_optimality}
    \begin{alignat}{2}\label{eq:rof_optimality.1}
         \vert \mathrm{D}u\vert(\Omega)&=-\int_{\Omega}{u\operatorname{div}z\,\mathrm{d}x}\,,\\
        \operatorname{div}z&=\alpha (u-g)\quad \text{ a.e.\ in }\Omega\,.\label{eq:rof_optimality.2}
    \end{alignat}
    \end{subequations}
    %Note that the dual solution $z\in W^2_N(\operatorname{div};\Omega)$ is possibly non-unique (\textit{cf}.\ \cite[Expl.\ {BartelsToveyWassmer2022}
    In the  case  $u\in W^{1,1}(\Omega)\cap L^2(\Omega)$, which cannot  be expected in general, by the equality condition in the Fenchel--Young inequality (\textit{cf}.\ \cite[Prop.\ 51.2]{Zeidler1985III}), the primal optimality relation \eqref{eq:rof_optimality.1} is equivalent to
    \begin{align*}
        z\in \left\{\begin{aligned}
            \,\big\{\tfrac{1}{\vert \nabla u\vert}\nabla u\}&\quad\text{ if }\vert \nabla u\vert>0\,,\,\\
            \,K_1^d(0)&\quad\text{ if }\vert \nabla u\vert=0\,
        \end{aligned}\right\}\quad \text{ a.e.\ in }\Omega\,.
    \end{align*} 

    \subsubsection{The discrete problem}\vspace{-0.5mm}

    \hspace{5mm}Given a \emph{discrete noisy image} $g_h\hspace{-0.15em}\in\hspace{-0.15em} V_h$, \textit{i.e.}, an approximation of the actual~noisy~\mbox{image}~${g\hspace{-0.15em}\in\hspace{-0.15em} L^2(\Omega)}$, %and a \textit{fidelity parameter} $\alpha>0$, 
    the discrete TV-minimization problem seeks a \emph{discrete denoised~image} $u_h\in V_h$ that minimizes the discrete primal energy functional $I_h\colon V_h\to \mathbb{R}$, for every $v_h\in V_h$ defined by\vspace{-0.5mm}
    \begin{align}\label{eq:rof_discrete_primal}
        I_h(v_h)\coloneqq \int_{\Omega}{\phi_\varepsilon(\nabla v_h)\,\mathrm{d}x}+\frac{\alpha}{2}\int_{\Omega}{I_{V_h}\{\vert v_h-g_h\vert^2\}\,\mathrm{d}x}\,.\\[-6.25mm]\notag
    \end{align}
    Here, \hspace{-0.15mm}the \hspace{-0.15mm}discrete \hspace{-0.15mm}energy \hspace{-0.15mm}density \hspace{-0.15mm}$\phi_\varepsilon 
    \hspace{-0.15em} \coloneqq\hspace{-0.15em} \min\{\tfrac{1}{2\varepsilon}\vert\cdot\vert^2,\vert \cdot\vert-\tfrac{\varepsilon}{2}\}\colon\hspace{-0.15em} \mathbb{R}^d\hspace{-0.15em}\to\hspace{-0.15em} \mathbb{R}$ \hspace{-0.15mm}denotes \hspace{-0.15mm}the \hspace{-0.15mm}Huber~\hspace{-0.15mm}\mbox{regularization}\linebreak (\textit{cf}.\ \cite{Huber1964}) obtained as Moreau envelope of the energy density~${\phi\coloneqq\vert\cdot\vert\colon \mathbb{R}^d\to \mathbb{R}}$~(\textit{cf}.~\mbox{Example}~\ref{expl:moreau_envelopes}(\hyperlink{expl:moreau_envelopes.i}{i})). The direct method~in~the~\mbox{calculus} of variations yields the existence of a unique discrete primal solution $u_h\in V_h$, \textit{i.e.}, a minimizer~of~\eqref{eq:rof_discrete_primal}.

    According \hspace{-0.15mm}to \hspace{-0.15mm}Proposition \hspace{-0.15mm}\ref{prop:discrete_duality}\hspace{-0.15mm}(\hyperlink{prop:discrete_duality.i}{i}), \hspace{-0.15mm}a \hspace{-0.15mm}(Fenchel) \hspace{-0.15mm}dual \hspace{-0.15mm}problem \hspace{-0.15mm}(in \hspace{-0.15mm}the \hspace{-0.15mm}sense \hspace{-0.15mm}of \hspace{-0.15mm}\cite[Rem.\ \hspace{-0.15mm}4.2,~\hspace{-0.15mm}p.~\hspace{-0.15mm}60/61]{EkelandTemam1999}) to the minimization of \eqref{eq:rof_discrete_primal} is given via the maximization of the discrete dual energy functional $D_h\colon Y_h\to \mathbb{R}\cup\{-\infty\}$, for every $y_h\in Y_h$ defined by\vspace{-0.75mm}
    \begin{align}\label{eq:rof_discrete_dual}
        D_h(y_h)\coloneqq-\int_{\Omega}{\phi_\varepsilon^*(y_h)\,\mathrm{d}x}-\frac{1}{2\alpha}\int_{\Omega}{I_{V_h}\{\vert\!\operatorname{div}_hy_h+\alpha g_h\vert^2\}\,\mathrm{d}x}+\frac{\alpha}{2}\int_{\Omega}{I_{V_h}\{\vert g_h\vert^2\}\,\mathrm{d}x}\,,\\[-6mm]\notag
    \end{align}
    where the Fenchel conjugate $\smash{\phi_\varepsilon^*\colon \mathbb{R}^d\to \mathbb{R}\cup\{+\infty\}}$, %of the discrete energy density $\phi_\varepsilon  \colon \mathbb{R}^d\to \mathbb{R}$, 
    for every $t^*\in \mathbb{R}^d$, is given via\vspace{-0.75mm}
    \begin{align}\label{eq:rof_phih_prime}
        \smash{\phi_\varepsilon^*(t^*)=\tfrac{\varepsilon}{2}\vert t^*\vert^2+I_{K_1^d(0)}(t^*)\,.}\\[-6mm]\notag
    \end{align}
   %In consequence, inserting \eqref{eq:rof_phih_prime} in \eqref{eq:rof_discrete_dual}, the discrete dual energy functional $D_h\colon Y_h\to \mathbb{R}\cup\{-\infty\}$, for every $y_h\in Y_h$, admits the form
   % \begin{align*}
   %     D_h(y_h)=-I_{K_1^d(0)}^{\Omega}(y_h)-\frac{\varepsilon}{2}\int_{\Omega}{\vert y_h\vert^2\,\mathrm{d}x}-\frac{1}{2\alpha}\int_{\Omega}{I_{V_h}\{\vert\!\operatorname{div}_hy_h+\alpha g_h\vert^2\}\,\mathrm{d}x}+\frac{\alpha}{2}\int_{\Omega}{I_{V_h}\{\vert g_h\vert^2\}\,\mathrm{d}x}\,.
   % \end{align*}
    By Proposition \ref{prop:discrete_duality}(\hyperlink{prop:discrete_duality.ii}{ii}), there exists a (due to the strict convexity of $\smash{\phi_\varepsilon^*\colon \mathbb{R}^d\to \mathbb{R}\cup\{+\infty\}}$, unique)
    %(not necessarily 
     %) 
    discrete dual solution $z_h\in Y_h$, \textit{i.e.}, a maximizer~of~\eqref{eq:rof_discrete_dual}, and  discrete strong duality applies, \textit{i.e.}, we have that $I_h(u_h)=D_h(z_h)$, which is equivalent to the discrete primal optimality relations\enlargethispage{3.5mm}\vspace{-0.5mm}
    \begin{subequations}\label{eq:rof_discrete_optimality}
    \begin{alignat}{2}
         z_h&=\mathrm{D}\phi_\varepsilon(\nabla u_h)
         =\smash{\min\{\tfrac{1}{\vert \nabla u_h\vert},\tfrac{1}{\varepsilon}\}\nabla u_h}
         %=\left.\begin{cases}
         %    \tfrac{1}{\varepsilon}\nabla u_h&\text{ if }\vert \nabla u_h\vert \leq \varepsilon\,,\\
         %   \tfrac{1}{\vert \nabla u_h\vert}\nabla u_h&\text{ if }\vert \nabla u_h\vert > \varepsilon\,,
        % \end{cases}\right\}
         &&\quad\text{ a.e.\ in }\Omega\,,\label{eq:rof_discrete_optimality.1}\\
        \operatorname{div}_hz_h&=\alpha (u_h-g_h)&&\quad \text{ in }\Omega\,.\label{eq:rof_discrete_optimality.2}
    \end{alignat}
    \end{subequations}

    \subsubsection{The semi-smooth Newton method}\label{subsec:rof_newton}\vspace{-0.5mm}

    \hspace{5mm}Given proximity parameters $\gamma_1,\gamma_2>0$, following the reasoning in the beginning~of~\mbox{Section}~\ref{sec:semi-smooth_newton}, we observe that the primal optimality relations \eqref{eq:rof_discrete_optimality} are equivalent to the existence of a root of the mapping $\mathtt{F}_h\colon Y_h\times V_h\to Y_h\times V_h$, defined by \eqref{sec:semi-smooth_newton.3}, where, by virtue of Proposition~\ref{prop:proximity_of_integral_functionals}, Lemma~\ref{lem:relations}(\hyperlink{lem:relations.iii}{iii}), and Example \ref{expl:proximity}(\hyperlink{expl:proximity.i}{i}), for every $y_h\in Y_h$ and $v_h\in V_h$, respectively, we have that\vspace{-0.5mm}
    \begin{subequations}\label{eq:rof_proximity}
    \begin{alignat}{2}\label{eq:rof_proximity.1}
        \operatorname{prox}_{\gamma_1 G_h}(y_h)&=\smash{(1-\min\{\tfrac{\gamma_1}{\vert y_h\vert},\tfrac{\gamma_1}{\gamma_1+\varepsilon}\})}y_h
    %  \tfrac{1}{\gamma_1+h}\{h+\gamma_1 \max\{0,1-\tfrac{\gamma_1+h}{\vert y_h\vert}\}\}y_h
    &&\quad\text{ a.e.\ in }\Omega\,,\\ 
        \operatorname{prox}_{\gamma_2 F_h}(v_h)&=\tfrac{1}{1+\smash{\gamma_2}\alpha}(v_h+\smash{\gamma_2}\alpha g_h)&&\quad\text{ in }\Omega\,.\label{eq:rof_proximity.2}
    \end{alignat}\\[-4mm]\notag
    \end{subequations}
    A Newton derivative selection $\mathtt{J}_{\mathtt{F}_h}\colon Y_h\times V_h\to \mathcal{L}(Y_h\times V_h)$ of the mapping $F_h\colon Y_h\times V_h\to Y_h\times V_h$ %for every $(y_h,v_h)\in Y_h\times V_h$, 
    is given via \eqref{sec:semi-smooth_newton.4}, where, for every $y_h\in Y_h$ and $v_h\in V_h$, respectively, we have that\vspace{-0.75mm}
    \begin{subequations}\label{eq:rof_proximity_derivatives}
    \begin{alignat}{2}\label{eq:rof_proximity_derivatives.1}
        \mathtt{J}_{\smash{\operatorname{prox}_{\gamma_1 G_h}}}(y_h)&=\left.\begin{cases}
           \mathbbone_{\mathcal{L}(Y_h)}-\frac{\gamma_1}{\vert y_h\vert}(  \mathbbone_{\mathcal{L}(Y_h)}-\tfrac{y_h\otimes y_h}{\vert y_h\vert^2})&\text{ if }\vert  y_h\vert\ge \varepsilon+\gamma_1\,,\\ 
          \tfrac{\varepsilon}{\gamma_1+\varepsilon}\mathbbone_{\mathcal{L}(Y_h)}&\text{ if } \vert  y_h\vert<\varepsilon+\gamma_1
        \end{cases}\right\}&&\quad\text{ a.e.\ in }\Omega\,,\\ 
        \mathtt{J}_{\smash{\operatorname{prox}_{\gamma_2 F_h}}}(v_h)&=\tfrac{1}{1+\smash{\gamma_2}\alpha}\mathbbone_{\mathcal{L}(V_h)}&&\quad\text{ in }\Omega\,.\label{eq:rof_proximity_derivatives.2}
    \end{alignat}
    \end{subequations}

    Note that the upper bound condition 
$\mathtt{J}_{\smash{\operatorname{prox}_{\gamma_1G_h}}}(y_h)\hspace{-0.15em}\prec\hspace{-0.15em} \mathbbone_{\mathcal{L}(Y_h)}$ for all $y_h\hspace{-0.15em}\in\hspace{-0.15em} Y_h$ in Assumption~\ref{ass:wellposedness}(\hyperlink{ass:wellposedness.i}{i})\linebreak is, in general, not satisfied for the Newton derivative selection given via \eqref{eq:rof_proximity_derivatives.1}. In~fact,~there~holds
$\mathtt{J}_{\smash{\operatorname{prox}_{\smash{\gamma_1G_h}}}}(y_h)=\mathbbone_{\mathcal{L}(Y_h)}-\smash{\frac{\gamma_1}{|y_h|}}\mathtt{P}_{\smash{y_h}}$ a.e.\ in  $\{|y_h|>\gamma_1+\varepsilon\}$, where
$\mathtt{P}_{\smash{y_h}}\coloneqq\mathbbone_{\mathcal{L}(Y_h)}-\smash{\frac{y_h\otimes y_h}{|y_h|^2}}\in \mathcal{L}(Y_h)$
denotes the orthogonal projection onto $\smash{(\mathbb{R}y_h)^\perp}$. In particular, since
$\mathtt{P}_{\smash{y_h}} y_h =\mathtt{0}_d$ a.e.\ in $\Omega$, we obtain
$\mathtt{J}_{\smash{\operatorname{prox}_{\smash{\gamma_1G_h}}}}(y_h)y_h=y_h$  a.e.\ in $\Omega$,
\textit{i.e.}, the eigenvalue $1$ is attained in the radial direction $\mathbb{R}y_h$.  
However, thanks to the lower-order quadratic fidelity term, resorting to a~discrete~inverse~estimate, one can still prove global well-posedness and  local super-linear convergence of Algorithm~\ref{alg:semi-smooth_newton}.

    \begin{theorem}%[Global well-posedness and local super-linear convergence of Algorithm \ref{alg:semi-smooth_newton}]
    \label{thm:wellposedness_tv}
    Algorithm~\ref{alg:semi-smooth_newton} employing the representations \eqref{eq:rof_proximity} and  \eqref{eq:rof_proximity_derivatives} is globally well-posed and locally super-linearly convergent (with $\alpha_k=1$ for all $k\in \mathbb{N}$) to a root~${(z_h,u_h)\in Y_h\times V_h}$.\vspace{-0.5mm}
    \end{theorem}
    
    \begin{proof}
        See \cite[Thm.\ 4.2]{BartelsKaltenbach2026}.  
    \end{proof}\newpage
    
    \subsection{$p$-Dirichlet problem}\label{subsec:p-laplace} 

    \hspace{5mm}In this subsection, we consider the \emph{$p$-Dirichlet problem} (\textit{cf}.\ \cite{Ladyzhenskaya1968,Lindqvist2017}), a nonlinear generalization of the classical Dirichlet problem arising in the modelling of non-Newtonian fluid flow and nonlinear elasticity. Originating from the early studies of power-law rheology (\textit{cf}.\ \cite{Ostwald1925I,deWaele1923}), the model is formulated as the minimization of the \emph{$p$-Dirichlet energy functional}, which combines nonlinear diffusion with prescribed mixed (homogeneous) Dirichlet and Neumann~boundary~conditions, and leads to the \emph{$p$-Laplace equation}, describing stationary states in media whose effective~conductivity or viscosity depends nonlinearly on the gradient magnitude. In contrast to the~linear~case~$p=2$, however, the design of robust nonlinear solvers for the $p$-Dirichlet problem is delicate, and~there appears to be no universally preferred solution strategy across the full range $p\in(1,+ \infty)$.\vspace{-1mm}\enlargethispage{7.5mm}

    \subsubsection{The continuous problem} 

    \hspace{5mm}Given an \emph{exponent} $p\in (1,+\infty)$, a \emph{right-hand side} $f\in L^{p'}(\Omega)$, and a homogeneous Dirichlet boundary \hspace{-0.15mm}condition \hspace{-0.15mm}on \hspace{-0.15mm}a \hspace{-0.15mm}non-empty \hspace{-0.15mm}Dirichlet \hspace{-0.15mm}boundary \hspace{-0.15mm}part \hspace{-0.15mm}$\Gamma_D\hspace{-0.175em}\subseteq\hspace{-0.175em} \partial\Omega$, \hspace{-0.15mm}the~\hspace{-0.15mm}\mbox{$p$-Dirichlet}~\hspace{-0.15mm}\mbox{problem}~\hspace{-0.15mm}seeks a function $u\in W^{1,p}_D(\Omega)$ that minimizes the primal energy functional $I\colon W^{1,p}_D(\Omega)\to \mathbb{R}$, for every $v\in W^{1,p}_D(\Omega)$ defined by\vspace{-0.5mm}
    \begin{align}\label{eq:pdirichlet_primal}
        I(v)\coloneqq \frac{1}{p}\int_{\Omega}{\vert\nabla v\vert^p\,\mathrm{d}x}-\int_{\Omega}{fv\,\mathrm{d}x}\,.\\[-6mm]\notag
    \end{align}
    The direct method in the calculus of variations yields the existence of a unique primal solution $u\in W^{1,p}_D(\Omega)$, \textit{i.e.}, a minimizer of \eqref{eq:pdirichlet_primal}.

    A (Fenchel) dual problem (in the sense of \cite[Rem.\ 4.2,  p.\ 60/61]{EkelandTemam1999}) to the  minimization~of~\eqref{eq:pdirichlet_primal} (\textit{e.g.}, derived in \cite[Subsec.\ 2.2, p.\ 81/82]{EkelandTemam1999}) is given via  the maximization of the dual~energy~functional $D\colon \smash{W^{p'}_N(\operatorname{div};\Omega)}\to \mathbb{R}\cup\{-\infty\}$, for every $y\in  \smash{W^{p'}_N(\operatorname{div};\Omega)}$ defined by\vspace{-0.5mm}
    \begin{align}\label{eq:pdirichlet_dual}
        D(y)\coloneqq -\frac{1}{p'}\int_{\Omega}{\vert y\vert^{p'}\,\mathrm{d}x}-I_{\{-f\}}^{\Omega}(\operatorname{div}y)\,,\\[-6mm]\notag
    \end{align}
    where the indicator functional $\smash{I_{\{-f\}}^{\Omega}}\colon \smash{L^{p'}(\Omega)}\to \mathbb{R}\cup\{+\infty\}$, for every $\widehat{v}\in \smash{L^{p'}(\Omega)}$, is defined by\vspace{-0.5mm}
    \begin{align}
        I_{\{-f\}}^{\Omega}(\widehat{v})\coloneqq\begin{cases}
            0&\text{ if }\widehat{v}=-f\text{ a.e.\ in }\Omega\,,\\
            +\infty&\text{ else}\,.
        \end{cases}\\[-6.5mm]\notag
    \end{align}
    By \cite[Prop.\ 2.2, p.\ 82]{EkelandTemam1999}, there exists a (due to the strict convexity of $\phi^*=\smash{\frac{1}{p'}}\vert \cdot\vert^{p'}\colon \mathbb{R}^d\to \mathbb{R}$,~unique) dual solution $z\in \smash{W^{p'}_N(\operatorname{div};\Omega)}$, \textit{i.e.}, a maximizer of \eqref{eq:pdirichlet_dual}, and a strong duality relation applies, \textit{i.e.}, we have that $I(u)=D(z)$, which is equivalent to the primal optimality relations\vspace{-0.5mm}
    \begin{subequations} \label{eq:pdirichlet_optimality}
    \begin{alignat}{2}\label{eq:pdirichlet_optimality.1}
        z&=\vert \nabla u\vert^{p-2}\nabla u&&\quad \text{ a.e.\ in }\Omega\,,\\
        \operatorname{div}z&=-f&&\quad \text{ a.e.\ in }\Omega\,.\label{eq:pdirichlet_optimality.2}\\[-6.5mm]\notag
    \end{alignat}
    \end{subequations}

    \subsubsection{The discrete problem} 

    \hspace{5mm}Given \hspace{-0.15mm}a \hspace{-0.15mm}\emph{discrete \hspace{-0.15mm}right-hand \hspace{-0.15mm}side} \hspace{-0.15mm}$f_h\hspace{-0.15em}\in\hspace{-0.15em} V_h$, \hspace{-0.15mm}the \hspace{-0.15mm}\emph{discrete \hspace{-0.15mm}$p$-Dirichlet \hspace{-0.15mm}problem} \hspace{-0.15mm}seeks~\hspace{-0.15mm}a~\hspace{-0.15mm}\mbox{function}~\hspace{-0.15mm}${u_h\hspace{-0.15em}\in\hspace{-0.15em} V_h}$ that minimizes the discrete primal energy functional   $I_h\colon V_h\to \mathbb{R}$, for every $v_h\in V_h$ defined by\vspace{-0.5mm}
    \begin{align}\label{eq:pdirichlet_primal_discrete}
        I_h(v_h)\coloneqq \int_{\Omega}{\phi_\varepsilon(\nabla v_h)\,\mathrm{d}x}-\int_{\Omega}{I_{V_h}\{f_hv_h\}\,\mathrm{d}x}\,.\\[-6mm]\notag
    \end{align}
    Here, $\phi_\varepsilon\coloneqq \smash{\frac{1}{p}(\varepsilon^2+\vert \cdot\vert^2)^{\frac{p}{2}}}\colon \mathbb{R}^d\to \mathbb{R}$ denotes the standard regularization of the $p$-Dirichlet density $\phi\coloneqq \smash{\tfrac{1}{p}\vert\cdot\vert^p}\colon \mathbb{R}^d\to \mathbb{R}$ (\textit{cf}.\ \cite[Expl.\ 2.1(i)]{BartelsDieningNochetto2018}). The direct method in the calculus of variations yields the existence of a unique discrete primal solution $u_h\in V_h$, \textit{i.e.}, a minimizer of \eqref{eq:pdirichlet_primal_discrete}.

      According \hspace{-0.15mm}to \hspace{-0.15mm}Proposition \hspace{-0.15mm}\ref{prop:discrete_duality}\hspace{-0.15mm}(\hyperlink{prop:discrete_duality.i}{i}), \hspace{-0.15mm}a \hspace{-0.15mm}(Fenchel) \hspace{-0.15mm}dual \hspace{-0.1mm}problem \hspace{-0.15mm}(in \hspace{-0.15mm}the \hspace{-0.1mm}sense \hspace{-0.15mm}of \hspace{-0.15mm}\cite[Rem.\ \hspace{-0.15mm}4.2,~\hspace{-0.15mm}p.~\hspace{-0.15mm}60/61]{EkelandTemam1999}) to the minimization~of~\eqref{eq:pdirichlet_primal_discrete} is given via the maximization of the discrete dual energy functional $D_h\colon Y_h\to \mathbb{R}\cup\{-\infty\}$, for every $y_h\in  Y_h$ defined by\vspace{-0.5mm}
    \begin{align}\label{eq:pdirichlet_dual_discrete}
        D_h(y_h)\coloneqq -\int_{\Omega}{\phi_\varepsilon^*(y_h)\,\mathrm{d}x}-I_{\{-f_h\}}^{\Omega}(\operatorname{div}_hy_h)\,.
    \end{align}
    Note that, in general, there is no explicit representation for the Fenchel conjugate $\phi_\varepsilon^*\colon \mathbb{R}^d\to \mathbb{R}$.
     By Proposition \ref{prop:discrete_duality}(\hyperlink{prop:discrete_duality.ii}{ii}), there exists a (due to the strict convexity of $\phi_\varepsilon^*\colon \mathbb{R}^d\to \mathbb{R}$, unique) discrete dual solution $z_h\in Y_h$, \textit{i.e.}, a maximizer of \eqref{eq:pdirichlet_dual_discrete}, discrete strong duality applies, \textit{i.e.},~we~have~that $I_h(u_h)=D_h(z_h)$, which is equivalent to the discrete primal optimality relations\vspace{-1mm}
    \begin{subequations} \label{eq:pdirichlet_discrete_optimality}
    \begin{alignat}{2}\label{eq:pdirichlet_discrete_optimality.1}
         z_h&=\mathrm{D}\phi_\varepsilon(\nabla u_h)=(\varepsilon^2+\vert \nabla u_h\vert^2)^{\frac{p-2}{2}}\nabla u_h&&\quad\text{ a.e.\ in }\Omega\,,\\
         \operatorname{div}_hz_h&=-f_h&&\quad\text{ in }\Omega\,.
         \label{eq:pdirichlet_discrete_optimality.2}
    \end{alignat}
    \end{subequations}

    \subsubsection{The semi-smooth Newton method}\vspace{-0.5mm}

    \hspace*{5mm}Given proximity parameters $\gamma_1,\gamma_2>0$, following the reasoning in the beginning~of~\mbox{Section}~\ref{sec:semi-smooth_newton}, we observe that the primal optimality relations \eqref{eq:pdirichlet_discrete_optimality} are equivalent to the existence of a root of the mapping $\mathtt{F}_h\colon Y_h\times V_h\to Y_h\times V_h$, defined by \eqref{sec:semi-smooth_newton.3}, where, by virtue of Proposition \ref{prop:proximity_of_integral_functionals}, for every $y_h\in Y_h$ and $v_h\in V_h$, respectively, we have that\vspace{-0.5mm}
    \begin{subequations}\label{eq:pdirichlet_proximity}
\begin{alignat}{2}
\operatorname{prox}_{\gamma_1 G_h}(y_h)(x)
    &=\operatorname{prox}_{\gamma_1 \phi_h(x,\cdot)}(y_h(x))
    &&\quad\text{ for a.e.\ }x\in \Omega\,,
    \label{eq:pdirichlet_proximity.1}
    \\
\operatorname{prox}_{\gamma_2 F_h}(v_h)
    &=v_h+\smash{\gamma_2}f_h
    &&\quad\text{ in }\Omega\,.
    \label{eq:pdirichlet_proximity.2}
\end{alignat}
\end{subequations} 

For the proximity operator 
$\operatorname{prox}_{\gamma_1 G_h}\colon Y_h\to Y_h$,
in general, no explicit representation~formula~is available.
In fact, due to the non-quadratic $p$-growth of the regularized~energy~density~${\phi_\varepsilon\hspace{-0.1em}\in\hspace{-0.1em} C^1(\mathbb{R}^d)}$, closed formulas can merely be derived in particular cases,
\textit{e.g.}, for $p\in\{1,2,3,4\}$.
For this reason, for a given $y_h\in Y_h$, the value
$\operatorname{prox}_{\gamma_1 G_h}(y_h)\in Y_h$
has to be computed numerically.

\begin{remark}[on computation of $\operatorname{prox}_{\gamma_1 G_h}(y_h)\in Y_h$]\label{rem:comput_prox} For a given $y_h\in Y_h$, due to the definition of the proximity operator (\textit{cf}.\ Definition \ref{def:proximity}) and $\phi_\varepsilon\in C^1(\mathbb{R}^d)$, the value $\widehat y_h\coloneqq\operatorname{prox}_{\gamma_1 G_h}(y_h)\in Y_h$
is characterized as the unique solution of
\begin{align}\label{eq:pdirichlet_prox_equation}
    \smash{\mathrm{D}\phi_\varepsilon(\widehat y_h)
    +\tfrac{1}{\gamma_1}(\widehat y_h-y_h)=0
    \quad\text{in }Y_h\,.}
\end{align}
Since 
$Y_h$
consists of element-wise constant vector fields,
the system~\eqref{eq:pdirichlet_prox_equation}
decouples~\mbox{element-wise}, \textit{i.e.}, 
for every $T\in\mathcal T_h$,
 $\widehat y_h|_T\in\mathbb R^d$
depends only on $y_h|_T\in\mathbb R^d$
and is determined by a~system~in~$\mathbb R^d$.\linebreak
Moreover, exploiting the radial symmetry of $\phi_\varepsilon\colon \mathbb{R}^d\to \mathbb{R}$, for every $T\in\mathcal T_h$, the respective system reduces to a scalar equation for the modulus
$|\widehat y_h|_T|\in \mathbb{R}$.
As a consequence,~the~\mbox{computation} of $\widehat y_h=\operatorname{prox}_{\gamma_1 G_h}(y_h)\in Y_h$
reduces to the parallel solution of scalar equations on each element $T\in \mathcal{T}_h$. 
Due to the additional quadratic term in the definition of the proximity
operator (\textit{cf}.\ Definition~\ref{def:proximity}),\linebreak the  scalar equations are $\frac{1}{\gamma_1}$-strongly monotone.
Therefore,~they~admit~unique~\mbox{solutions}, their derivatives do not vanish at
the root, and the scalar Newton iterations~converge~locally~\mbox{quadratically}.
\end{remark}  

Since $\phi_\varepsilon\in C^2(\mathbb R^d)$, due to Lemma \ref{lem:prop_proximity}(\hyperlink{lem:prop_proximity.iv}{iv}), 
the proximity operator
$\operatorname{prox}_{\gamma_1 G_h}\colon Y_h\to Y_h$~is~even continuously
Fr\'echet differentiable. In particular, for every $y_h\in Y_h$ and $v_h\in V_h$, respectively, we have that\vspace{-1mm}
\begin{subequations}\label{eq:pdirichlet_proximity_derivatives}
\begin{alignat}{2}
\mathtt{J}_{\smash{\operatorname{prox}_{\gamma_1 G_h}}}(y_h)
    &=
    \bigl(
        \mathbbone_{\mathcal L(Y_h)}
        +
        \gamma_1
        \mathrm{D}^2G_h(
            \operatorname{prox}_{\gamma_1 G_h}(y_h)
        )
    \bigr)^{-1}
    &&\quad\text{a.e.\ in }\Omega\,,
    \label{eq:pdirichlet_proximity_derivatives.1}
    \\
\mathtt{J}_{\smash{\operatorname{prox}_{\gamma_2 F_h}}}(v_h)
    &=\mathbbone_{\mathcal L(V_h)}
    &&\quad\text{in }\Omega\,.
    \label{eq:pdirichlet_proximity_derivatives.2}
\end{alignat}
\end{subequations}
In other words,  the evaluation of the Fr\'echet derivative
$\mathtt{J}_{\smash{\operatorname{prox}_{\gamma_1 G_h}}}(y_h)\in \mathcal{L}(Y_h)$
requires~only~the computation of the value
$\operatorname{prox}_{\gamma_1 G_h}(y_h)\in Y_h$ itself.

The explicit representations in \eqref{eq:pdirichlet_proximity_derivatives}
show that the assumptions of Theorems~\ref{prop:well-posedness}
and \ref{prop:convergence} are satisfied in the present setting.\enlargethispage{5mm}

\begin{theorem}%[Global well-posedness and local super-linear convergence of Algorithm~\ref{alg:semi-smooth_newton}]
\label{prop:wellposedness_pdirichlet}
Algorithm~\ref{alg:semi-smooth_newton} employing the representations \eqref{eq:pdirichlet_proximity} and  \eqref{eq:pdirichlet_proximity_derivatives} is globally well-posed and locally super-linearly convergent (with $\alpha_k=1$ for all $k\in \mathbb{N}$) to a root~${(z_h,u_h)\in Y_h\times V_h}$. 
\end{theorem}

    \begin{proof}
        From representation  \eqref{eq:pdirichlet_proximity_derivatives.1}, it follows that $\mathbbzero_{\mathcal{L}(Y_h)} \prec \mathtt{J}_{\smash{\operatorname{prox}_{\gamma_1 G_h}}}(y_h)\prec \mathbbone_{\mathcal{L}(Y_h)}$~for~all~$y_h\in  Y_h$\linebreak  and 
        $\underline{\smash{\mu}}_h^1\mathbbone_{\mathcal{L}(Y_h)}\prec \mathtt{J}_{\smash{\operatorname{prox}_{\gamma_1 G_h}}}(y_h)\prec\overline{\mu}_h^1\mathbbone_{\mathcal{L}(Y_h)}$ for all $y_h\in  B_{\varepsilon_h}^{Y_h}(0)$, where $0<\underline{\mu}_h^1\leq \overline{\mu}_h^1<1$~\mbox{depend}~on $B_{\varepsilon_h}^{Y_h}(0)$, so that, also using representation  \eqref{eq:pdirichlet_proximity_derivatives.2}, we conclude that
        Assumptions~\ref{ass:wellposedness}~and~\ref{ass:convergence} are satisfied.
    \end{proof}
    
    \newpage

    \subsection{Elasto-plastic torsion problem}\label{subsec:elasto-plastic}

    \hspace{5mm}In this subsection, we consider the \emph{elasto-plastic torsion problem} (\textit{cf}.\  \cite{DuvautLions1972}), a classical model problem in structural mechanics --originating in Prandtl’s torsion theory (\textit{cf}.\ \cite{Prandtl1903})--, describing the torsional deformation of a prismatic bar made of an elasto-plastic material, formulated as a variational problem with a point-wise gradient constraint reflecting the yield condition.

    \subsubsection{The continuous problem}

    \hspace{5mm}Given \hspace{-0.15mm}an \hspace{-0.15mm}\emph{external \hspace{-0.15mm}load} \hspace{-0.15mm}(or \hspace{-0.15mm}\emph{torque}) \hspace{-0.15mm}$f\hspace{-0.175em}\in \hspace{-0.175em}L^1(\Omega)$
    \hspace{-0.15mm}and \hspace{-0.15mm}a \hspace{-0.15mm}homogeneous \hspace{-0.15mm}Dirichlet~\hspace{-0.15mm}boundary~\hspace{-0.15mm}\mbox{condition}~\hspace{-0.15mm}on\linebreak a non-empty Dirichlet boundary part $\Gamma_D\subseteq \partial\Omega$,
 the elasto-plastic torsion problem seeks~a~\emph{Prandtl stress potential} $u\hspace{-0.15em}\in\hspace{-0.15em} W^{1,\infty}_D(\Omega)$ that minimizes the primal energy functional ${I\colon \hspace{-0.15em}W^{1,\infty}_D(\Omega)\hspace{-0.15em}\to \hspace{-0.15em}\mathbb{R}\hspace{-0.1em}\cup\hspace{-0.1em}\{+\infty\}}$, for every $v\in W^{1,\infty}_D(\Omega)$ defined by
    \begin{align}\label{eq:elasto_primal}
        I(v)\coloneqq \frac{1}{2}\int_{\Omega}{\vert \nabla v\vert^2\,\mathrm{d}x}+I_{K_1^d(0)}^{\Omega}(\nabla v)-\int_{\Omega}{fv\,\mathrm{d}x}\,.
    \end{align}
    The direct method in the calculus of variations yields the existence of a unique  primal solution $u\in W^{1,\infty}_D(\Omega)$ (\textit{cf}.\ \cite[Sec.\ 3]{AntilBartelsKaltenbachKhandelwal2025}), \textit{i.e.}, a minimizer~of~\eqref{eq:elasto_primal}.

    A (Fenchel) dual problem (in the sense of \cite[Rem.\ 4.2, p.\ 60/61]{EkelandTemam1999}) to the minimization~of~\eqref{eq:elasto_primal} was identified in \cite[Thm.\ 3.1(i)]{AntilBartelsKaltenbachKhandelwal2025} and is given via the maximization of the dual energy functional\linebreak $D\colon \hspace{-0.15em}(\mathrm{ba}(\Omega))^d\hspace{-0.15em}\to\hspace{-0.15em} \mathbb{R}\cup\{-\infty\}$, for every $\nu \hspace{-0.15em}=\hspace{-0.15em}y\otimes \mathrm{d}x+\nu^s\hspace{-0.15em}\in\hspace{-0.15em} (\mathrm{ba}(\Omega))^d$, where $y\hspace{-0.15em}\in\hspace{-0.15em} (L^1(\Omega))^d$~and~${\nu^s\hspace{-0.15em}\in\hspace{-0.15em} (\mathrm{ba}(\Omega))^d}$ with $\nu^s\perp y\otimes \mathrm{d}x$ form the unique Lebesgue decomposition of $\nu \in (\mathrm{ba}(\Omega))^d$ (\textit{cf}.\ \cite[Thm.~3.1]{Toland20}), defined by
    \begin{align}\label{eq:elasto_dual}
        D(\nu)\coloneqq -\frac{1}{2}\int_{\Omega}{\vert y\vert^2\,\mathrm{d}x}-\frac{1}{2}\int_{\Omega}{(\vert y\vert-1)_+^2\,\mathrm{d}x}-\vert\nu^s\vert(\Omega)-I_{K^*}(\nu)\,,
    \end{align}
    Here, $(\mathrm{ba}(\Omega))^d$ denotes the \emph{space of bounded
and finitely additive vector measures} on $\mathcal{M}(\mathrm{d}x;\Omega)$ (\textit{cf}.\ \cite[Def.\ 4.1]{Toland20}) and 
    the indicator functional $I_{K^*}\colon (\mathrm{ba}(\Omega))^d\to \mathbb{R}\cup\{+\infty\}$, for every $\widehat{\nu} \in (\mathrm{ba}(\Omega))^d$, is defined by 
    \begin{align*}
        I_{K^*}(\widehat{\nu})\coloneqq \begin{cases}
            0&\text{ if }\displaystyle\langle \widehat{\nu},\nabla v\rangle_{(L^\infty(\Omega))^d}=\int_{\Omega}{fv\,\mathrm{d}x}\text{ for all }v\in W^{1,\infty}_D(\Omega)\,,\\
            +\infty&\text{ else}\,,
        \end{cases}
    \end{align*}
    \textit{i.e.}, for every $\widehat{\nu}=\widehat{y}\otimes \mathrm{d}x \in (\mathrm{ba}(\Omega))^d$, where $\widehat{y}\in (L^1(\Omega))^d$, there holds
    \begin{align*}
        I_{K^*}(\widehat{\nu})=\begin{cases}
            0&\text{ if }\widehat{y}\in W^1_N(\operatorname{div};\Omega)\text{ with}\operatorname{div}\widehat{y}=-f\text{ a.e.\ in }\Omega \,,\\
            +\infty&\text{ else}\,.
        \end{cases}
    \end{align*}
    In particular, for every $\nu=y\otimes \mathrm{d}x \in (\mathrm{ba}(\Omega))^d$, where $y\in W_N^1(\operatorname{div};\Omega)$, we have that
    \begin{align}\label{eq:elasto_dual2}
        D(\nu)=-\frac{1}{2}\int_{\Omega}{\vert y\vert^2\,\mathrm{d}x}-\frac{1}{2}\int_{\Omega}{(\vert y\vert-1)_+^2\,\mathrm{d}x}-I_{\{-f\}}^{\Omega}(\operatorname{div}y)\,.
    \end{align}
    By \cite[Thm.\ 3.1(ii)]{AntilBartelsKaltenbachKhandelwal2025}, 
    there exists a (not necessarily unique) dual solution ${\mu\hspace{-0.1em}=\hspace{-0.1em}z\hspace{-0.1em}\otimes\hspace{-0.1em}\mathrm{d}x\hspace{-0.15em}+\hspace{-0.15em}\mu^s\hspace{-0.1em}\in\hspace{-0.1em} (\mathrm{ba}(\Omega))^d}$,  where $z\in (L^1(\Omega))^d$ and $\mu^s\in (\mathrm{ba}(\Omega))^d$ with $\mu^s\perp z\otimes \mathrm{d}x$ form the unique Lebesgue decomposition of $\mu \in (\mathrm{ba}(\Omega))^d$ (\textit{cf}.\ \cite[Thm.~3.1]{Toland20}),  \textit{i.e.}, a maximizer of \eqref{eq:elasto_dual}, and strong duality applies, \textit{i.e.}, we have that $I(u)=D(\mu)$, which is equivalent to the primal optimality relations~(\textit{cf}.~\mbox{\cite[Thm.~3.1(ii)]{AntilBartelsKaltenbachKhandelwal2025}})\vspace{-2.5mm}
    \begin{subequations}\label{eq:elasto_optimality}
    \begin{alignat}{2}\label{eq:elasto_optimality.1}
         \nabla u&=\smash{\min\{1,\tfrac{1}{\vert z\vert}\}}z&&\quad \text{ a.e.\ in }\Omega\,,\\[2.5mm]
        \vert \mu^s\vert(\Omega)&=\langle \mu^s,\nabla u\rangle_{(L^\infty(\Omega))^d}\,,\label{eq:elasto_optimality.2}\\
        \langle \mu,\nabla v\rangle_{(L^\infty(\Omega))^d}&=\int_{\Omega}{fv\,\mathrm{d}x}&&\quad\text{ for all }v\in W^{1,\infty}_D(\Omega)\,.\label{eq:elasto_optimality.3}
    \end{alignat}
    \end{subequations}

    \subsubsection{The discrete problem}\vspace{-0.5mm}

    \hspace{5mm}Given a \emph{discrete external load} (or \emph{torque}) $f_h\in V_h$, the  discrete elasto-plastic torsion~problem seeks a \emph{discrete Prandtl stress potential} $u_h\in V_h$ that minimizes the discrete primal~energy~functional $I_h\colon V_h\to \mathbb{R}$, for every $v_h\in V_h$ defined by\vspace{-0.5mm}
    \begin{align}\label{eq:elasto_discrete_primal}
        I_h(v_h)\coloneqq \int_{\Omega}{\phi_\varepsilon(\nabla v_h)\,\mathrm{d}x}-\int_{\Omega}{I_{V_h}\{f_hv_h\}\,\mathrm{d}x}\,.\\[-6.5mm]\notag
    \end{align}
    Here, the discrete energy density  $\phi_\varepsilon\coloneqq\tfrac{1}{2(1+\varepsilon)}\vert \cdot\vert^2+\tfrac{1}{2\varepsilon(1+\varepsilon)}(\vert \cdot\vert-(1+\varepsilon))_+^2\colon \hspace{-0.15em}\mathbb{R}^d\hspace{-0.15em}\to\hspace{-0.15em} \mathbb{R}$~\mbox{denotes}\linebreak the Moreau envelope of the energy density $\phi\coloneqq\tfrac{1}{2}\vert \cdot\vert^2+I_{K_1^d(0)}\colon \hspace{-0.15em}\hspace{-0.15em}\mathbb{R}^d\hspace{-0.15em}\to\hspace{-0.15em} \mathbb{R}\cup\{+\infty\}$~(\textit{cf}.~Example~\ref{expl:moreau_envelopes}(\hyperlink{expl:moreau_envelopes.ii}{ii})). The direct method~in~the~\mbox{calculus} of variations yields the existence of a unique discrete primal solution $u_h\in V_h$, \textit{i.e.}, a minimizer~of~\eqref{eq:elasto_discrete_primal}. 

      According \hspace{-0.15mm}to \hspace{-0.15mm}Proposition \hspace{-0.15mm}\ref{prop:discrete_duality}\hspace{-0.15mm}(\hyperlink{prop:discrete_duality.i}{i}), \hspace{-0.15mm}a \hspace{-0.15mm}(Fenchel) \hspace{-0.15mm}dual \hspace{-0.1mm}problem \hspace{-0.15mm}(in \hspace{-0.1mm}the \hspace{-0.1mm}sense \hspace{-0.15mm}of \hspace{-0.1mm}\cite[Rem.\ \hspace{-0.15mm}4.2,~\hspace{-0.15mm}p.~\hspace{-0.15mm}60/61]{EkelandTemam1999}) to the minimization of \eqref{eq:elasto_discrete_primal} is given via the maximization of the discrete dual energy functional $D_h\colon Y_h\to \mathbb{R}\cup\{-\infty\}$, for every $y_h\in Y_h$ defined by\vspace{-0.5mm}
    \begin{align}\label{eq:elasto_discrete_dual}
        D_h(y_h)\coloneqq -\int_{\Omega}{\phi_\varepsilon^*(y_h)\,\mathrm{d}x}-I_{\{-f_h\}}^{\Omega}(\operatorname{div}_hy_h)\,,\\[-6.5mm]\notag
    \end{align}
    where the Fenchel conjugate $\phi_\varepsilon^*\colon \hspace{-0.15em}\mathbb{R}^d\hspace{-0.15em}\to \hspace{-0.15em}\mathbb{R}$ of the discrete energy density $\phi_\varepsilon \colon\hspace{-0.15em}  \mathbb{R}^d\hspace{-0.15em}\to\hspace{-0.15em} \mathbb{R}$,~for~\mbox{every}~${t^*\hspace{-0.15em}\in\hspace{-0.15em} \mathbb{R}^d}$, is given via\vspace{-1mm}
    \begin{align}\label{eq:elasto_phih_prime}
        \phi_\varepsilon^*(t^*)=\tfrac{\varepsilon}{2}\vert t^*\vert^2+\vert t^*\vert-\tfrac{1}{2}+\tfrac{1}{2}(1-\vert t^*\vert)_+^2
        =\begin{cases}
            \frac{1+\varepsilon}{2}\vert t^*\vert^2&\text{ if }\vert t^*\vert \leq 1\,,\\
            \frac{\varepsilon}{2}\vert t^*\vert^2+\vert t^*\vert-\frac{1}{2}&\text{ if }\vert t^*\vert > 1\,.
        \end{cases}\\[-6.5mm]\notag
    \end{align}
   %In consequence, inserting \eqref{eq:elasto_phih_prime} in \eqref{eq:elasto_discrete_dual}, the discrete dual energy functional $D_h\colon Y_h\to \mathbb{R}\cup\{-\infty\}$, for every $y_h\in Y_h$, admits the form
   % \begin{align*}
   %     D_h(y_h)\coloneqq-\frac{\varepsilon}{2}\int_{\Omega}{\vert y_h\vert^2\,\mathrm{d}x}-\int_{\Omega}{\{\vert y_h\vert-\tfrac{1}{2}\}\,\mathrm{d}x}-\frac{1}{2}\int_{\Omega}{(\max\{0,1-\vert y_h\vert\})^2\,\mathrm{d}x}-I_{\{f_h\}}(\operatorname{div}_hy_h)\,.
   % \end{align*}
    By Proposition \ref{prop:discrete_duality}(\hyperlink{prop:discrete_duality.ii}{ii}), there exists a (not necessarily unique) discrete dual solution $z_h\in Y_h$, \textit{i.e.}, a maximizer~of~\eqref{eq:elasto_discrete_dual}, and discrete strong duality applies, \textit{i.e.}, we have that $I_h(u_h)=D_h(z_h)$, which is equivalent to the discrete primal optimality relations\vspace{-0.5mm}
    \begin{subequations} \label{eq:elasto_discrete_optimality}
    \begin{alignat}{2}\label{eq:elasto_discrete_optimality.1}
         z_h&=\mathrm{D}\phi_\varepsilon(\nabla u_h)=\smash{\tfrac{1}{\varepsilon}(1-\min\{\tfrac{1}{1+\varepsilon},\tfrac{1}{\vert \nabla u_h\vert}\})\nabla u_h}
         %=\left.\begin{cases}
         %    \tfrac{1}{1+\varepsilon}\nabla u_h&\text{ if }\vert \nabla u_h\vert \leq 1+\varepsilon\,,\\
        %     \tfrac{1}{\varepsilon}\{1-\tfrac{1}{\vert \nabla u_h\vert}\}\nabla u_h&\text{ if }\vert \nabla u_h\vert > 1+\varepsilon\,,
        % \end{cases}\right\}
        &&\quad \text{ a.e.\ in }\Omega\,,\\
         \operatorname{div}_hz_h&=-f_h&&\quad\text{ in }\Omega\,.\label{eq:elasto_discrete_optimality.2}
    \end{alignat}
    \end{subequations}\vspace{-6.5mm}
    
    \subsubsection{The semi-smooth Newton method}\vspace{-0.5mm}

    \hspace{5mm}Given proximity parameters $\gamma_1,\gamma_2>0$, following the reasoning in the beginning~of~\mbox{Section}~\ref{sec:semi-smooth_newton}, we observe that the primal optimality relations \eqref{eq:elasto_discrete_optimality} are equivalent to the existence of a root of the mapping $\mathtt{F}_h\colon Y_h\times V_h\to Y_h\times V_h$, defined by \eqref{sec:semi-smooth_newton.3}, where, by virtue of Proposition~\ref{prop:proximity_of_integral_functionals}, Lemma \ref{lem:relations}(\hyperlink{lem:relations.iii}{iii}), and Example \ref{expl:proximity}(\hyperlink{expl:proximity.ii}{ii}), for every $y_h\in Y_h$ and $v_h\in V_h$, respectively, we have that\vspace{-4.5mm}
    \begin{subequations}\label{eq:elasto_proximity}
    \begin{alignat}{2}\label{eq:elasto_proximity.1}
        \operatorname{prox}_{\gamma_1 G_h}(y_h)&=
        \min\{\tfrac{1+\varepsilon}{1+\varepsilon+\gamma_1},\tfrac{\varepsilon}{\varepsilon+\gamma_1}+\tfrac{\gamma_1}{(\varepsilon+\gamma_1)\vert y_h\vert}\}y_h&&\quad \text{ a.e.\ in }\Omega\,,\\[-0.25mm]
        \operatorname{prox}_{\gamma_2 F_h}(v_h)&=v_h+\smash{\gamma_2}f_h&&\quad \text{ in }\Omega\,.\label{eq:elasto_proximity.2}
    \end{alignat}
    \end{subequations}
    A global Newton derivative selection $\mathtt{J}_{\mathtt{F}_h}\colon Y_h\times V_h\to \mathcal{L}(Y_h\times V_h)$, for every $(y_h,v_h)\in Y_h\times V_h$, is given by \eqref{sec:semi-smooth_newton.4}, where for every $y_h\in Y_h$ and $v_h\in V_h$, respectively, we have that\vspace{-0.5mm}
    \begin{subequations}\label{eq:elasto_proximity_derivatives}
    \begin{alignat}{2}\label{eq:elasto_proximity_derivatives.1}
        \mathtt{J}_{\smash{\operatorname{prox}_{\gamma_1 G_h}}}(y_h)&=\left.\begin{cases}
          \frac{1+\varepsilon}{1+\varepsilon+\gamma_1} \mathbbone_{\mathcal{L}(Y_h)}&\text{ if }\vert  y_h\vert\le 1+\varepsilon+\gamma_1\,,\\
          \left.\begin{aligned}
              &\tfrac{\varepsilon}{\varepsilon+\gamma_1}\mathbbone_{\mathcal{L}(Y_h)}\\&+\tfrac{\gamma_1}{\varepsilon+\gamma_1}\tfrac{1}{\vert y_h\vert}(\mathbbone_{\mathcal{L}(Y_h)}-\tfrac{y_h\otimes y_h}{\vert y_h\vert^2})
          \end{aligned}\right\}&\text{ if } \vert  y_h\vert>1+\varepsilon+\gamma_1\,,
        \end{cases}\right\}&&\quad \text{ a.e.\ in }\Omega\,,\\[-0.25mm]
        \mathtt{J}_{\smash{\operatorname{prox}_{\gamma_2 F_h}}}(v_h)&=\mathbbone_{\mathcal{L}(V_h)}&&\quad \text{ in }\Omega\,.\label{eq:elasto_proximity_derivatives.2}
    \end{alignat}
    \end{subequations}

    The explicit representations in \eqref{eq:elasto_proximity_derivatives}
show that the assumptions of Theorems~\ref{prop:well-posedness}
and \ref{prop:convergence} are satisfied in the present setting.\enlargethispage{5mm}

     \begin{theorem} 
     %[Global well-posedness and local super-linear convergence of Algorithm \ref{alg:semi-smooth_newton}]
     \label{thm:wellposedness_elasto}
        Algorithm~\ref{alg:semi-smooth_newton} employing the representations \eqref{eq:elasto_proximity} and  \eqref{eq:elasto_proximity_derivatives} is globally well-posed and locally super-linearly convergent (with $\alpha_k=1$ for all $k\in \mathbb{N}$) to a root~${(z_h,u_h)\in Y_h\times V_h}$.
   \end{theorem}

    \begin{proof}
    From representation  \eqref{eq:elasto_proximity_derivatives.1}, it follows that 
        $\smash{\tfrac{\varepsilon}{\varepsilon+\gamma_1}}\mathbbone_{\mathcal{L}(Y_h)}\hspace{-0.1em}\preceq \hspace{-0.1em} \mathtt{J}_{\smash{\operatorname{prox}_{\gamma_1 G_h}}}(y_h)\hspace{-0.1em}\preceq\hspace{-0.1em} \smash{\tfrac{1+\varepsilon}{1+\varepsilon+\gamma_1}}\mathbbone_{\mathcal{L}(Y_h)}$~for all $y_h\hspace{-0.1em}\in\hspace{-0.1em}  Y_h$, so that, also using representation  \eqref{eq:elasto_proximity_derivatives.2}, we conclude that
        Assumptions~\ref{ass:wellposedness}~and~\ref{ass:convergence} are satisfied. 
    \end{proof}

    \section{Numerical experiments}\label{sec:experiments}\vspace{-0.5mm}

    \hspace{5mm}In this section, we carry out numerical experiments reviewing the findings on the $\operatorname{prox}$-based semi-smooth Newton methods derived in Section \ref{sec:applications}, \textit{i.e.}, for the TV-minimization problem, the $p$-Dirichlet problem, and the elasto-plastic torsion problem.\vspace{-0.5mm}

    \subsection{Implementation details}\vspace{-0.5mm}

    \hspace{5mm}All experiments were conducted using the finite element software package \texttt{FEniCS}  (version 2019.1.0, \textit{cf}.\  \cite{LoggMardalWells2012}). All graphics were generated using the \texttt{Matplotlib}~library~(version~3.5.1,~\textit{cf}.~\cite{Hunter2007}). 

    \subsubsection{Linear solver}\vspace{-0.5mm}

    \hspace{5mm}The linear systems \eqref{alg:semi-smooth_newton.1} arising in the semi-smooth Newton iteration in Algorithm~\ref{alg:semi-smooth_newton} are solved according to the procedure outlined in Remark~\ref{rem:reduced_problem}. In particular, the symmetric positive definite linear systems \eqref{rem:reduced.4} for the computation of the primal update direction are solved by a sparse Cholesky factorization using \texttt{CHOLMOD}~\cite{cholmod} through the Python package \texttt{scikit-sparse} (version~0.5.0,~\textit{cf}.~\cite{scikit-sparse}). The same direct solver is used for the linear systems arising in the primal semi-smooth Newton method.\vspace{-0.5mm}\enlargethispage{2.5mm}

    \subsubsection{Computation of proximity operators and Newton derivatives}\vspace{-0.5mm} 

    \hspace{5mm}For the TV-minimization problem (\textit{cf}.\ Subsection \ref{subsec:tv-min}) and the elasto-plastic~torsion~\mbox{problem} (\textit{cf}.\ Subsection \ref{subsec:elasto-plastic}), the proximity operator and its Newton derivative are computed using the explicit representation formulas \eqref{eq:rof_proximity_derivatives} and \eqref{eq:elasto_proximity_derivatives}, respectively. In the case of the $p$-Dirichlet~\mbox{problem} (\textit{cf}.\ Subsection \ref{subsec:p-laplace}), the proximity operator is computed according to the procedure~outlined~in Remark~\ref{rem:comput_prox},  while its Newton derivative is evaluated via the implicit representation formula~\eqref{eq:pdirichlet_proximity_derivatives.1}.\vspace{-0.5mm}

    \subsubsection{Globalization strategies}\vspace{-0.5mm} 

\hspace{5mm}In Theorems \ref{thm:wellposedness_tv},
\ref{prop:wellposedness_pdirichlet}, and \ref{thm:wellposedness_elasto},
we established the global well-posedness and~local~\mbox{super-linear} convergence
\hspace{-0.1mm}of \hspace{-0.1mm}the \hspace{-0.1mm}$\operatorname{prox}$-based \hspace{-0.1mm}semi-smooth \hspace{-0.1mm}Newton \hspace{-0.1mm}methods \hspace{-0.1mm}(\textit{cf}.\ \hspace{-0.1mm}Algorithm~\hspace{-0.1mm}\ref{alg:semi-smooth_newton})~\hspace{-0.1mm}for~\hspace{-0.1mm}the~\hspace{-0.1mm}\mbox{TV-minimi-} zation problem, the $p$-Dirichlet problem and the elasto-plastic torsion
problem.~In~this~\mbox{connection}, we globalize these $\operatorname{prox}$-based semi-smooth Newton methods either by combining them with one of
the semi-implicit discretized (primal/dual) $L^2$-gradient-flow initializations~described~in~\mbox{Remark}~\ref{rem:gradient_flow_initializations}, or by incorporating a nonlinear residual-based Armijo--Goldstein-type step-size criterion: 
\begin{itemize}[noitemsep,topsep=2pt,leftmargin=!,labelwidth=\widthof{$\bullet$}]
    \item[$\bullet$] \emph{Strategy 1 ($L^2$-gradient-flow initialization):}
    \hypertarget{Strategy 1}{} First, depending on the growth properties of the discrete energy density
    $\phi_h\colon \Omega\times \mathbb{R}^d\to \mathbb{R}$, exploiting that for the TV-minimization problem, the $p$-Dirichlet problem, and the elasto-plastic torsion problem, the discrete energy density $\psi_h\colon\Omega\times \mathbb{R}\to \mathbb{R}$ is at most quadratic, we run either Algorithm~\ref{alg:primal_gradient_flow} (primal) with zero initial iterate $\widetilde{u}_h^0=0\in V_h$ or
    Algorithm~\ref{alg:dual_gradient_flow} (dual) with zero initial iterate~${\widetilde{z}_h^0=\mathtt{0}_d\in Y_h}$~until~for~some~$\ell_0\in\mathbb N$, there holds $\|\mathtt F_h(\widetilde z_h^{\ell_0+1},
        \widetilde u_h^{\ell_0+1})\|_{Y_h\times V_h}
       < \varepsilon_{\mathtt{init}}^h$, 
    where $\varepsilon_{\mathtt{init}}^h>0$ is a prescribed stopping parameter. 
    If the subsequent semi-smooth Newton iteration fails to converge, the initialization may be strengthened by decreasing $\varepsilon^h_{\mathtt{init}}$.\smallskip

    Second, we run Algorithm~\ref{alg:semi-smooth_newton} with initial iterate $(z_h^0,u_h^0)
       \hspace{-0.1em} \coloneqq\hspace{-0.1em}
        (\widetilde z_h^{\ell_0+1},\widetilde u_h^{\ell_0+1})
        \hspace{-0.1em}\in\hspace{-0.1em} Y_h\times V_h$~and~\mbox{constant} step size $\alpha_k=1$ for all $k\in \mathbb{N}$ until for some $k^*\in\mathbb N$, there holds ${\|\mathtt{F}_h(z_h^{k^*+1},u_h^{k^*+1})\|_{Y_h\times V_h}
        < \varepsilon_{\mathtt{abs}}^h}$; 

    \item[$\bullet$] \emph{Strategy 2 (Armijo--Goldstein backtracking line search):}
    \hypertarget{Strategy 2}{} We run Algorithm~\ref{alg:semi-smooth_newton}
    directly~with~zero initial iterate $(z_h^0,u_h^0)=(\mathtt{0}_d,0)\in Y_h\times V_h$ and step sizes
    $\alpha_k>0$, $k\in \mathbb{N}$, chosen by the residual-based Armijo--Goldstein backtracking
    line search described in Remark~\ref{rem:linesearch} until for~some~$k^*\in\mathbb N$, there holds $\|\mathtt{F}_h(z_h^{k^*+1},u_h^{k^*+1})\|_{Y_h\times V_h}
        < \varepsilon_{\mathtt{abs}}^h$. 
\end{itemize}
  
In all experiments below, for both Strategy~\hyperlink{Strategy 1}{1} and
Strategy~\hyperlink{Strategy 2}{2}, Algorithm~\ref{alg:semi-smooth_newton} is run with
$\varepsilon_{\mathtt{abs}}^h \coloneqq 1\mathrm{e}{-}12$,
$\varepsilon_{\mathtt{rel}}^h \coloneqq 0$, 
$k_{\mathtt{max}}^h \coloneqq 25$ for  Strategy~\hyperlink{Strategy 1}{1}, and
$k_{\mathtt{max}}^h \coloneqq 250$ for  Strategy~\hyperlink{Strategy 2}{2}.
The gradient-flow initializations in Algorithms~\ref{alg:primal_gradient_flow}
and~\ref{alg:dual_gradient_flow} are run~with~$\varepsilon_{\mathtt{init}}^h \coloneqq 0.1$~and~$\ell_{\mathtt{max}}^h \coloneqq +\infty$.

\subsection{TV-minimization problem}

\hspace{5mm}In this test, we compare the $\operatorname{prox}$-based semi-smooth Newton method (\textit{cf}.\ Algorithm~\ref{alg:semi-smooth_newton}) with a canonical primal semi-smooth Newton method applied directly to the Fr\'echet derivative $\mathrm{D}I_h\colon V_h\to V^*_h$ of 
the discrete primal functional \eqref{eq:rof_discrete_primal}, which is globally Newton differentiable with possible (global) Newton derivative selection $\mathtt{J}_{\smash{\mathrm{D}I_h}}\colon  V_h\to  \mathcal{L}(V_h;V^*_h)$, for every $u_h,v_h,w_h\in  V_h$ defined by
\begin{align*}
\langle\mathtt{J}_{\smash{\mathrm{D}I_h}}(u_h) w_h,
v_h
\rangle_{V_h}
\coloneqq
\int_\Omega
\mathtt{J}_{\mathrm{D}\phi_\varepsilon}(\nabla u_h)
\nabla w_h\cdot \nabla v_h
\,\mathrm{d}x
+
\alpha\int_\Omega I_{V_h}\{w_hv_h\}\,\mathrm{d}x\,,
\end{align*}
 where $\mathtt{J}_{\mathrm{D}\phi_\varepsilon}\colon \mathbb{R}^d\to \mathbb{R}^{d\times d}$, for every $t\in \mathbb{R}^d$ defined by\vspace{-1.25mm}
 \begin{align*}
    \mathtt{J}_{\mathrm{D}\phi_\varepsilon}(t)
\coloneqq
\begin{cases}
    \frac{1}{\varepsilon}\mathbbone &\text{ if }\vert t\vert<\varepsilon\,,\\
    \frac{1}{\vert t\vert}\mathtt{P}_t&\text{ if }\vert t\vert\ge \varepsilon\,,
\end{cases}\\[-6.5mm]\notag
 \end{align*}
 denotes a possible (global) Newton derivative of $\mathrm{D}\phi_\varepsilon\colon \mathbb{R}^d\to \mathbb{R}^d$.

\begin{algorithm}[Primal semi-smooth  Newton method for discrete TV-minimization]
\label{alg:Newton-disc}
Let $\varepsilon>0$ be a regularization parameter, let
$\varepsilon_{\mathtt{abs}}^h>0$ be a stopping parameter, let $u_h^0\in V_h$ be an initial iterate, and let $k_{\mathtt{max}}^h\in \mathbb{N}\cup\{+\infty\}$ be a maximal 
number of iterations. Then, for ${k=0,\ldots,\smash{k_{\mathtt{max}}^h}}$, perform the following iteration loop:
\begin{itemize}[noitemsep,topsep=2pt,leftmargin=!,labelwidth=\widthof{(2)}]
\item[(1)] \hypertarget{alg:Newton-disc.1}{} Compute the primal update direction $\delta \smash{u_h^k\in V_h}$ such that
\begin{align*}
    \mathtt{J}_{\smash{\mathrm{D}I_h}}(u_h^k)\delta u_h^k = -\mathrm{D}I_h(u_h^k)
\quad\text{ in }V_h^*\,,
\end{align*} 
and the updated iterate $u_h^{k+1}\coloneqq u_h^k+\alpha_k\delta u_h^k\hspace{-0.1em}\in \hspace{-0.1em}V_h$, where $\alpha_k\hspace{-0.1em}>\hspace{-0.1em}0$ is a (possibly~\mbox{variable})~step~size; 

\item[(2)]  \hypertarget{alg:Newton-disc.2}{} If %$\|\mathrm{D}I_h^{\varepsilon}(u_h^{k+1})\|_{V_h^\ast}<\max\{\smash{\varepsilon_{\mathtt{abs}}^h},\smash{\varepsilon_{\mathtt{rel}}^h}\|\mathrm{D}I_h^{\varepsilon}(u_h^0)\|_{V_h^\ast}\}$, 
$\|\mathrm{D}I_h(u_h^{k+1})\|_{V_h^\ast}<\smash{\varepsilon_{\mathtt{abs}}^h}$, 
then \textup{STOP}; otherwise, set $k\to k+1$ and continue with~Step~(\hyperlink{alg:Newton-disc.1}{1}).
\end{itemize}
\end{algorithm}

We consider the following  benchmark example for the TV-minimization problem that provides a discontinuous primal solution and a Lipschitz continuous dual solution (see, \textit{e.g.}, \cite[Expl.~10.4]{Bartels2015}).

\begin{example}\label{ex:41}
Let $\Omega = (-1,1)^{d}$, $d\in \mathbb{N}$, $\Gamma_{D} = \partial\Omega$, $\alpha>0$, $0<r<1$, and $g = \chi_{B_{r}^d(0)}\in L^2(\Omega)$.
Then, \hspace{-0.1mm}the \hspace{-0.1mm}unique \hspace{-0.1mm}primal \hspace{-0.1mm}solution \hspace{-0.1mm}$u\hspace{-0.1em}\in\hspace{-0.1em} BV(\Omega)\cap L^2(\Omega)$ \hspace{-0.1mm}with \hspace{-0.1mm}$u\hspace{-0.1em}=\hspace{-0.1em}0$ \hspace{-0.1mm}a.e.\ \hspace{-0.1mm}on \hspace{-0.1mm}$\partial\Omega$,~\hspace{-0.1mm}\textit{i.e.},~\hspace{-0.1mm}a~\hspace{-0.1mm}\mbox{minimizer}~\hspace{-0.1mm}of~\hspace{-0.1mm}\eqref{eq:rof_primal}, 
and a dual solution $z\in W^2(\operatorname{div};\Omega)$, \textit{i.e.}, a maximizer of \eqref{eq:rof_dual},  are given via
\begin{subequations} 
\begin{alignat}{2} 
  u &\coloneqq \max\bigl\{ 0, 1 - \tfrac{d}{\alpha r}\bigr\}\chi_{B_{r}^d(0)}&&\quad \text{ a.e.\ in }\Omega\,,\\
  z &\coloneqq  
      -\tfrac{r}{\max\{r,\vert \cdot\vert\}^2}\min\left\{1,\tfrac{\alpha r}{d}\right\}\operatorname{id}_{\mathbb{R}^d} &&\quad\text{ in }\Omega\,. 
\end{alignat} 
\end{subequations}
\end{example}

More precisely, in the following, we consider the case $d=3$, $\alpha=10$, $r=\frac{1}{2}$, and a sequence of triangulations $\{\mathcal{T}_{h_i}\}_{i=0,\ldots,6}$ (\textit{cf}.\ Figure \ref{fig:3D_triang}), each consisting of $6\cdot 8^i$ tetrahedra and~$(2^i+1)^3$~vertices,\linebreak obtained by red-refinement starting from a standard Kuhn triangulation $\mathcal{T}_0$ (\textit{cf}.\ \cite{Kuhn1960})~of~${\overline{\Omega}\hspace{-0.15em}=\hspace{-0.15em}[-1,1]^3}$, $\gamma_1=\gamma_2=1$, and $\varepsilon=h$ (which preserves the quasi-optimal error~decay~rates~for~element-wise affine approximations of the TV-minimization problem under uniform mesh-refinement, \textit{cf}.~\cite{ChambollePock2020,Bartels2021,BartelsKaltenbach2022}).

\begin{figure}[H]
    \centering
    \includegraphics[width=\linewidth]{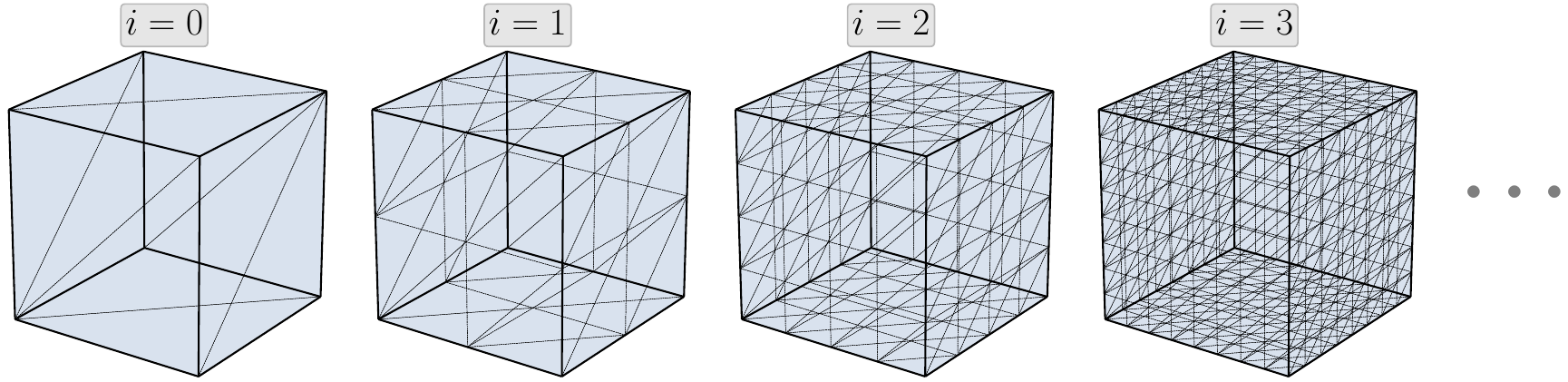}
    \caption{Sequence of triangulations $\{\mathcal{T}_{h_i}\}_{i=0,\ldots,6}$ obtained by red-refinement of a standard~Kuhn triangulation $\mathcal{T}_0$ (\textit{cf}.\ \cite{Kuhn1960}) of $\overline{\Omega}=[-1,1]^3$.}
    \label{fig:3D_triang}
\end{figure}

For the above experimental setup, Figures \ref{fig:tv1} and \ref{fig:tv2} compare the $\operatorname{prox}$-based semi-smooth Newton
method (\textit{cf}.\ Algorithm~\ref{alg:semi-smooth_newton} employing the representations \eqref{eq:rof_proximity} and  \eqref{eq:rof_proximity_derivatives}) with the
 primal semi-smooth Newton method (\textit{cf}.\ Algorithm~\ref{alg:Newton-disc}) under mesh-refinement.~To enable a direct comparison, we report the common $\operatorname{prox}$-based nonlinear residual~$\|\mathtt{F}_h(z_h^k,u_h^k)\|_{(L^2(\Omega))^d\times L^2(\Omega)}$,~${k\in \mathbb{N}_0}$,
for both methods. 
The convergent runs exhibit fast local residual decay for both methods. Therefore, the
main difference is not the local decay rate, but the robustness of the iteration: for both the
mass-lumped (\hyperlink{ML}{ML}) and non-mass-lumped (\hyperlink{NML}{NML}) variants and both globalization strategies
(\textit{cf}.\ Strategies~\hyperlink{Strategy 1}{1}~and~\hyperlink{Strategy 2}{2}),
the $\operatorname{prox}$-based semi-smooth Newton method remains effective for a larger
range of refined triangulations.

\begin{figure}[H]
    \centering
    \includegraphics[width=\linewidth]{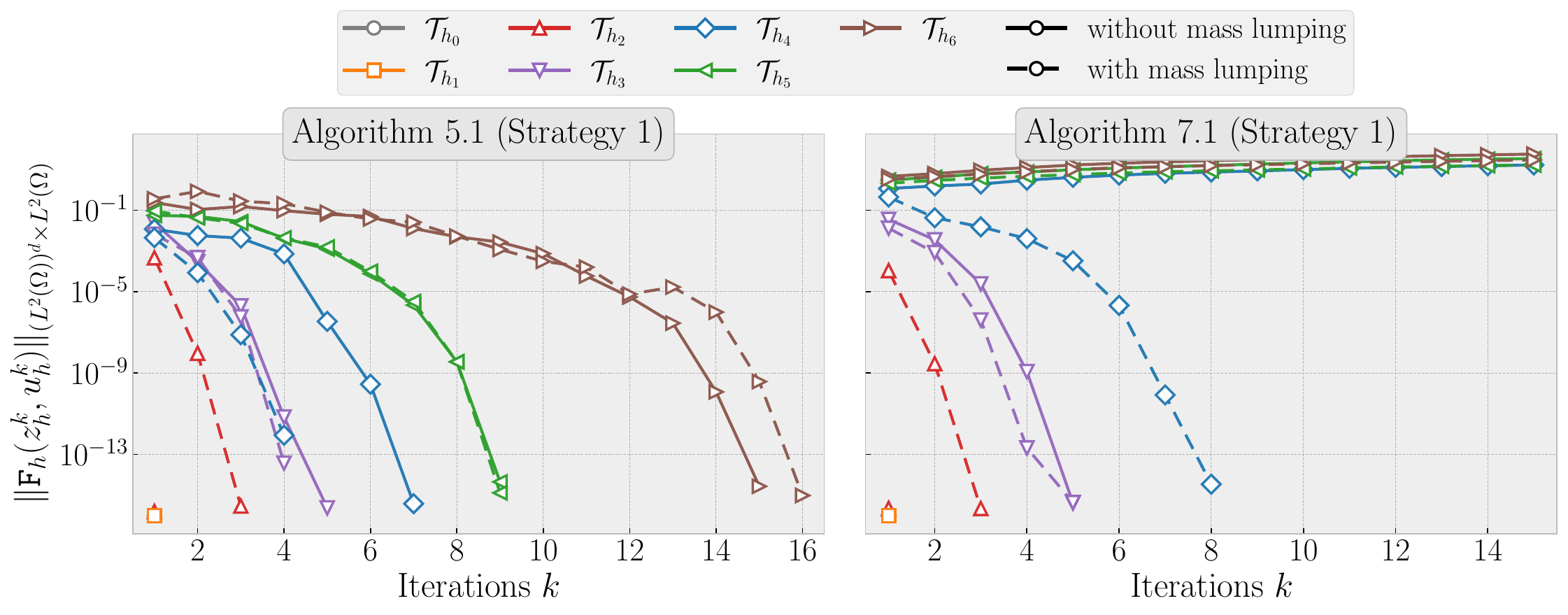}\vspace{-1.5mm}
   \caption{
   Residual decay for the TV-minimization problem (\textit{cf}.\ Subsection \ref{subsec:tv-min}) using
Strategy~\protect\hyperlink{Strategy 1}{1}, \textit{i.e.}, the $L^2$-gradient-flow initialization described in Remark \ref{rem:gradient_flow_initializations}; each
measured by the $\operatorname{prox}$-based nonlinear residual
$\|\mathtt{F}_h(z_h^k,u_h^k)\|_{(L^2(\Omega))^d\times L^2(\Omega)}$, $k\in\mathbb{N}_0$:
\textit{left:} $\operatorname{prox}$-based semi-smooth Newton method (\textit{cf}.\ Algorithm \ref{alg:semi-smooth_newton} employing the representations \eqref{eq:rof_proximity} and  \eqref{eq:rof_proximity_derivatives}); \textit{right:} primal semi-smooth Newton method
(\textit{cf}.\ Algorithm \ref{alg:Newton-disc}). The $\operatorname{prox}$-based semi-smooth Newton method is less sensitive to mesh-refinement and to the use
of mass lumping.}
    \label{fig:tv1}
\end{figure}\vspace{-5mm}

\begin{figure}[H]
    \centering
    \includegraphics[width=\linewidth]{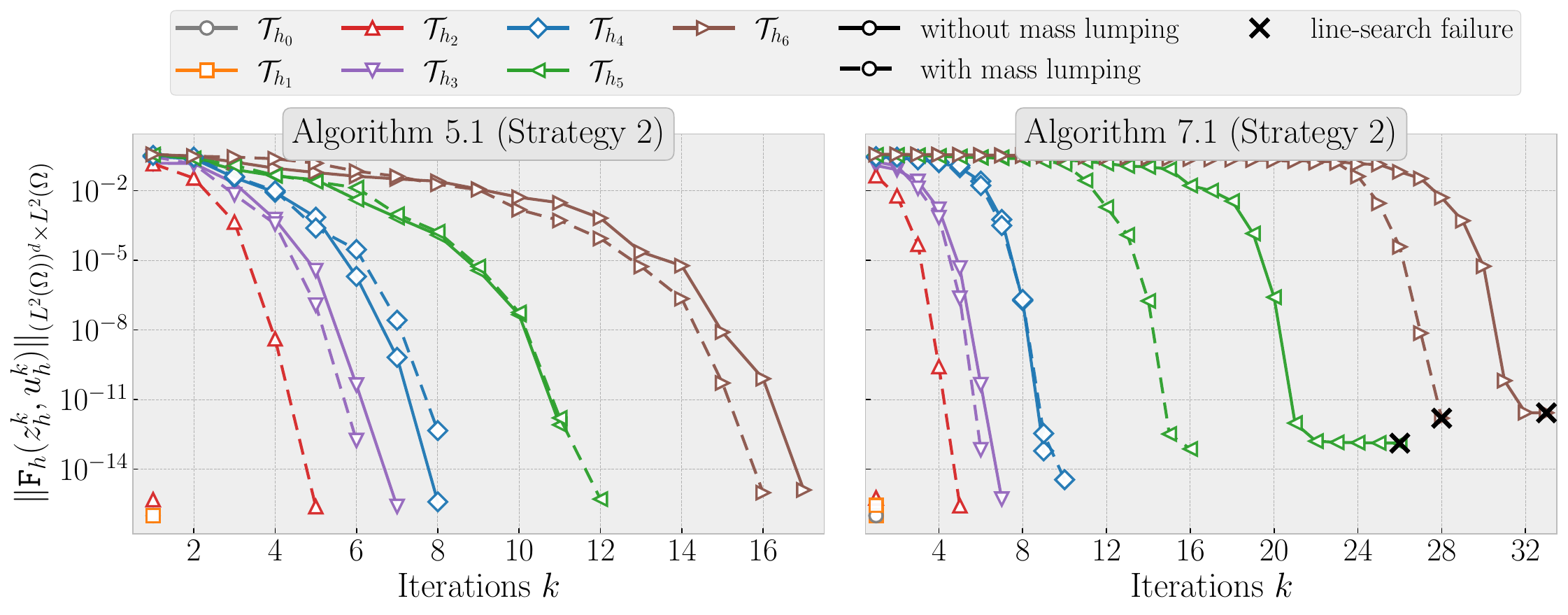}\vspace{-1.5mm}
    \caption{Residual decay for the TV-minimization problem (\textit{cf}.\ Subsection \ref{subsec:tv-min}) using
Strategy~\protect\hyperlink{Strategy 2}{2}, \textit{i.e.}, the residual-based Armijo--Goldstein line search; each measured by the $\operatorname{prox}$-based nonlinear residual
$\|\mathtt{F}_h(z_h^k,u_h^k)\|_{(L^2(\Omega))^d\times L^2(\Omega)}$, $k\in\mathbb{N}_0$:
\textit{left:} $\operatorname{prox}$-based semi-smooth Newton method (\textit{cf}.\ Algorithm \ref{alg:semi-smooth_newton} employing the representations \eqref{eq:rof_proximity} and  \eqref{eq:rof_proximity_derivatives}); \textit{right:}  primal semi-smooth~Newton method
(\textit{cf}.\ Algorithm \ref{alg:Newton-disc}). The $\operatorname{prox}$-based semi-smooth Newton method is less sensitive to mesh-refinement and to the use
of mass lumping.}
    \label{fig:tv2}
\end{figure}
\pagebreak

\if0
\begin{table}[H]
\centering
\renewcommand{\arraystretch}{1.05}
\setlength{\tabcolsep}{5pt}
\begin{tabular}{c|ccccccc}
\toprule
  & $\mathcal{T}_{0}$ & $\mathcal{T}_{1}$ & $\mathcal{T}_{2}$ & $\mathcal{T}_{3}$ & $\mathcal{T}_{4}$ & $\mathcal{T}_{5}$ & $\mathcal{T}_{6}$ \\
\midrule
$h_i$ & $3.4641\mathrm{e}{+}0$ & $1.7321\mathrm{e}{+}0$ & $8.6603\mathrm{e}{-}1$ & $4.3301\mathrm{e}{-}1$ & $2.1651\mathrm{e}{-}1$ & $1.0825\mathrm{e}{-}1$ & $5.4127\mathrm{e}{-}2$ \\
$\vert \mathcal{N}_i\vert$ & $8$ & $27$ & $125$ & $729$ & $4913$ & $35937$ & $274625$ \\
$\vert\mathcal{T}_{h_i}\vert$ & $6$ & $48$ & $384$ & $3072$ & $24576$ & $196608$ & $1572864$ \\
\bottomrule
\end{tabular}\vspace{-1mm}
\caption{Mesh sizes, numbers of vertices, and numbers of tetrahedra.}
\label{tab:boxmesh_triangulations}
\end{table}\fi
%\newpage

\subsection{$p$-Dirichlet problem}\label{subsec:experiments_pdirichlet}

\hspace{5mm}In this test, we compare the $\operatorname{prox}$-based semi-smooth Newton method (\textit{cf}.\ Algorithm~\ref{alg:semi-smooth_newton}) with the classical primal Newton method applied directly to 
the discrete primal functional~\eqref{eq:pdirichlet_primal_discrete}, which  is twice continuously Fr\'echet differentiable with second Fr\'echet derivative $\mathrm{D}^2 I_h\colon  V_h\to  \mathcal{L}(V_h;V^*_h)$, for every $u_h,v_h,w_h\in  V_h$ defined by
\begin{align*}
\langle \mathrm{D}^2I_h(u_h) w_h,
v_h
\rangle_{V_h}
\coloneqq
\int_\Omega
\mathrm{D}^2\phi_\varepsilon(\nabla u_h)
\nabla w_h\cdot \nabla v_h
\,\mathrm{d}x\,,
\end{align*}
 where the second derivative $\mathrm{D}^2\phi_\varepsilon\colon \mathbb{R}^d\to \mathbb{R}^{d\times d}$, for every $t\in \mathbb{R}^d$, is given via
 \begin{align*}
    \mathrm{D}^2\phi_\varepsilon(t)
\coloneqq
(|t|^2+\varepsilon^2)^{\smash{\frac{p-2}{2}}} \mathbbone
+
(p-2)
(|t|^2+\varepsilon^2)^{\smash{\frac{p-4}{2}}}
 t\otimes t\,.
 \end{align*}

\begin{algorithm}[Primal classical Newton method for the discrete $p$-Dirichlet problem]
\label{alg:p_laplace_Newton}
Let $\varepsilon>0$ be a regularization parameter, let
$\varepsilon_{\mathtt{abs}}^h>0$ be a stopping parameter, let $u_h^0\in V_h$ be an initial iterate, and let $k_{\mathtt{max}}^h\in \mathbb{N}\cup\{+\infty\}$ be a maximal 
number of iterations. Then, for $k=0,\ldots,\smash{k_{\mathtt{max}}^h}$, perform the following iteration loop:
\begin{itemize}[noitemsep,topsep=2pt,leftmargin=!,labelwidth=\widthof{(2)}]
\item[(1)] \hypertarget{alg:Newton-disc.1}{} Compute the primal update direction $\delta \smash{u_h^k\in V_h}$ such that\vspace{-0.5mm}
\begin{align*}
    \mathrm{D}^2I_h(u_h^k)\delta u_h^k = -\mathrm{D}I_h(u_h^k)
\quad\text{ in }V_h^*\,,\\[-6mm]\notag
\end{align*} 
and the updated iterate $u_h^{k+1}\hspace{-0.1em}\coloneqq \hspace{-0.1em}u_h^k+\alpha_k\delta u_h^k\hspace{-0.1em}\in\hspace{-0.1em} V_h$, where $\alpha_k\hspace{-0.1em}>\hspace{-0.1em}0$~is~a~(\smash{possibly}~\mbox{variable})~\mbox{step~size}; 

\item[(2)]  \hypertarget{alg:Newton-disc.2}{} If %$\smash{\|\mathrm{D}I_h^{\varepsilon}(u_h^{k+1})\|_{V_h^\ast}}<\max\{\smash{\varepsilon_{\mathtt{abs}}^h},\smash{\varepsilon_{\mathtt{rel}}^h}\|\mathrm{D}I_h^{\varepsilon}(u_h^{k+1})\|_{V_h^\ast}\}$, 
$\smash{\|\mathrm{D}I_h(u_h^{k+1})\|_{V_h^\ast}}<\smash{\varepsilon_{\mathtt{abs}}^h}$, 
then \textup{STOP}; otherwise, set $k\to k+1$ and continue~with~Step~(\hyperlink{alg:Newton-disc.1}{1}).
\end{itemize}
\end{algorithm}

We consider the following benchmark example for the $p$-Dirichlet problem, 
%with constant right-hand side on the unit ball, 
which provides a
radially symmetric primal solution and an affine dual solution (see, \textit{e.g.}, \cite[§3]{Kawohl1990}).

\begin{example}\label{ex:pdirichlet_unit_ball}
Let $\Omega \hspace{-0.1em}=\hspace{-0.1em} B_1^d(0)$, $d\hspace{-0.1em}\in\hspace{-0.1em}\mathbb{N}$, $\Gamma_D\hspace{-0.1em}=\hspace{-0.1em}\partial\Omega$,
$p\hspace{-0.1em}\in\hspace{-0.1em}(1,+\infty)$,  and $f\hspace{-0.1em}\equiv \hspace{-0.1em}1\hspace{-0.1em}\in \hspace{-0.1em}L^{p'}(\Omega)$.~Then,~the~unique primal \hspace{-0.175mm}solution \hspace{-0.175mm}$u\hspace{-0.18em}\in\hspace{-0.18em} W_D^{1,p}(\Omega)$, \hspace{-0.175mm}\textit{i.e.},
\hspace{-0.175mm}a \hspace{-0.175mm}minimizer \hspace{-0.175mm}of \hspace{-0.175mm}\eqref{eq:pdirichlet_primal}, 
\hspace{-0.175mm}and \hspace{-0.175mm}the \hspace{-0.175mm}unique 
\hspace{-0.175mm}dual \hspace{-0.175mm}solution~\hspace{-0.175mm}${z\hspace{-0.18em}\in\hspace{-0.18em} \smash{W^{p'}_N(\operatorname{div};\Omega)}}$,
\textit{i.e.}, the maximizer of \eqref{eq:pdirichlet_dual}, are given
via
\begin{subequations}
\begin{alignat}{2} 
    u
    &\coloneqq
    \tfrac{p-1}{p}d^{-\smash{\frac{1}{p-1}}}
    (1-|\cdot|^{\smash{p'}})
    &&\quad \text{ in } \Omega\,,\\
    z
    &\coloneqq
    -\tfrac{1}{d}\,\mathrm{id}_{\mathbb R^d}
    &&\quad \text{ in } \Omega\,. 
\end{alignat}
\end{subequations}
\end{example}

More precisely, in the following, we consider the case $d=2$, $p\in \{1.1,100\}$,~a~\mbox{sequence}~of triangulations
 $\{\mathcal{T}_{h_i}\}_{i=0,\ldots,9}$ of $\overline{\Omega}=K_1^2(0)$ (\textit{cf}.\ Figure~\ref{fig:2D_triang} and Table~\ref{tab:tscherner_triangulations}, for the~\mbox{corresponding}~mesh~data), obtained~according to \cite[Sec.\ 5]{BartelsTscherner2025}, $\gamma_1=1$ if $p=1.1$, $\gamma_1=\frac{1}{10}$ if $p=100$, $\gamma_2=1$, and $\varepsilon=h^2$.

\begin{figure}[H]
    \centering
    \includegraphics[width=\linewidth]{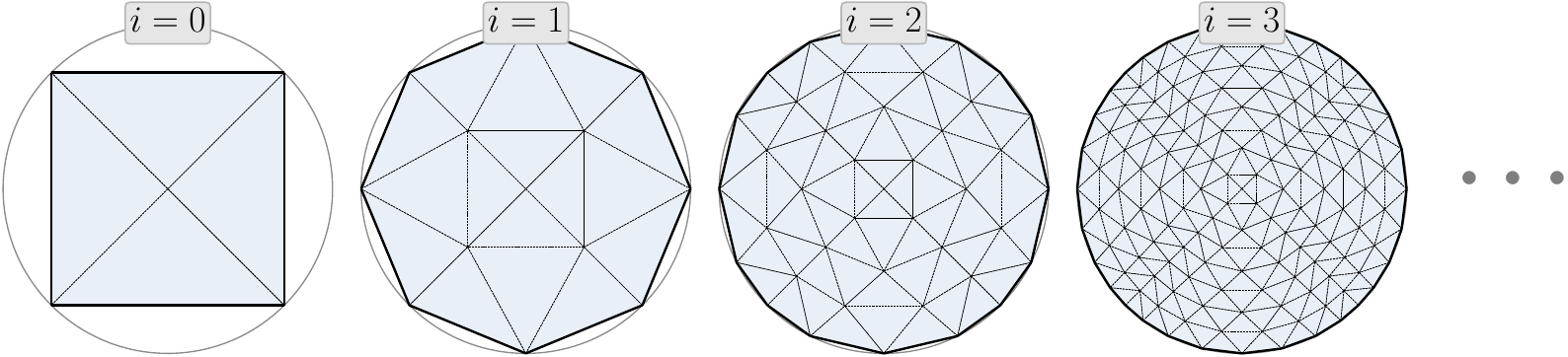}\vspace{-1mm}
    \caption{Sequence \hspace{-0.1mm}of \hspace{-0.1mm}triangulations
\hspace{-0.1mm}$\{\mathcal{T}_{h_i}\}_{i=0,\ldots,9}$ \hspace{-0.1mm}obtained \hspace{-0.1mm}according \hspace{-0.1mm}to \hspace{-0.1mm}\cite[\hspace{-0.1mm}Sec.~\hspace{-0.1mm}5]{BartelsTscherner2025}~\hspace{-0.1mm}on~\hspace{-0.1mm}${\overline{\Omega}\hspace{-0.05em}=\hspace{-0.05em}K_1^2(0)}$.}
    \label{fig:2D_triang}
\end{figure}\vspace{-2.5mm}

\begin{table}[H]
\centering
\renewcommand{\arraystretch}{1.05}
\setlength{\tabcolsep}{3.75pt}
\begin{tabular}{c|ccccccccc}
\toprule
 & $\mathcal{T}_{h_1}$ & $\mathcal{T}_{h_2}$ & $\mathcal{T}_{h_3}$ & $\mathcal{T}_{h_4}$ & $\mathcal{T}_{h_5}$ & $\mathcal{T}_{h_6}$ & $\mathcal{T}_{h_7}$ & $\mathcal{T}_{h_8}$ & $\mathcal{T}_{h_9}$ \\
\midrule
$h_i$ & $7.65\mathrm{e}{-}1$ & $4.74\mathrm{e}{-}1$ & $2.47\mathrm{e}{-}1$ & $1.25\mathrm{e}{-}1$ & $6.25\mathrm{e}{-}2$ & $3.12\mathrm{e}{-}2$ & $1.56\mathrm{e}{-}2$ & $7.81\mathrm{e}{-}3$ & $3.91\mathrm{e}{-}3$ \\
$\vert \mathcal{N}_i\vert$ &  $13$ & $41$ & $145$ & $545$ & $2,\!113$ & $8,\!321$ & $33,\!025$ & $131,\!585$ & $525,\!313$ \\
$\vert\mathcal{T}_{h_i}\vert$ &  $16$ & $64$ & $256$ & $1,\!024$ & $4,\!096$ & $16,\!384$ & $65,\!536$ & $262,\!144$ & $1,\!048,\!576$ \\
\bottomrule
\end{tabular}\vspace{-1mm}
\caption{Mesh data for the triangulations $\{\mathcal{T}_{h_i}\}_{i=1,\ldots,9}$ obtained  according to \cite[Sec.\ 5]{BartelsTscherner2025}~on ${\overline{\Omega}=K_1^2(0)}$.}
\label{tab:tscherner_triangulations}
\end{table}
 
For the above experimental setup, Figures \ref{fig:plaplace1} and \ref{fig:plaplace2} compare the $\operatorname{prox}$-based semi-smooth Newton
method (\textit{cf}.\ Algorithm~\ref{alg:semi-smooth_newton} employing the representations \eqref{eq:pdirichlet_proximity} and  \eqref{eq:pdirichlet_proximity_derivatives}) with the
classical primal Newton method (\textit{cf}.\ Algorithm~\ref{alg:p_laplace_Newton}) under mesh-refinement. To
enable a direct comparison, we again report the common $\operatorname{prox}$-based nonlinear residual $\|\mathtt{F}_h(z_h^k,u_h^k)\|_{(L^2(\Omega))^d\times L^2(\Omega)}$, $k\in\mathbb{N}_0$, 
for both methods.  
Again, the convergent runs exhibit fast local residual decay~for~both~methods. Therefore,
the main difference is not the local decay rate, but the robustness of the iteration:\linebreak for both
values of $p\in\{1.1,100\}$ and both globalization strategies
(\textit{cf}.\ Strategies~\hyperlink{Strategy 1}{1}~and~\hyperlink{Strategy 2}{2}),~the~$\operatorname{prox}$-based semi-smooth Newton method remains effective for a larger
range~of~refined~triangulations.\enlargethispage{6mm}\vspace{-1mm}

\begin{figure}[H]
    \centering
    \includegraphics[width=\linewidth]{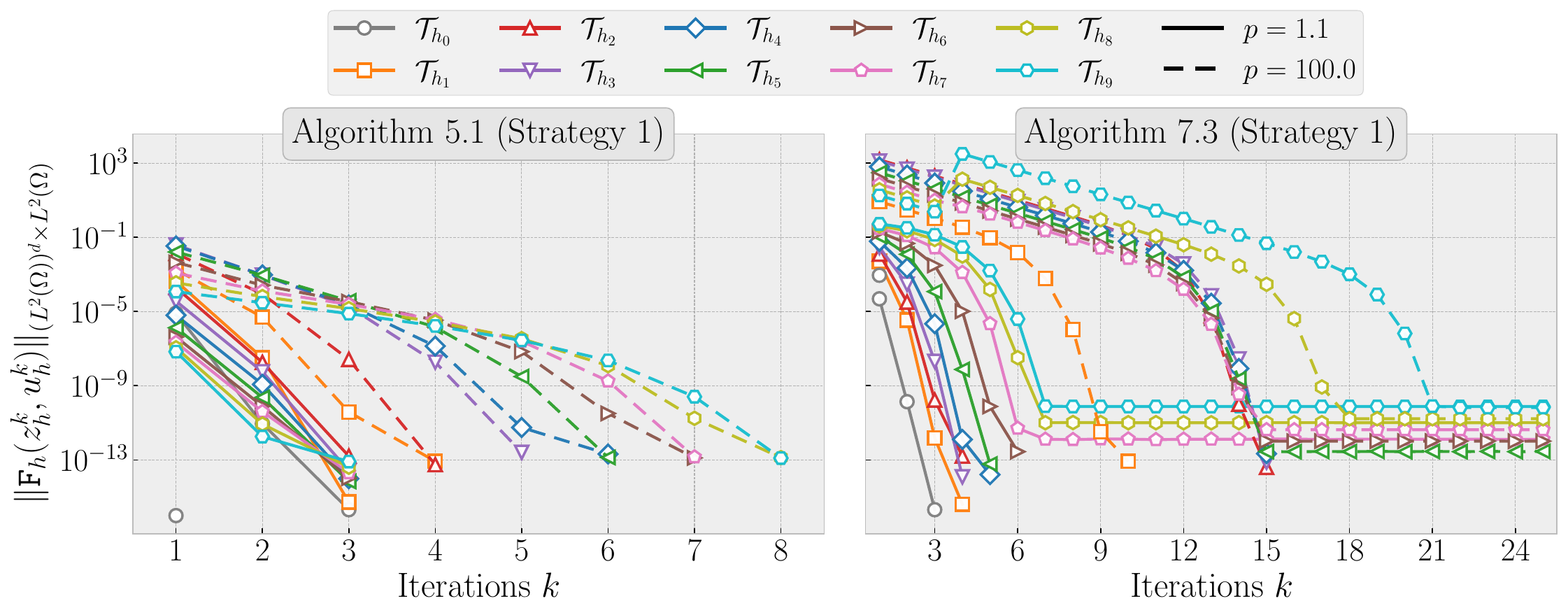}\vspace{-1.75mm}
   \caption{Residual decay for the $p$-Dirichlet problem (\textit{cf}.\ Subsection 6.2) using
Strategy~\protect\hyperlink{Strategy 1}{1}, \textit{i.e.}, the $L^2$-gradient-flow initialization described in Remark \ref{rem:gradient_flow_initializations}; each
measured by the $\operatorname{prox}$-based nonlinear residual
$\|\mathtt{F}_h(z_h^k,u_h^k)\|_{(L^2(\Omega))^d\times L^2(\Omega)}$, $k\in\mathbb{N}_0$:
\textit{left:} $\operatorname{prox}$-based semi-smooth Newton method (\textit{cf}.\ Algorithm \ref{alg:semi-smooth_newton} employing the representations \eqref{eq:pdirichlet_proximity} and  \eqref{eq:pdirichlet_proximity_derivatives}); \textit{right:} classical primal Newton method
(\textit{cf}.\ Algorithm \ref{alg:p_laplace_Newton}). The $\operatorname{prox}$-based semi-smooth Newton method is less sensitive to mesh-refinement. For finer triangulations, the $\operatorname{prox}$-based nonlinear residual of the primal Newton method reaches a numerical floor caused by the second equation \eqref{sec:semi-smooth_newton.2.2} in the discrete proximal optimality conditions \eqref{sec:semi-smooth_newton.2} and its $L^2$-Riesz representation; this is not observed when the flux is updated as an independent variable as in the $\operatorname{prox}$-based~\mbox{semi-smooth}~Newton~method.}
    \label{fig:plaplace1}
\end{figure}\vspace{-5mm} 

\begin{figure}[H]
    \centering
    \includegraphics[width=\linewidth]{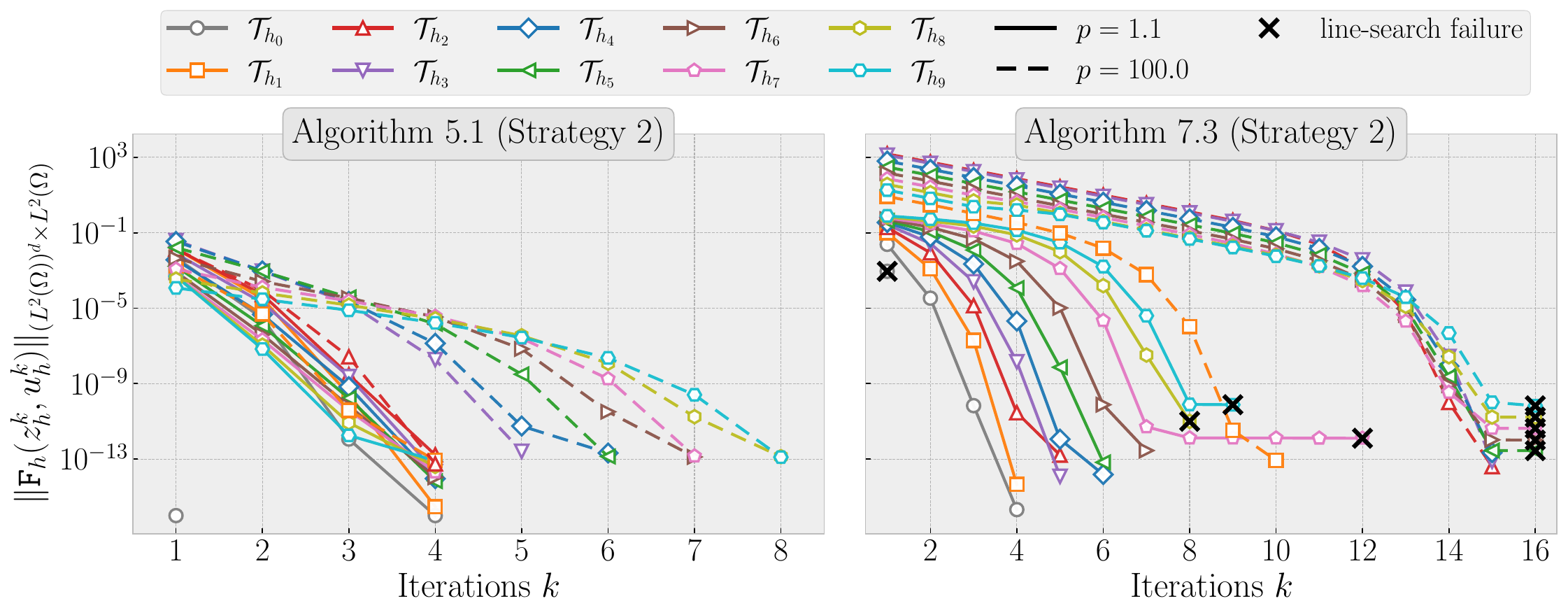}\vspace{-1.75mm}
    \caption{Residual decay for the $p$-Dirichlet problem (\textup{cf}.\ Subsection \ref{subsec:p-laplace}) using
Strategy~\protect\hyperlink{Strategy 2}{2}, \textit{i.e.}, the residual-based Armijo--Goldstein line search; each measured by the
 $\operatorname{prox}$-based nonlinear residual
$\|\mathtt{F}_h(z_h^k,u_h^k)\|_{(L^2(\Omega))^d\times L^2(\Omega)}$, $k\in\mathbb{N}_0$:
\textit{left:} $\operatorname{prox}$-based semi-smooth Newton method (\textit{cf}.\ Algorithm \ref{alg:semi-smooth_newton} employing the representations \eqref{eq:pdirichlet_proximity} and  \eqref{eq:pdirichlet_proximity_derivatives}); \textit{right:} classical primal Newton method
(\textit{cf}.\ Algorithm \ref{alg:p_laplace_Newton}). For $p=100$, the second $L^2$-gradient-flow iterate~is~used~as~\mbox{initial}~\mbox{iterate}. The $\operatorname{prox}$-based semi-smooth Newton method is less sensitive to mesh-refinement.}
    \label{fig:plaplace2}
\end{figure}
\pagebreak

\newpage
\subsection{Elasto-plastic torsion problem}

\hspace{5mm}In this test, we compare the $\operatorname{prox}$-based semi-smooth Newton
method (\textit{cf}.\ Algorithm~\ref{alg:semi-smooth_newton}) with a
canonical dual semi-smooth Newton method applied directly to the first-order optimality system of the discrete dual functional \eqref{eq:elasto_discrete_dual}.
More precisely, the variation
$\mathtt{G}_h\colon Y_h\times V_h\to (Y_h\times V_h)^*$,\linebreak for every
$(z_h,u_h),(y_h,v_h)\in Y_h\times V_h$ defined by
\begin{align*}
\langle
\mathtt{G}_h(z_h,u_h),(y_h,v_h)
\rangle_{Y_h\times V_h}
&\coloneqq
\int_\Omega
\mathrm{D}\varphi_\varepsilon^*(z_h)\cdot y_h\,\mathrm{d}x
\\&\quad+
\int_\Omega
I_{V_h}\{u_h\operatorname{div}_h y_h\}\,\mathrm{d}x
+
\int_\Omega
I_{V_h}\{(\operatorname{div}_h z_h+f_h)v_h\}\,\mathrm{d}x\,,
\end{align*}
is globally Newton differentiable with
possible (global) Newton derivative selection
$\mathtt{J}_{\mathtt{G}_h}\colon
Y_h\times V_h\to \mathcal{L}(Y_h\times V_h;(Y_h\times V_h)^*)$, for every
$(z_h,u_h),(\xi_h,w_h),(y_h,v_h)\in Y_h\times V_h$ defined by
\begin{align*}
\langle
\mathtt{J}_{\mathtt{G}_h}(z_h,u_h)(\xi_h,w_h),
(y_h,v_h)\rangle_{Y_h\times V_h}
\coloneqq
&\int_\Omega
\mathtt{J}_{\mathrm{D}\varphi_\varepsilon^*}(z_h)\xi_h\cdot y_h
\,\mathrm{d}x
\\
&+
\int_\Omega
I_{V_h}\{w_h\operatorname{div}_h y_h\}\,\mathrm{d}x
+
\int_\Omega
I_{V_h}\{v_h\operatorname{div}_h \xi_h\}\,\mathrm{d}x\, ,
\end{align*}
where $\mathtt{J}_{\smash{\mathrm{D}\varphi_\varepsilon^*}}\colon
\mathbb{R}^d\to\mathbb{R}^{d\times d}$, for every $t\in\mathbb{R}^d$ 
defined by
\begin{align*}
\mathtt{J}_{\smash{\mathrm{D}\varphi_\varepsilon^*}}(t)
\coloneqq
\begin{cases}
(1+\varepsilon)\mathbbone
&\text{ if } |t|<1\,,\\
\varepsilon\mathbbone+\smash{\frac{1}{|t|}}\mathtt{P}_t
&\text{ if } |t|\ge 1\,,
\end{cases} 
\end{align*}
denotes a possible (global) Newton derivative selection of
$\mathrm{D}\varphi_\varepsilon^*\colon\mathbb{R}^d\to\mathbb{R}^d$.

\begin{algorithm}[Dual semi-smooth Newton method for discrete~\mbox{elasto-plastic torsion}]
\label{alg:dual_Newton_elasto}
Let $\varepsilon>0$ be a regularization parameter, let
$\varepsilon_{\mathtt{abs}}^h>0$ be a stopping
parameter, let $(z_h^0,u_h^0)\in Y_h\times V_h$ be an initial iterate, and
let $k_{\mathtt{max}}^h\hspace{-0.1em}\in\hspace{-0.1em}\mathbb{N}\cup\{+\infty\}$ be a maximal number of
iterations.~Then,~for~${k\hspace{-0.1em}=\hspace{-0.1em}0,\ldots,\smash{k_{\mathtt{max}}^h}}$, perform the
following iteration loop:
\begin{itemize}[noitemsep,topsep=2pt,leftmargin=!,labelwidth=\widthof{(2)}]
\item[(1)] \hypertarget{alg:dual_Newton_elasto.1}{}
Compute the dual update direction
$(\delta z_h^k,\delta u_h^k)\in Y_h\times V_h$ such that
\begin{align*}
    \mathtt{J}_{\mathtt{G}_h}(z_h^k,u_h^k)
    (\delta z_h^k,\delta u_h^k)
    =
    -\mathtt{G}_h(z_h^k,u_h^k)
    \quad\text{in }(Y_h\times V_h)^*\,,
\end{align*}
and the updated iterate $\smash{(z_h^{k+1},u_h^{k+1})
    \coloneqq
    (z_h^k,u_h^k)+\alpha_k(\delta z_h^k,\delta u_h^k)
    \in Y_h\times V_h}$,
where $\alpha_k>0$ is a (possibly variable) step size;

\item[(2)] \hypertarget{alg:dual_Newton_elasto.2}{}
If %$\|\mathtt{G}_h(z_h^{k+1},u_h^{k+1})\|_{(Y_h\times V_h)^*} < \max\{ \varepsilon_{\mathtt{abs}}^h, \varepsilon_{\mathtt{rel}}^h \|\mathtt{G}_h(z_h^0,u_h^0)\|_{(Y_h\times V_h)^*}\}$,  
$\|\mathtt{G}_h(z_h^{k+1},u_h^{k+1})\|_{(Y_h\times V_h)^*}
<
\varepsilon_{\mathtt{abs}}^h$, 
then \textup{STOP}; otherwise, set $k\to k+1$ and continue with Step (\hyperlink{alg:dual_Newton_elasto.1}{1}).
\end{itemize}
\end{algorithm}

We consider the following benchmark example for the elasto-plastic torsion
problem, which provides a radially symmetric primal solution that
exhibits an elastic-plastic transition for~\mbox{sufficiently} large loads, together
with an affine  dual solution (see, \textit{e.g.},
\cite[Expl.\ 2.3.2, p.\ 122]{GlowinskiLionsTremolieres1981}).

\begin{example}\label{ex:elasto_plastic_torsion_ball_d}
Let $\Omega \coloneqq B_r^d(0)$, $d\in\mathbb N$, $r>0$,
$\Gamma_D\coloneqq \partial\Omega$,
and $f\equiv C_f\in L^1(\Omega)$ with $C_f>0$.~Then,\linebreak the \hspace{-0.1mm}unique \hspace{-0.1mm}primal \hspace{-0.1mm}solution
\hspace{-0.1mm}$u\hspace{-0.15em}\in\hspace{-0.15em} W^{1,\infty}_D(\Omega)$, \hspace{-0.1mm}\textit{i.e.}, \hspace{-0.1mm}the \hspace{-0.1mm}minimizer \hspace{-0.1mm}of \hspace{-0.1mm}\eqref{eq:elasto_primal}, \hspace{-0.1mm}and \hspace{-0.1mm}an \hspace{-0.1mm}absolutely~\hspace{-0.1mm}\mbox{continuous} dual solution
$\mu=z\otimes \mathrm{d}x\in (\mathrm{ba}(\Omega))^d$ with $z\in W^1_N(\operatorname{div};\Omega)$,~\textit{i.e.},~a~\mbox{maximizer}~of~\eqref{eq:elasto_dual2},~are~given~via
\begin{align}
\begin{aligned}
u
&\coloneqq
\begin{cases}
    \smash{\tfrac{C_f}{2d}}(r^2-\vert\cdot\vert^2),
    &\text{if } \smash{C_f}\le \tfrac{d}{r}\,,\\
    \left.\begin{cases}
        r-\vert\cdot\vert
        &\text{if } \smash{\tfrac{d}{C_f}}\le \vert\cdot\vert\le r\,,\\
        -\smash{\tfrac{C_f}{2d}}\vert\cdot\vert^2+r-\smash{\tfrac{d}{2C_f}}
        &\text{if } 0\le \vert\cdot\vert\le \smash{\tfrac{d}{C_f}}
    \end{cases}\right\}
    &\text{if } \smash{C_f}\ge \tfrac{d}{r}
\end{cases}&&\quad\text{ in }\Omega\,,\\
z
&\coloneqq
-\smash{\tfrac{C_f}{d}}\operatorname{id}_{\mathbb{R}^d}&&\quad\text{ in }\Omega\,.
\end{aligned}
\end{align}
\end{example}

More precisely, in the following, we consider the case $d=2$, $r=1$, $\smash{C_f}\in \{5,10\}$,~the same sequence of triangulations
$\{\mathcal{T}_{h_i}\}_{i=0,\ldots,9}$ of
$\overline{\Omega}=K_1^d(0)$ (\textit{cf}.\ Figure~\ref{fig:2D_triang}) constructed in  Subsection~\ref{subsec:experiments_pdirichlet}, 
$\gamma_1\hspace{-0.1em}=\hspace{-0.1em}\frac{1}{10}$, $\gamma_2\hspace{-0.1em}=\hspace{-0.1em}1$, and $\varepsilon\hspace{-0.1em}=\hspace{-0.1em}\smash{C_f} h^2$ (which preserves the quasi-optimal error~decay~rates~for~element-wise affine approximations of the elasto-plastic torsion~problem~\mbox{under}~\mbox{uniform}~\mbox{mesh-refinement},~\textit{cf}.~\cite{AntilBartelsKaltenbachKhandelwal2025}).\pagebreak

For the above experimental setup, Figures~\ref{fig:ept1} and~\ref{fig:ept2} compare the $\operatorname{prox}$-based~semi-smooth~New\-ton
method (\textit{cf}.\ Algorithm~\ref{alg:semi-smooth_newton} employing the representations \eqref{eq:elasto_proximity} and  \eqref{eq:elasto_proximity_derivatives}) with the
dual semi-smooth Newton method (\textit{cf}.\ Algorithm~\ref{alg:dual_Newton_elasto}) under mesh-refinement. To
enable~a~\mbox{direct}~\mbox{comparison},\linebreak we again report the common $\operatorname{prox}$-based nonlinear residual $\|\mathtt{F}_h(z_h^k,u_h^k)\|_{(L^2(\Omega))^d\times L^2(\Omega)}$, $k\in\mathbb{N}_0$, 
for both methods.
Again, the convergent runs exhibit fast local residual decay for both methods. Therefore,
the main difference is not the local decay rate, but the robustness of the iteration:\linebreak for both
values of $C_f\in\{5,10\}$ and both globalization strategies
(\textit{cf}.\ Strategies~\hyperlink{Strategy 1}{1}~and~\hyperlink{Strategy 2}{2}),~the~$\operatorname{prox}$-based semi-smooth Newton method remains effective for a larger
range~of~refined~triangulations.

\begin{figure}[H]
    \centering
    \includegraphics[width=\linewidth]{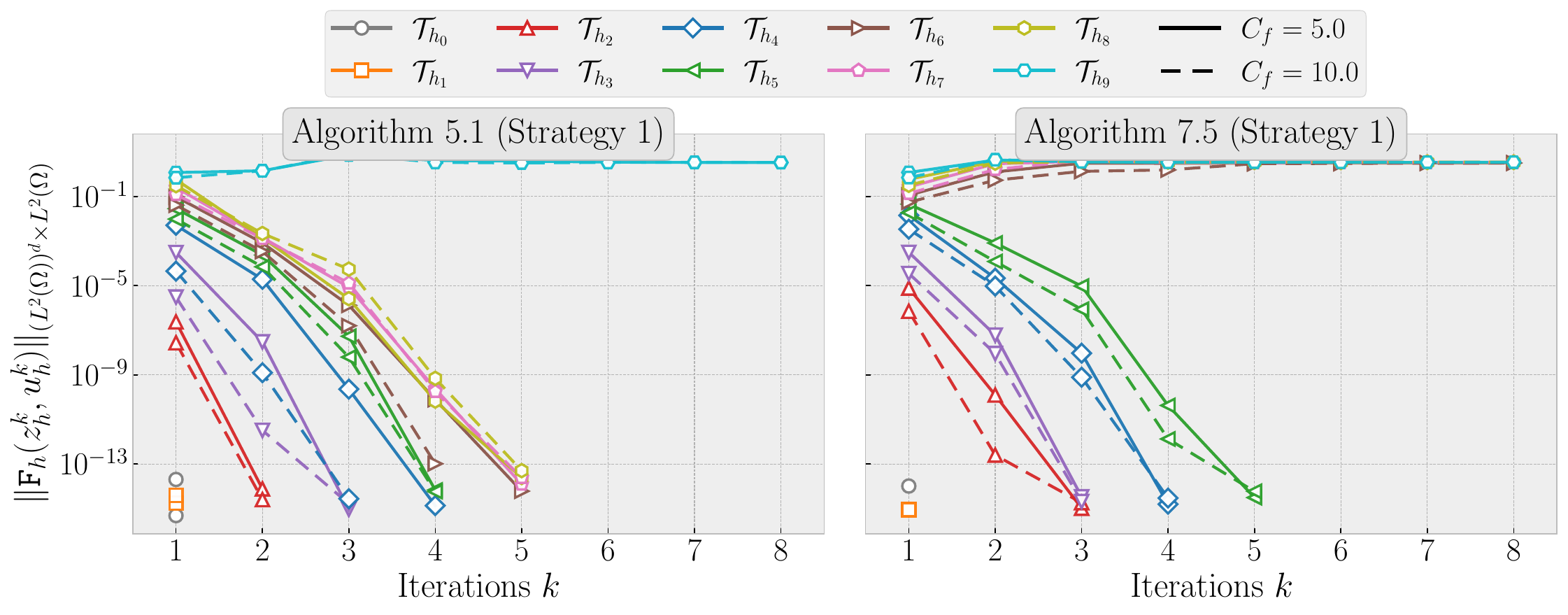}\vspace{-1.5mm}
   \caption{Residual decay for the elasto-plastic torsion problem (\textit{cf}.\ Subsection \ref{subsec:elasto-plastic})
using Strategy~\protect\hyperlink{Strategy 1}{1}, \textit{i.e.}, the $L^2$-gradient-flow initialization described in Remark \ref{rem:gradient_flow_initializations};
each measured by the $\operatorname{prox}$-based nonlinear residual
$\|\mathtt{F}_h(z_h^k,u_h^k)\|_{(L^2(\Omega))^d\times L^2(\Omega)}$, $k\in\mathbb{N}_0$:
\textit{left:} $\operatorname{prox}$-based semi-smooth Newton method (\textit{cf}.\ Algorithm \ref{alg:semi-smooth_newton} employing the representations \eqref{eq:elasto_proximity} and  \eqref{eq:elasto_proximity_derivatives}); \textit{right:} dual semi-smooth Newton method
(\textit{cf}.\ Algorithm \ref{alg:dual_Newton_elasto}). The $\operatorname{prox}$-based semi-smooth Newton method is less sensitive to mesh-refinement.}
    \label{fig:ept1}
\end{figure}

\begin{figure}[H]
    \centering
    \includegraphics[width=\linewidth]{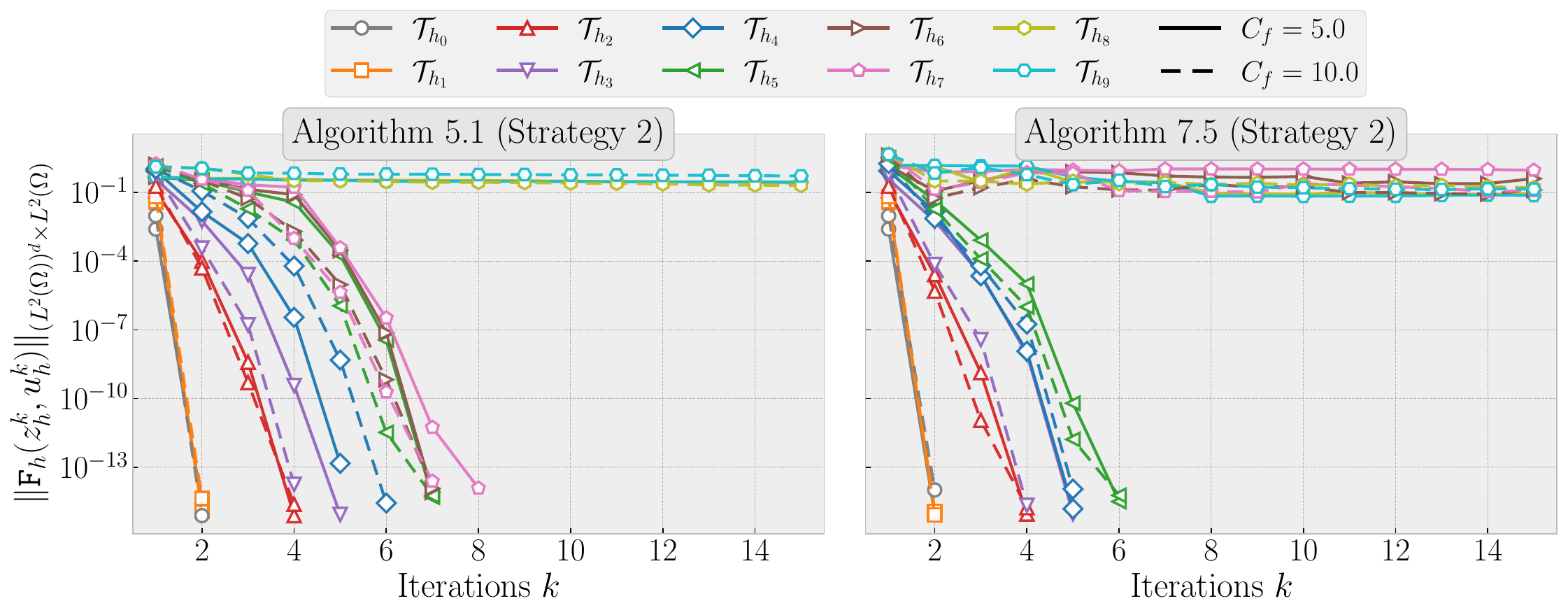}\vspace{-1.5mm}
    \caption{Residual decay for the elasto-plastic torsion problem (\textit{cf}.\ Subsection \ref{subsec:elasto-plastic})
using Strategy~\protect\hyperlink{Strategy 2}{2}, \textit{i.e.}, the residual-based Armijo--Goldstein line search; each measured by
the common $\operatorname{prox}$-based nonlinear residual
$\|\mathtt{F}_h(z_h^k,u_h^k)\|_{(L^2(\Omega))^d\times L^2(\Omega)}$, $k\in\mathbb{N}_0$:
\textit{left:} $\operatorname{prox}$-based semi-smooth Newton method (\textit{cf}.\ Algorithm \ref{alg:semi-smooth_newton} employing the representations \eqref{eq:elasto_proximity} and  \eqref{eq:elasto_proximity_derivatives}); \textit{right:} dual semi-smooth Newton method
(\textit{cf}.\ Algorithm \ref{alg:dual_Newton_elasto}). The $\operatorname{prox}$-based semi-smooth Newton method is less sensitive to mesh-refinement.}
    \label{fig:ept2}
\end{figure}

	{\setlength{\bibsep}{0pt plus 0.0ex}\small 
		
		\bibliographystyle{aomplain}
		\bibliography{references} 
		
	}
	
\end{document}